\documentclass[a4paper,fleqn]{cas-sc}
\usepackage[numbers,sort&compress]{natbib}

\def\tsc#1{\csdef{#1}{\textsc{\lowercase{#1}}\xspace}}
\tsc{WGM}
\tsc{QE}
\tsc{EP}
\tsc{PMS}
\tsc{BEC}
\tsc{DE}
\usepackage{algorithm}     
\usepackage{algpseudocode}  
\usepackage{setspace}
\let\Algorithm\algorithm
\renewcommand\algorithm[1][]{\Algorithm[#1]\setstretch{1.5}}
\usepackage{float}  
\usepackage{flafter}
\usepackage{placeins}
\usepackage{mathtools}
\usepackage{booktabs}
\usepackage{tcolorbox}
\usepackage{amsthm}
\usepackage{graphicx}

\newcommand{\bunderline}[1]{\underline{#1}}
\renewcommand{\vec}[1]{{\bunderline{#1}}}

\newcommand{\mat}[1]{{\bunderline{\bunderline{#1}}}}

\def\reals{\mathbb{R}}
\def\Tm{{\mathcal T}}
\def\mom{{\mathsf M}}
\def\morder{{M_\text{\normalfont o}}}
\def\mdeg{{M_\text{\normalfont deg}}}
\def\mbasis{{M_\text{\normalfont basis}}}
\def\mpos{{M_\text{\normalfont pos}}}

\makeatletter
\renewcommand*\env@matrix[1][\arraystretch]{%
  \edef\arraystretch{#1}%
  \hskip -\arraycolsep
  \let\@ifnextchar\new@ifnextchar
  \array{*\c@MaxMatrixCols c}}
\makeatother

\newtheorem{theorem}{Theorem}
\newtheorem{lemma}[theorem]{Lemma}
\newdefinition{rmk}{Remark}

\begin{document}
\shorttitle{}

\shortauthors{D. Gambo and J.A. Rossmanith}

\title [mode = title]{Positivity-preserving semi-Lagrangian discontinuous Galerkin methods for multi-species Vlasov-Amp\`ere models of plasma}                      

\author{Dauda Gambo}[type=editor,orcid=0000-0001-8364-2065]

\ead{dgambo@iastate.edu}

\affiliation{organization={Department of Mathematics, Iowa State University},
    addressline={411 Morrill Road}, 
    city={Ames},
    postcode={50011}, 
    state={IA},
    country={USA}}

\author{James A. Rossmanith}[type=editor,orcid=0000-0002-3629-8895]

\cormark[1]

\ead{rossmani@iastate.edu}
\cortext[cor1]{Corresponding author}
\begin{abstract}
We develop a high-order and positivity-preserving kinetic solver for the multi-species Vlasov-Amp\`ere system that is built within a semi-Lagrangian discontinuous Galerkin (SLDG) framework. The method is designed for efficiency and accuracy across a broad range of conditions, including stiff, multi-scale regimes with strong ion-electron scale separation. Phase space is discretized via high-order discontinuous Galerkin finite elements; time stepping is handled via high-order operator splitting and a series of one-dimensional, unconditionally stable, semi-Lagrangian updates. To improve both accuracy and efficiency, each plasma species is represented on a different phase-space mesh. On every element and for each species, local positivity-preserving limiters are applied to prevent spurious undershoots and rigorously ensure positivity of the cell averages at the next time step. The accuracy and robustness of the resulting method are verified on several standard single- and two-species test cases, including manufactured solutions, two-stream instability, weak and strong Landau damping, ion-acoustic wave propagation, ion-acoustic shock formation, and KEEN wave dynamics.
\end{abstract}


\maketitle

\tableofcontents

\section{Introduction}
One important classifier of plasma regimes is the non-dimensional plasma parameter: 
\begin{equation}
\Lambda \propto \sqrt{\frac{T^3}{n}},
\end{equation}
where $T$ and $n$ are characteristic temperature and number density values, respectively. For large $\Lambda$, the plasma is classified as weakly-coupled; examples of such applications include the solar wind ($\Lambda \sim 5 \times 10^{10}$), tokamak fusion devices ($\Lambda \sim 4 \times 10^8$), and the interstellar medium ($\Lambda \sim 4 \times 10^6$) (e.g., see Fitzpatrick \cite{link:Fitzpatrick2023}). Many plasmas in the weakly coupled regime require a fully kinetic description because they are far from thermodynamic equilibrium and exhibit low collisionality. In these cases, the plasma can often be accurately modeled by kinetic Vlasov equations (one equation per species), where particles interact only through electromagnetic fields \cite{article:Vlasov1968}. These equations are nonlinear, nonlocal, and posed in high-dimensional phase space. Analytical solutions are generally unavailable, so it becomes essential to develop high-fidelity numerical algorithms to resolve the complex multiscale dynamics. This has driven sustained interest in the plasma science community to develop numerical methods that are accurate, efficient, and capable of preserving important physical properties such as the positivity of the particle density functions and the conservation of mass, momentum, and energy.

\subsection{Vlasov models}
Kinetic Vlasov models can be solved numerically using either particle-based or grid-based methods. Among particle-based schemes, the Particle-in-Cell (PIC) method is the dominant choice in the plasma physics community (e.g., see Ren and Lapenta \cite{article:RenLapenta2024} for a recent review). PIC represents the distribution with a finite number of macro-particles pushed along characteristics, while the electromagnetic field is computed on a fixed background mesh. To couple the particles to the fields, charges and currents are deposited on the grid via particle-to-mesh interpolation formulas; once the fields have been updated, they are gathered back to particles via mesh-to-particle interpolation formulas (e.g., see Birdsall and Langdon \cite{BirdsallLangdon2004}). PIC is efficient and robust under severe phase-space distortion and filamentation, but it can suffer from statistical noise and interpolation errors when particle-per-mesh cell counts become too large or too small.

An alternative to PIC is to employ a fully deterministic, grid-based approach to evolve the Vlasov equation on a fixed phase-space mesh. This eliminates particle noise and can be designed for high-order accuracy, enabling the resolution of fine phase-space structure, at the expense of higher memory cost in high dimensions. Grid-based methods have been studied extensively using a broad spectrum of solvers spanning finite difference (e.g., \cite{BanksOduBergerChapmanArrighiBrunner2019,WhealtonMcGaffeyMeszaros1986}), finite element (e.g., \cite{HeathGambaMorrisonMichler2012, RossmanithSeal2011, QiuShu2011,article:RossmanithVaughan2026}), finite volume (e.g., \cite{BanksHittinger2010, Despres2008, FilbetSonnendruckerBertrand2001, FilbetSonnendrucker2003, XiongCohenRognlienXu2008}), and spectral (e.g., \cite{Giraldo1998,LeBourdiecDeVuystJacquet2006}) methods in both Eulerian and semi-Lagrangian settings.

\subsection{Semi-Lagrangian mesh-based methods}
Semi-Lagrangian time discretizations advance the Vlasov equation by tracing characteristics either forward or backward in time, depending on the specific method, on a fixed background grid, allowing large time-steps while maintaining high-order accuracy and robustness \cite{BesseMehrenberger2008}. Within this class, semi-Lagrangian schemes can be broadly categorized into split and unsplit schemes. Split schemes, originally based on Strang operator splitting \cite{Strang1968} and its adoption by Cheng and Knorr \cite{ChengKnorr1976}, decompose the Vlasov equation into free-streaming and acceleration subsystems, each advanced efficiently by semi-Lagrangian dynamics \cite{ChengKnorr1976}. Although computationally efficient and straightforward to implement in discontinuous Galerkin (DG) frameworks \cite{QiuShu2011,RossmanithSeal2011}, operator splitting introduces splitting error, which may accumulate for nonlinear kinetic models over a large number of time-steps \cite{HuotGhizzoBertrandSonnendruckerCoulaud2003}, resulting in significant numerical drift in quantities that are invariant at the continuous level \cite{CrouseillesLiuYue2025}. 

By contrast, unsplit semi-Lagrangian schemes evolve the full multi-dimensional Vlasov system without dimensional decomposition, preserving phase-space coupling and typically improving long-time accuracy; however, such methods are difficult to extend efficiently to high orders. Several unsplit semi-Lagrangian strategies have been proposed to enhance spatial and temporal consistency \cite{CaiGuoQiu2018,CaiGuoQiu2019,CrouseillesGlancHirstoagaMadauleMehrenbergerPetri2014}. Qiu and collaborators have advanced several high-order unsplit semi-Lagrangian schemes, including a conservative semi-Lagrangian DG scheme for the Vlasov-Poisson equation \cite{CaiGuoQiu2018}, an extension via Runge-Kutta exponential integrators \cite{CaiBoscarinoQiu2021,ZhengHayesChristliebQiu2025}, and a fourth-order conservative semi-Lagrangian finite volume WENO method \cite{ZhengCaiQiuQiu2022}. In parallel, conservative semi-Lagrangian methodologies based on volume remapping and positivity-preserving limiting have been developed to strengthen mass conservation and suppress nonphysical oscillations \cite{CrouseillesMehrenbergerSonnendrucker2010, ZhengCaiQiuQiu2022}. Building on this line, the semi-Lagrangian adaptive-rank (SLAR) approach exploits low-rank tensor compression to evolve the Vlasov-Poisson equation in full phase space efficiently \cite{ZhengHayesChristliebQiu2025}, and related low-rank unsplit integrators further support this direction \cite{EinkemmerLubich2018,EinkemmerLubich2019,EinkemmerOstermannPiazzola2020,EinkemmerJoseph2021}. 

Despite their improved coupling properties, unsplit semi-Lagrangian methods typically require genuinely multi-dimensional characteristic tracing and reconstruction, which increases algorithmic complexity, memory footprint, and data movement, often rendering the overall computational cost prohibitive in Vlasov-Amp\`ere or Vlasov-Maxwell settings, where both phase space dimensionality and field coupling must be handled efficiently. Consequently, for large-scale computation, it is often preferable to retain a split semi-Lagrangian discretization, where the update is reduced to a sequence of lower-dimensional substeps that are simpler to implement and parallelize, and then to control splitting-induced errors through conservative remapping, limiting, or projection mechanisms that enforce the desired invariants without sacrificing scalability \cite{CrouseillesMehrenbergerSonnendrucker2010, qiu2010, ZhengCaiQiuQiu2022}.

\subsection{Multi-species models}
Multi-species Vlasov equations introduce additional stiffness beyond the single-species case. Ion-to-electron mass and/or temperature disparities can produce widely separated temporal and velocity scales, forcing explicit methods to resolve the faster electron scales even when the target physics evolves on ion time scales. The difficulty is compounded by the fact that multi-species conservation is intrinsically global; the total energy involves the sum of kinetic energies across species plus the field energy, so any cross-species inconsistency in the field update can manifest as artificial heating or cooling. 

Multi-species Vlasov-Maxwell solvers built on Hamiltonian structure-preserving splitting report robust long-time behavior with small energy and entropy errors \cite{NufiEtAlMultispeciesVM}. An energy-conserving DG method for the two-species Vlasov-Amp\`ere system that preserves species-wise particle number and the fully discrete total energy was proposed by Cheng, Christlieb, and Zhong \cite{ChengChristliebZhong2015_TwoSpeciesVA}. Despite these efforts, few multi-species deterministic approaches that are simultaneously high-order, scalable, and positivity-preserving exist in the literature, which motivates further development. Recent reviews and progress on Vlasov solvers and multi-species Vlasov modeling and structure-preserving algorithms can be found in \cite{PalmrothGansePfauKempfEtAl2025, WilhelmTorrilhon2025NFIInstabilitiesRGD, WilhelmEtAl2025HighFidelityMultispeciesVlasov}.

\subsection{Scope of this work}
In this work, we propose an extension of the semi-Lagrangian discontinuous Galerkin (SLDG) approach developed by Rossmanith and Seal \cite{RossmanithSeal2011} and Seal \cite{thesis:Seal2012} for solving the single and multi-species Vlasov-Amp\`ere system. The design principles of this proposed approach can be summarized as follows:
\begin{itemize}
\item High-order operator splitting is used to split the spatial and velocity transport terms in the Vlasov equation;
\item The time, space, and velocity discretizations are all high-order accurate; 
\item Semi-Lagrangian time-stepping is used to achieve an unconditionally stable method, allowing large time steps;
\item The electric field is updated by exactly solving a harmonic oscillator subsystem in the operator split framework; 
\item Positivity-preserving limiters are implemented that yield provable positivity of all particle distribution functions;
\item The scheme is automatically locally (and therefore also globally) mass/charge conservative;
\item In the multi-species case, the particle distribution function for each species is defined over a different velocity range to improve computational efficiency and storage.
\end{itemize}
The resulting numerical method is implemented in the DoGPack \cite{dogpack} software package and applied to several single and multi-species test cases from the existing literature.

The remainder of this paper is organized as follows. Section~\ref{sec:model} introduces the multi-species Vlasov-Amp\`ere system, the non-dimensionalization, and a derivation of the mass, momentum, and energy conservation laws. Section~\ref{sec:numerics} presents the proposed semi-Lagrangian DG framework, including the operator splitting time integrators, the electric field update, and the positivity-preserving limiters. Numerical validation studies are reported in Section~\ref{sec:results}. Summary and concluding remarks are given in Section~\ref{sec:conclusion}. %

\section{Vlasov-Amp\`ere system}\label{sec:model}
In this section we use the following convention: (1) all dimensional variables are decorated with a tilde, and (2) non-dimensionalized variables are not. We consider a plasma made up of ${\text{S}}$ ion species, where each species satisfies a collisionless Vlasov equation of the form:
\begin{equation}
\label{Vlasov_eqn}
\text{Vlasov:} \quad \frac{\partial \tilde{f}_s}{\partial \tilde{t}} + \vec{\tilde{v}}\cdot\vec{\nabla_{\tilde{x}}} \tilde{f}_s
  + \frac{\tilde{q}_s}{\tilde{m}_s}\,\vec{\tilde{E}}\cdot \vec{\nabla_{\tilde{v}}} \tilde{f}_s
  = 0,
\end{equation}
where $\tilde{t} \in \reals_{\ge 0}$ is the time coordinate, $\vec{\tilde{x}} \in \reals^D$ is the spatial coordinate, $\vec{\tilde{v}} \in \reals^V$ is the velocity coordinate, $s\!=\!1,2,\ldots,\text{S}$ is the species index, $\tilde{q}_s$ and $\tilde{m}_s$ are the particle charge and mass, respectively,  $\tilde{f}_s\left(\tilde{t}, \vec{\tilde{x}}, \vec{\tilde{v}}\right)\!:\!\reals_{\ge 0} \times \reals^D \times \reals^V \mapsto \reals_{\ge 0}$ is the particle density function (PDF), and $D$ and $V$ are the number of spatial and velocity dimensions, respectively. In this completely collisionless setting, the particle species interact with each other only through the electrostatic force, which depends on the electric field, $\vec{\tilde{E}}\left(\tilde{t},\vec{\tilde{x}}\right)\!:\!\reals_{\ge 0} \times \reals^D \mapsto \reals^D$, and satisfies both the Gauss law and the Amp\`ere equation:
\begin{alignat}{2}
\label{gausslaw_eqn}
  \text{Gauss law:} \quad & \vec{\nabla_{\tilde{x}}} \cdot \vec{\tilde{E}} = 
  - \vec{\nabla_{\tilde{x}}} \cdot \left( \vec{\nabla_{\tilde{x}}} \, \phi \right) = \frac{\tilde{\sigma}}{\tilde\varepsilon_0}, \qquad &&
  \text{Total charge:} \quad \tilde{\sigma} := \sum_{s=1}^{\text{S}} \tilde{q}_s \int_{\reals^V} \tilde{f}_s \, d\vec{\tilde{v}},  \\
\label{ampere_eqn}
  \text{Amp\`ere:} \quad &\tilde\varepsilon_0 \frac{\partial \vec{\tilde{E}}}{\partial \tilde{t}} + \vec{\tilde{J}} = \vec{0},
  \qquad &&
  \text{Total current:} \quad   \vec{\tilde{J}} := \sum_{s=1}^{\text{S}} \tilde{q}_s \int_{\reals^V} \vec{\tilde{v}} \, \tilde{f}_s \, d\vec{\tilde{v}},
\end{alignat}
where $\varepsilon_0$ is the permittivity of free space.
The Gauss and Amp\`ere equations are connected through {\it charge conservation} (multiply \eqref{Vlasov_eqn} by $\tilde{q}_s$, sum over $s$, integrate over $\vec{v} \in \reals^V$, and assume $\tilde{f}_s \rightarrow 0$ sufficiently fast as $\| \vec{\tilde{v}} \| \rightarrow \infty$):
\begin{equation}
\frac{\partial}{\partial \tilde{t}} \left( \sum_{s=1}^{\text{S}} \tilde{q}_s \int_{\reals^V} \tilde{f}_s \, d\vec{\tilde{v}} \right) + \vec{\nabla_{\tilde{x}}} \cdot
\left( \sum_{s=1}^{\text{S}} \tilde{q}_s \int_{\reals^V} \vec{\tilde{v}} \, \tilde{f}_s \, d\vec{\tilde{v}} \right) = 0 \quad \Longrightarrow \quad
\frac{\partial {\tilde{\sigma}}}{\partial \tilde{t}} + \vec{\nabla_{\tilde{x}}} \cdot \vec{\tilde{J}} = 0.
\end{equation}

\subsection{Non-dimensionalization}
We define the following non-dimensional variables:
\begin{equation}\label{eq:dimensionless_vars}
  {t}= \left(\frac{\tilde{v}_0}{\tilde{x}_0}\right) \tilde{t},\quad
  {\vec{x}}=\frac{\vec{\tilde{x}}}{\tilde{x}_0},\quad
  {\vec{v}}=\frac{\vec{\tilde{v}}}{\tilde{v}_0},\quad
  f_s=\left(  \frac{\tilde{q}_0 \tilde{x}_0 \tilde{v}_0^{V}}{\tilde\varepsilon_0 \tilde{E}_0} \right) \tilde{f}_s,\quad
  {\vec{E}}=\frac{\vec{\tilde{E}}}{\tilde{E}_0}, \quad
  m_s = \frac{\tilde{m}_s}{\tilde{m}_0}, \quad
  q_s = \frac{\tilde{q}_s}{\tilde{q}_0},
\end{equation}
where $\tilde{x}_0$ (length), $\tilde{v}_0$ (velocity), $\tilde{m}_0$ (mass), $\tilde{q}_0$ (charge), $\tilde{E}_0$ (electric field), as well as $\tilde{n}_0$ (number density) and $\tilde{T}_0$ (temperature), which do not appear explicitly in \eqref{eq:dimensionless_vars}, are all dimensional scaling parameters that are globally constant in time, space, and velocity. In this work, we use electron parameters to non-dimensionalize by anchoring the following four values:
\begin{equation}
\begin{gathered}
\tilde{m}_0 = \tilde{m}_e \quad \left(\text{electron mass}\right), \quad
\tilde{q}_0 = \tilde{q}_e \quad \left(\text{elementary charge}\right), \\
\tilde{n}_0 = \tilde{n}_e \quad \left(\text{initial electron number density}\right), \quad
\tilde{T}_0 = \tilde{T}_e \quad \left(\text{initial electron temperature}\right),
\end{gathered}
\end{equation}
and arriving at the following derived electron parameters:
\begin{equation}
\quad \tilde{x}_0 = \sqrt{\frac{\tilde\varepsilon_0 \tilde{k}_{B} \tilde{T}_e}{\tilde{n}_e  \tilde{q}_e^2}} \quad \left(\text{Debye length}\right),
\quad \tilde{v}_0=\sqrt{\frac{\tilde{k}_B \tilde{T}_e}{\tilde{m}_e}} \quad \left(\text{thermal velocity}\right), \quad
  \tilde{E}_0=\frac{\tilde{m}_e \tilde{v}_0^2}{\tilde{q}_e \tilde{x}_0} = \frac{\tilde{k}_B \tilde{T}_e}{\tilde{q}_e \tilde{x}_0},
\end{equation}
where $\tilde{k}_B$ is the Boltzmann constant.

Plugging the non-dimensionalization of \eqref{eq:dimensionless_vars} into the dimensional Vlasov equation \eqref{Vlasov_eqn}, as well as the dimensional Gauss \eqref{gausslaw_eqn} and Amp\`ere \eqref{ampere_eqn} equations, and simplifying, yields the final non-dimensionalized form of the equations:
\begin{alignat}{2}
\label{eq:vlasov-multid}
\text{Vlasov:} \quad &\frac{\partial f_s}{\partial t} + \vec{v}\cdot\vec{\nabla_{x}} f_s
  + \frac{q_s}{m_s}\,\vec{E}\cdot \vec{\nabla_v} f_s
  = 0, \\
\label{eq:gauss-multid}
  \text{Gauss law:} \quad & \vec{\nabla_{x}} \cdot \vec{E} = 
  - \vec{\nabla_{x}} \cdot \left( \vec{\nabla_{x}} \, \phi \right) = \sigma, \qquad &&
  \text{Total charge:} \quad \sigma := \sum_{s=1}^{\text{S}} q_s \int_{\reals^V} f_s \, d\vec{v},  \\
\label{eq:ampere-multid}
  \text{Amp\`ere:} \quad &\frac{\partial \vec{{E}}}{\partial {t}} + \vec{{J}} = \vec{0},
  \qquad &&
  \text{Total current:} \quad   \vec{{J}} := \sum_{s=1}^{\text{S}} {q}_s \int_{\reals^V} \vec{{v}} \, {f}_s \, d\vec{{v}}, \\
  \text{Charge conservation:} \quad &\frac{\partial \sigma}{\partial t} + \vec{\nabla_x} \cdot \vec{J} = 0,
\end{alignat}
where the non-dimensional particle masses, $m_s$, and the non-dimensional particle charges, $q_s$, are the only remaining parameters.

\subsection{Conservation of mass, momentum, and energy}\label{subsec:invariants}
The Vlasov equation \eqref{eq:vlasov-multid} is a Hamiltonian system (e.g., see \cite{article:Morrison1980,article:Marsden1982}) and thus contains an infinite number of conserved quantities; however, numerical discretizations are generally unable to exactly conserve more than a few quantities. In many practical applications, the conservation of total mass, momentum, and energy is particularly important, especially for long-time simulations. 

Derivation of these three conservation laws begins with computing the first five moments of \eqref{eq:vlasov-multid}, multiplying by the particle mass $m_s$, summing across all species $s=1,\ldots,\text{S}$, and assuming $f_s \rightarrow 0$ sufficiently fast as $\| \vec{v} \| \rightarrow \infty$:
\begin{align}
\label{eqn:multid-mass-pre}
 \text{Mass:} &\quad \frac{\partial}{\partial t} \left( \sum_{s=1}^{\text{S}} m_s \, \rho_{s} \right) + \vec{\nabla_x} \cdot  \left( \sum_{s=1}^{\text{S}} m_s \, \rho_s \vec{u_s} \right) = 0, \\
\label{eqn:multid-mom-pre}
 \text{Momentum:} &\quad \frac{\partial}{\partial t} \left( \sum_{s=1}^{\text{S}} m_s \, \rho_s \vec{u_s} \right) + \vec{\nabla_x} \cdot \left( \sum_{s=1}^{\text{S}}  m_s \, \mat{{\mathbb E}_s} \right) - \vec{E} \left( \sum_{s=1}^{\text{S}} q_s \rho_s \right)= \vec{0}, \\
\label{eqn:multid-energy-pre}
 \text{Energy:} &\quad \frac{\partial}{\partial t} \left( \sum_{s=1}^{\text{S}} m_s \, {\mathcal E}_s \right) + \vec{\nabla_x} \cdot \left( \sum_{s=1}^{\text{S}} m_s \, \vec{{\mathcal F}_s} \right) - \vec{E} \cdot \left( \sum_{s=1}^{\text{S}} q_s \rho_s \vec{u_s} \right) = 0,
\end{align}
where
\begin{equation}
  \left\{ \rho_{s} \, , \, \, \rho_{s} \vec{u_s} \, , \, \, {\mathcal E}_s \, , \, \, \mat{{\mathbb E}_s} \, , \, \, \vec{{\mathcal F}_s} \right\} = 
  \int_{\reals^V} \left\{ 1 \, , \, \, \vec{v} \, , \, \, \frac{\|\vec{v}\|^2}{2} \, , \, \, \vec{v}\otimes\vec{v} \, , \, \, \frac{\vec{v} \| \vec{v} \|^2}{2} \right\} \, f_{s} \, d\vec{v}.
\end{equation}
We use Gauss' law \eqref{eq:gauss-multid} and the fact that the electric field is curl-free to rewrite the last term in \eqref{eqn:multid-mom-pre}
as follows:
\begin{align}
-\vec{E}  \sum_{s=1}^{\text{S}} q_s \rho_s  = -\vec{E} \sigma = - \vec{E} \, \vec{\nabla_x} \cdot  \vec{E} = 
- \vec{\nabla_x} \cdot  \mat{{\mathbb T}}, \quad
\mat{{\mathbb T}} = 
\begin{bmatrix}
 \frac{1}{2} \left(  E_{1}^2 -  E_{2}^2 - E_{3}^2\right)_{,x} + \left( E_1 E_2 \right)_{,y}  + \left( E_1 E_3 \right)_{,z}    \\
 \left( E_2 E_{1} \right)_{,x} + \frac{1}{2} \left(  E_{2}^2 - E_{1}^2 - E_{3}^2 \right)_{,y} + \left( E_2 E_{3} \right)_{,z}   \\
\left( E_3 E_1 \right)_{,x}  + \left( E_3 E_{2} \right)_{,y} 
 + \frac{1}{2} \left(  E_{3}^2 -  E_{1}^2 -  E_{2}^2  \right)_{,z}
\end{bmatrix}.
\end{align}
We use the Amp\`ere equation \eqref{eq:ampere-multid} to rewrite the last term in \eqref{eqn:multid-energy-pre}
as follows:
\begin{align}
-\vec{E} \cdot  \sum_{s=1}^{\text{S}} q_s \rho_s \vec{u_s}  = -\vec{E} \cdot \vec{J} = 
\vec{E} \cdot \frac{\partial \vec{E}}{\partial t} = 
\frac{\partial}{\partial t} \left( \frac{\|\vec{E}\|^2}{2} \right).
\end{align}
This results in the following conservation of total mass, momentum, and energy equations:
\begin{align}
\label{eqn:multid-mass-post}
 \text{Mass:} &\quad \frac{\partial}{\partial t} \left( \sum_{s=1}^{\text{S}} m_s \, \rho_{s} \right) + \vec{\nabla_x} \cdot  \left( \sum_{s=1}^{\text{S}} m_s \, \rho_s \vec{u_s} \right) = 0, \\
\label{eqn:multid-mom-post}
 \text{Momentum:} &\quad \frac{\partial}{\partial t} \left( \sum_{s=1}^{\text{S}} m_s \, \rho_s \vec{u_s} \right) + \vec{\nabla_x} \cdot \left( \sum_{s=1}^{\text{S}}  m_s \, \mat{{\mathbb E}_s} - \mat{\mathbb{T}} \right) = \vec{0}, \\
\label{eqn:multid-energy-post}
 \text{Energy:} &\quad \frac{\partial}{\partial t} \left( \frac{1}{2} \| \vec{E}\|^2 + \sum_{s=1}^{\text{S}} m_s \, {\mathcal E}_s \right) + \vec{\nabla_x} \cdot \left( \sum_{s=1}^{\text{S}} m_s \, \vec{{\mathcal F}_s} \right) = 0.
\end{align}
Under appropriate boundary conditions (e.g., zero flux or periodic), we arrive at the following conservation statements over some domain $\Omega \subset \reals^D$:
\begin{align}
 \frac{d}{dt} \int_{\Omega} \left\{ \sum_{s=1}^{\text{S}} m_s \rho_{s} \, , \, \,
 \sum_{s=1}^{\text{S}} m_s \rho_s \vec{u_s} \, , \, \, \frac{1}{2} \|\vec{E}\|^2 + \sum_{s=1}^{\text{S}} m_s {\mathcal E}_{s} \right\} \, d\vec{x} = \Biggl\{ 0 \, , \, \, \vec{0} \, , \, \, 0 \Biggr\}.
\end{align}

\subsection{1D1V single-species Vlasov-Amp\`ere system}
The single-species 1D1V Vlasov-Amp\`ere model (i.e., $S\!=\!1$, $D\!=\!1$, and $V\!=\!1$) 
is obtained by assuming a stationary, spatially uniform
ion background and dynamically evolving electrons. Thus, only the electron distribution
$f(t,x,v)$ is advanced. In non-dimensional variables, we arrive at the following system:
\begin{align}
\label{eq:va_single_vlasov}
  \text{Vlasov:} & \quad f_{,t} + v f_{,x} - E f_{,v} = 0, \\
\label{eq:va_single_ampere}
  \text{Amp\`ere:} & \quad E_{,t} + J = 0,
  \quad
  J = J_0 - \rho u, \quad
  \rho u = \int_{-\infty}^{\infty} v f \,dv,
\end{align}
on the domain $x \in \left(x_{\text{min}}, x_{\text{max}} \right)$ and $v \in \left(-\infty, \infty\right)$,
where $J_0$ is the background ion current. The initial electric field
can be computed from the Gauss law:
\begin{equation}\label{eq:va_single_poisson}
   \text{Gauss law:} \quad E_{,x} = \rho_0 - \rho, \quad \rho = \int_{-\infty}^{\infty} f \, dv,
\end{equation}
where $\rho_0$ is the constant background ion charge density. As an example, for problems with periodic boundary conditions, $\rho_0$ and $J_0$ are taken as the average initial density and momentum so that the electric field is periodic on the domain $x\in \left(x_{\text{min}}, x_{\text{max}} \right)$ (see \eqref{eq:va_single_poisson}) and the net electric field across this domain remains zero (see \eqref{eq:va_single_ampere}):
\begin{equation}
\rho_0 = {\left(x_{\text{max}}-x_{\text{min}}\right)^{-1}} \int_{x_{\text{min}}}^{x_{\text{max}}} \rho(t=0,x) \, dx, \quad
J_0 = {\left(x_{\text{max}}-x_{\text{min}}\right)^{-1}} \int_{x_{\text{min}}}^{x_{\text{max}}} \rho u(t=0,x) \, dx.
\end{equation}
With boundary conditions that result in net zero flux through the boundaries, we arrive at the following conservation of mass, momentum, and total energy:
\begin{align}
\label{eqn:conservation_1d1v_single}
 \frac{d}{dt} \int_{x_{\text{min}}}^{x_{\text{max}}} \left\{ \rho, \, 
 \rho u, \, \frac{1}{2} E^2 + {\mathcal E} \right\} \, dx = \Bigl\{ 0 \, , \, \, 0 \, , \, \, 0 \Bigr\}.
\end{align}

\subsection{1D1V two-species Vlasov-Amp\`ere system}
\label{subsec:two_species_va}
The dimensionless two-species 1D1V Vlasov-Amp\`ere system (i.e., $S\!=\!2$, $D\!=\!1$, and $V\!=\!1$) models an electrostatic collisionless plasma in which ions and electrons are coupled through the self-consistent electric field.
Let $f_s\!=\!f_s(t,x,v)$ denote the phase space distribution of species
$s\in\{e, i\}$ with charge $q_s$ and mass $m_s$. The governing equations are
\begin{align}
\label{eq:va_two_species}
  \text{Vlasov:} & \quad  f_{s,t} + v \, f_{s,x} + \frac{q_s}{m_s} E f_{s,v} = 0, \quad s \in \left\{ e, i \right\}, \\
  \label{eq:va_two_species_E}
  \text{Amp\`ere:} & \quad E_{,t} + J = 0, \quad 
  J = \sum_{s\in \left\{ e, i \right\}} q_s \rho_s u_s, \quad
  \rho_s u_s = \int_{-\infty}^{\infty} v\,f_s\,dv,
\end{align}
on the domain $x \in \left(x_{\text{min}}, x_{\text{max}} \right)$ and $v \in \left(-\infty, \infty\right)$.
The initial electric field can be computed from the Gauss law:
\begin{equation}\label{eq:va_two_species_poisson}
   \text{Gauss law:} \quad E_{,x} = \sigma, \quad \sigma = 
   \sum_{s\in \left\{ e, i \right\}} q_s \rho_s, \quad
   \rho_s = \int_{-\infty}^{\infty} f_s \, dv.
\end{equation}
With boundary conditions that result in net zero flux through the boundaries, we arrive at the following conservation of mass, momentum, and total energy:
\begin{align}
\label{eqn:conservation_1d1v_multi}
 \frac{d}{dt} \int_{x_{\text{min}}}^{x_{\text{max}}} \left\{ \sum_{s\in \left\{ e, i \right\}} m_s \rho_{s}, \, 
 \sum_{s \in \left\{ e, i \right\}} m_s \rho_s u_s, \, \frac{1}{2} E^2 + \sum_{s\in \left\{ e, i \right\}} m_s {\mathcal E}_{s} \right\} \, dx = 
 \Bigl\{ 0 \, , \, \, 0 \, , \, \, 0 \Bigr\}.
\end{align}

\section{Numerical methods}
\label{sec:numerics}
This section describes the proposed positivity-preserving semi-Lagrangian discontinuous Galerkin method for the multi-species 1D1V Vlasov-Amp\`ere system and summarizes its key properties. We make use of high-order operator splitting to convert the Vlasov-Amp\`ere system into two parts: (1) a free-streaming subproblem in which the electric field is assumed time-independent, and (2) an acceleration subproblem in which the Amp\`ere equation is used to update the electric field. On the acceleration branch, we present a harmonic-oscillator field-current update. Positivity limiters are introduced to produce positive element averages and to prevent negative values of the particle distribution function at key quadrature points.

\subsection{Operator splitting}\label{subsec:splitting}
We adopt an operator splitting strategy similar to Cheng and Knorr \cite{ChengKnorr1976} that reduces the Vlasov-Amp\`ere system into subproblems that can be handled efficiently via semi-Lagrangian DG updates and enables direct control of the discrete particle-field energy exchange. 
We split the evolution into a free-streaming and an acceleration subproblem. The free-streaming subproblem advances transport in physical space while the electric field remains unchanged,
\begin{equation}\label{eq:free-streaming}
\textnormal{Problem $\mathcal{A}$:} \quad
f_{s,t} + v f_{s,x} = 0 \quad  \forall s, \quad E_{,t} = 0,
\end{equation}
while the acceleration subproblem advances transport in velocity space, together with the electric field update,
\begin{equation}\label{eq:acceleration}
\textnormal{Problem $\mathcal{B}$:} \quad
f_{s,t} + \frac{q_s}{m_s} E f_{s,v} = 0 \quad \forall s, \quad 
E_{,t} + J = 0.
\end{equation}

In the free-streaming subproblem $\mathcal{A}$ \eqref{eq:free-streaming}, $f_s$ is transported in $x$ at each fixed $v$, in the acceleration subproblem $\mathcal{B}$ \eqref{eq:acceleration}, $f_s$ is transported in $v$ at each fixed $x$, although $E$ is time-dependent and is updated through Amp\`ere's equation. This separation into unidirectional advection equations is the key structure exploited by the proposed semi-Lagrangian DG discretization.

\subsubsection{Problem $\mathcal{A}$: Constant-coefficient advection}
\label{subsection:problemA}
Of the two steps in the operator splitting, Problem $\mathcal{A}$ \eqref{eq:free-streaming} is the simpler one and can be solved exactly via the method of characteristics. We look for curves in the $(t,x)$-plane parameterized by $c \in \reals$ along which $f_s$ is constant:
\begin{equation}
\frac{d}{dc} f_{s}\left(t(c), x(c), v\right) = \frac{dt}{dc} \frac{\partial f_s}{\partial t} 
+ \frac{dx}{dc} \frac{\partial f_s}{\partial x}  =
f_{s,t} + v f_{s,x} = 0 \quad \Longrightarrow \quad
\frac{dt}{dc} = 1, \quad \frac{dx}{dc} = v, \quad \frac{df_s}{dc} = 0.
\end{equation}
The curves along which $f_s$ is constant are straight lines in the $(t,x)$-plane given by $x - vt = \text{constant}$.
From this, we arrive at the following exact solution of Problem $\mathcal{A}$ \eqref{eq:free-streaming} at some time $t=t^n + \Delta t$, written here with initial condition $f_s\left(t^n,x,v\right)$:
\begin{equation}
\label{eqn:char-method-A}
	f_s\left( t^n + \Delta t, x, v \right) = f_s \bigl( t^n, x - v \Delta t, v \bigr).
\end{equation}
Furthermore, we note that the electric field remains unchanged in Problem $\mathcal{A}$ \eqref{eq:free-streaming}:
\begin{equation}
E\left( t^n + \Delta t, x \right) = E \bigl( t^n, x \bigr).
\end{equation}

\begin{lemma}\label{thm:problemA_invariant}
With boundary conditions at $x\!=\!x_{\text{min}}$ and $x\!=\!x_{\text{max}}$ that result in no net flux through the boundaries, the free-streaming equation \eqref{eq:free-streaming} conserves total mass, momentum, and energy.
\end{lemma}

\begin{proof}
We multiply \eqref{eq:free-streaming} by $m_s$ and sum over all $s$. We then multiply this by $v^p$ for any non-negative integer $p$, integrate over $v\in(-\infty,\infty)$, integrate-by-parts, and assume $f_s \rightarrow 0$ sufficiently fast as $\| \vec{v} \| \rightarrow \infty$, which results in
\begin{equation*}
	\mom_{p,t} + \mom_{p+1,x} = 0 \quad \text{where} \quad \mom_p := \sum_{s=1}^{\text{S}} m_s \int_{-\infty}^{\infty} v^p
	f_s \, dv.
\end{equation*}
If we take $p\!=\!0$, $1$, and $2$, and note that the electric field, $E$, is constant in time in Problem $\mathcal{A}$ \eqref{eq:free-streaming}, we arrive at conservation of total mass, momentum, and energy as stated in
either \eqref{eqn:conservation_1d1v_single} (in the single-species case) or \eqref{eqn:conservation_1d1v_multi} (in the multi-species case).
\end{proof}

\subsubsection{Problem $\mathcal{B}$: Harmonic oscillator and time-dependent advection}
\label{subsection:problemB}
Just as in \S\ref{subsection:problemA}, we look to solve Problem ${\mathcal B}$ \eqref{eq:acceleration} via the method of characteristics. We look for curves in the $(t,v)$-plane parameterized by $c \in \reals$ along which $f_s$ is constant:
\begin{equation}
\frac{d}{dc} f_{s}\left(t(c), x, v(c)\right) = \frac{dt}{dc} \frac{\partial f_s}{\partial t} 
+ \frac{dv}{dc} \frac{\partial f_s}{\partial v}  =
f_{s,t} + \frac{q_s}{m_s} E f_{s,v} = 0 \quad \Longrightarrow \quad
\frac{dt}{dc} = 1, \quad \frac{dv}{dc} = \frac{q_s}{m_s} E, \quad \frac{df_s}{dc} = 0.
\end{equation}
The curves along which $f_s$ is constant depend on the time-dependent (but velocity-independent) electric field. In particular, we arrive at the following exact solution of Problem $\mathcal{B}$ \eqref{eq:acceleration} at some time $t\!=\!t^n + \Delta t$, written here with initial condition $f_s\left(t^n,x,v\right)$:
\begin{equation}
\label{eqn:char-method-B}
	f_s\left( t^n + \Delta t, x, v \right) = f_s \left( t^n, x , v - \Delta t  \frac{q_s}{m_s}  E^{\star}(x)  \right), \quad
	\text{where} \quad
	E^{\star}(x) = \frac{1}{\Delta t}\int_{t^n}^{t^n+\Delta t} E\left(\tau, x \right) \, d\tau.
\end{equation}
Note that $E^{\star}(x)$ is the time-averaged electric field over $t \in \left[t^n, t^n + \Delta t \right]$.

The missing ingredient in the above solution is that we need to know the electric field; fortunately, finding the exact electric field in Problem ${\mathcal B}$ \eqref{eq:acceleration} is achievable. We first note that each species density, $\rho_s$, is constant in time in Problem ${\mathcal B}$ \eqref{eq:acceleration}:
\begin{equation}
\label{eq:problemB_mass}
\rho_{s,t} = \left(\int_{-\infty}^{\infty} f_s\,dv\right)_{,t} =
\int_{-\infty}^{\infty} f_{s,t} \,dv = -\frac{q_s}{m_s}E\int_{-\infty}^{\infty} f_{s,v}\,dv =0.
\end{equation}
Next, we compute the time derivative of the total current:
\begin{equation}\label{eq:J_evolution}
J_{,t} = \sum_{s=1}^{\text{S}} q_s\int_{-\infty}^{\infty} v\, f_{s,t}\,dv =
-\sum_{s=1}^{\text{S}} \frac{q_s^2}{m_s}E\int_{-\infty}^{\infty} v\,f_{s,v}\,dv =
\left(\sum_{s=1}^{\text{S}} \frac{q_s^2}{m_s}\rho_s^n\right)\!E
= \left(\omega^n\right)^2\!E,
\end{equation}
where integration-by-parts in $v$ was used, and the following quantity was defined:
\begin{equation}
\label{eqn:omega_def}
\omega^n\!\left(x\right) := \sqrt{\sum_{s=1}^{\text{S}} \frac{q_s^2}{m_s}\rho_s^n\!\left(x\right)}.
\end{equation}
Combining \eqref{eq:J_evolution} with Amp\`ere's law in \eqref{eq:acceleration} yields the following harmonic oscillator problem:
\begin{equation}\label{eq:field_current_system}
E_{,t} = -J,
\qquad
J_{,t} = \left(\omega^n\right)^2\!E,
\end{equation}
which can be solved exactly. We write the solution here with initial data $E\left(t^n,x\right)$ and $J\left(t^n,x\right)$ at time $t\!=\!t^n$:
\begin{align}
\label{eq:E_update}
E\left(t,x\right)
&= E^n\left(x\right) \cos\Bigl(\omega^n\!\left(x\right)\bigl(t - t^n\bigr)\Bigr)
-\frac{J^n\!\left(x\right)}{\omega^n\!\left(x\right)}\sin\Bigl(\omega^n\!\left(x\right) \bigl(t - t^n\bigr) \Bigr),
\\
J\left(t,x\right) &= J^n\!\left(x\right) \cos\Bigl(\omega^n\!\left(x\right) \bigl( t - t^n \bigr) \Bigr)
+\omega^n\!\left(x\right) E^n\!\left(x\right) \sin\Bigl(\omega^n\!\left(x\right) \bigl(t - t^n \bigr) \Bigr).
\label{eq:J_update}
\end{align}
This harmonic-oscillator solution also provides us with the missing time-averaged electric field, $E^{\star}$, needed in \eqref{eqn:char-method-B}:
\begin{equation}
\label{eq:E_star_HO}
    E^{\star}(x) = \frac{1}{\Delta t} \int_{t^n}^{t^{n}+\Delta t}E(\tau,x)\,d\tau
    = \frac{E^n(x)}{\omega^n(x)}\frac{\sin\bigl(\Delta t \, \omega^n(x) \bigr)}{\Delta t}
    - \frac{J^n(x)}{\left(\omega^n(x)\right)^2} \frac{1-\cos\bigl(\Delta t \, \omega^n(x)\bigr)}{\Delta t}.
\end{equation}

\begin{lemma}\label{thm:HO_invariant}
The acceleration and Amp\`ere system represented by \eqref{eq:acceleration} conserves total mass and energy.
\end{lemma}

\begin{proof}
Total mass conservation follows multiplying \eqref{eq:problemB_mass} by $m_s$, summing over the species index, $s$,
and integrating over $v\in\left(-\infty,\infty\right)$. Total energy conservation is established by multiplying the first equation in \eqref{eq:acceleration} by 
$v^2/2$ and integrating over $v\in\left(-\infty,\infty\right)$:
\begin{equation*}
\mathcal{E}_{s,t} + \frac{q_s}{m_s} E \int_{-\infty}^{\infty} \frac{1}{2} v^2 f_{s,v} = 0 \quad
\Longrightarrow \quad
\mathcal{E}_{s,t} - \frac{q_s}{m_s} E \int_{-\infty}^{\infty} v f_{s} = 0 \quad
\Longrightarrow \quad
\mathcal{E}_{s,t} - \frac{1}{m_s} E q_s \rho_s u_s = 0.
\end{equation*}
Multiplying this result by $m_s$ and summing over the species index, $s$, yields the desired result:
\begin{equation*}
\frac{\partial}{\partial t} \sum_{s=1}^{\text{S}} m_s \mathcal{E}_{s} - E J = 0 
\quad \Longrightarrow \quad
\frac{\partial}{\partial t} \sum_{s=1}^{\text{S}} m_s \mathcal{E}_{s} + E E_{,t} = 0
\quad \Longrightarrow \quad
\frac{\partial}{\partial t} \left( \sum_{s=1}^{\text{S}} m_s \mathcal{E}_{s} + \frac{1}{2} E^2 \right) = 0,
\end{equation*}
where to get from the first to the middle expression, we used Amp\`ere's equation from Problem $\mathcal{B}$ \eqref{eq:acceleration}.
\end{proof} 

\subsection{Discontinuous Galerkin (DG) framework and notation}\label{subsec:discretization}
The discussion in \S\ref{subsec:splitting} assumes an exact representation of the distribution functions, $f_s$, for any position $x \in \reals$ and velocity $v\in \reals$. In practice, the best we can do is represent $f_s$ using a finite-dimensional approximation. The same is true for the electric field, $E$, over $x\in \reals$. In this work, we follow Rossmanith and Seal \cite{RossmanithSeal2011} and represent the numerical distribution functions and the electric field using a discontinuous Galerkin approach (e.g., see Cockburn and Shu \cite{CockburnShu1998_JCP_V} and references therein).

\subsubsection{DG representation of the PDF}
In this work, we employ a two-dimensional DG representation on Cartesian meshes on the computational domain, where the velocity mesh depends on the plasma species $s$:
\begin{equation}
\Omega^{(s)} = \Omega_x \times \Omega_v^{(s)} = \left( x_{\min}, x_{\max} \right) \times \left( v^{(s)}_{\min}, v^{(s)}_{\max} \right), \quad \text{for each} \, s \in \left\{e, i \right\}.
\end{equation}
We partition $\Omega^{(s)}$ into $N_x \times N_v$ non-overlapping rectangles of the form:
\begin{equation}
\Tm^{(s)}_{ij}=\left(x_i-\tfrac{\Delta x}{2}, \, x_i+\tfrac{\Delta x}{2}\right) \times
          \left(v^{(s)}_j-\tfrac{\Delta v^{(s)}}{2}, \, v^{(s)}_j+\tfrac{\Delta v^{(s)}}{2}\right), \quad
          i=1,\ldots,N_x, \quad j=1,\ldots,N_v,
\end{equation}
with centers and grid spacings given by
\begin{equation}
   x_i = x_{\text{min}} + \left( i - \frac{1}{2} \right){\Delta x}, \quad
   v^{(s)}_j = v^{(s)}_{\text{min}} + \left( j - \frac{1}{2} \right){\Delta v^{(s)}}, \quad
   {\Delta x} = \frac{x_{\text{max}}-x_{\text{min}}}{N_x}, \quad
   {\Delta v}^{(s)} = \frac{v^{(s)}_{\text{max}}-v^{(s)}_{\text{min}}}{N_v}.
\end{equation}
Note that the number of velocity levels, $N_v$, is species-independent, while the velocity range, $v \in
\left(v^{(s)}_{\text{min}}, \, v^{(s)}_{\text{max}}\right)$, and the velocity grid spacing, $\Delta v^{(s)}$, depends on the species.
Over the domain $\Omega^{(s)}$ we then define the broken finite element space 
\begin{equation}
\label{eqn:dg-fem-space}
\mathcal{W}^{\left(\Delta x, \Delta v^{(s)}\right)}:=\left\{w^{\Delta x}\in L^\infty\!\left(\Omega^{(s)}\right): w^{\left(\Delta x,\Delta v^{(s)}\right)}\bigl\vert_{\mathcal{T}^{(s)}_{ij}}\in \mathbb{P}\left(M_\text{deg}\right) \quad \forall \mathcal{T}^{(s)}_{ij} \in \Omega^{(s)}\right\},
\end{equation}
where $\mathbb{P}(M_\text{deg})$ is the set of all polynomials of total degree at most $M_\text{deg} \ge 0$.
The representation is called {\it discontinuous} because it does not require continuity between neighboring elements.

Each cell, $\Tm^{(s)}_{ij}$, can be mapped to the reference element $(\xi,\eta)\in(-1,1)^2$ via the affine transformation:
\begin{equation}
x\left(\xi\right)=x_i+\xi\,\tfrac{\Delta x}{2},\qquad v\left(\eta\right)=v^{(s)}_j+\eta\,\tfrac{\Delta v^{(s)}}{2}.
\end{equation} 
On this reference element, we define the orthonormal Legendre polynomial basis:
\begin{equation}
\label{eqn:orthonormal}
\Bigl\{\phi_{\ell}(\xi,\eta)\Bigr\}_{\ell=1}^{M_\text{basis}}, \quad \text{where} \quad
\left\langle \phi_{\ell},\phi_{k} \right\rangle
:=\frac{1}{4} \iint_{-1}^{1}\phi_{\ell}(\xi,\eta)\,\phi_{k}(\xi,\eta)\,d\xi\,d\eta
=\begin{cases}
1 &  \text{if} \quad k=\ell, \\
0 &  \text{if} \quad k \ne \ell,
\end{cases}
\end{equation}
where $M_\text{basis}\!=\!\left(M_\text{deg}+1\right)\left(M_\text{deg}+2\right)/2$ is the number of basis functions. As an example, the $\mdeg\!=3$ ($\mbasis\!=10$) basis, which is used extensively throughout this work, contains the following mutually orthonormal basis functions:
\begin{equation}
\label{eqn:legendre_basis}
\begin{split}
\mathbb{P}\!\left(3\right) = \text{Span}\Biggl\{ & 1, \, \sqrt{3}\xi, \, \sqrt{3}\eta, \, \frac{\sqrt{5}}{2}\left(3\xi^2-1\right), \, 3\xi \eta, \, \frac{\sqrt{5}}{2}\left(3\eta^2-1\right), \, \frac{\sqrt{7}}{2}\left(5\xi^3-3\xi\right), 
\\ & \frac{\sqrt{15}}{2} \eta \left(3\xi^2-1\right),
\, \frac{\sqrt{15}}{2} \xi \left(3\eta^2-1\right),
\, \frac{\sqrt{7}}{2}\left(5\eta^3-3\eta\right)
\Biggr\}.
\end{split}
\end{equation}
Numerical solutions have the following representation on each element:
\begin{align}\label{eqn:q_ansatz}
   f^{\left(\Delta x, \Delta v\right)}_s\left(t,x\left(\xi\right),v\left(\eta\right)\right)\big|_{\Tm^{(s)}_{ij}} 
   &:= F_{\left(i,j,s\right)}(t,\xi,\eta)
:=\sum_{\ell=1}^{M_\text{basis}} Q_{\left(i, j, s, \ell\right)}(t)\,\phi_{\ell}(\xi,\eta),
\\ \text{where} \quad
Q_{\left(i, j, s, \ell\right)}(t) &\approx \left\langle \phi_{\ell}, \, f_s\left(t,x\left(\xi\right),v\left(\eta\right) \right)\big|_{\Tm^{(s)}_{ij}} \right\rangle.
\end{align}

We note that given a sufficiently smooth particle density function, $f_s$, in $(x,v)\in\reals\times\reals$, the discontinuous Galerkin ansatz \eqref{eqn:q_ansatz} produces an approximation with the following accuracy:
\begin{equation}
\label{eqn:theoretical_order_accuracy}
\max_{\left(\xi,\eta\right) \in [-1,1]^2} \left| f_s\bigl(t,x\left(\xi\right),v\left(\eta\right)\bigr)\big|_{\Tm^{(s)}_{ij}} 
   - F_{\left(i,j,s\right)}(t,\xi,\eta) \right| = {\mathcal O}\left( \Delta x^\morder + \Delta v^\morder \right) \quad \text{where} \quad
   \morder = \mdeg + 1.
\end{equation}
In other words, $\morder=\mdeg+1$ is the order of accuracy of the DG representation \cite{CockburnShu1998_JCP_V}.

\subsubsection{DG representation of the electric field and PDF moments}
In order to represent the electric field, $E(t,x)\!\!:\!\!\reals_{\ge 0} \times \reals \mapsto \reals$, we introduce a one-dimensional version of the Legendre basis:
\begin{equation}
\label{eqn:legendre_basis_1d}
\mathbb{P}_{\xi}^{\text{1D}}\!\left(\mdeg\right) = \text{Span}\Biggl\{ \phi^{\text{1D}}_1\!\left(\xi\right), \,
\phi^{\text{1D}}_2\!\left(\xi\right), \, \ldots, \, \phi^{\text{1D}}_{\morder}\!\left(\xi\right) \Biggr\}, \quad
\text{where} \quad \morder = \mdeg + 1.
\end{equation}
The basis functions are mutually orthonormal in the following sense:
\begin{equation}
\left\langle \phi^{\text{1D}}_{\ell},\phi^{\text{1D}}_{k} \right\rangle_{\text{1D}}
:=\frac{1}{2} \int_{-1}^{1} \phi^{\text{1D}}_{\ell}(\xi)\,\phi^{\text{1D}}_{k}(\xi)\,d\xi
=\begin{cases}
1 &  \text{if} \quad k=\ell, \\
0 &  \text{if} \quad k \ne \ell.
\end{cases}
\end{equation}
For example, the basis for the case $\mdeg\!=3$, which is widely used throughout this work, is given by
\begin{equation}
\mathbb{P}_{\xi}^{\text{1D}}\!\left(3\right) = \text{Span}\Biggl\{ 1, \, \sqrt{3}\xi, \, \frac{\sqrt{5}}{2}\left(3\xi^2-1\right), \, \frac{\sqrt{7}}{2}\left(5\xi^3-3\xi\right) \Biggr\}.
\end{equation}

The electric field then has the following representation on each 1D element, $\Tm^{\text{1D}}_i\!=\!\left( x_i - {\Delta x}/{2}, \, x_i + {\Delta x}/{2} \right)$:
\begin{align}
\label{eqn:Ef_ansatz}
   E^{\Delta x}\!\left(t,x\left(\xi\right)\right)\big|_{\Tm^{\text{1D}}_{i}} 
   &:= \sum_{\ell=1}^{\morder} {E}_{\left(i, \ell\right)}(t) \, \phi^{\text{1D}}_{\ell}(\xi),
\quad \text{where} \quad
{E}_{\left(i, \ell\right)}(t) \approx \left\langle \phi^{\text{1D}}_{\ell}, \, E\left(t,x\left(\xi\right) \right)\big|_{\Tm^{\text{1D}}_{i}} \right\rangle_{\text{1D}}.
\end{align}
The first three moments of the DG-represented particle density function (PDF) \eqref{eqn:q_ansatz}, have the following form:
\begin{align}
\label{eqn:dg-moments-expansion}
  \Bigl\{ \rho_s^{\Delta x}, \, \rho u^{\Delta x}_s, \, {\mathcal E}^{\Delta x}_s \Bigr\}\!\left(t,x\left(\xi\right)\right)\big|_{\Tm^{\text{1D}}_{i}}
  &:= \sum_{\ell=1}^{\morder} \Bigl\{ {\rho}_{\left(i, s, \ell\right)}, \, 
  {\rho u}_{\left(i, s, \ell\right)}, \, {\mathcal{E}}_{\left(i, s, \ell\right)} \Bigr\}(t) \, \phi^{\text{1D}}_{\ell}(\xi),
\end{align}
where
\begin{align}
\label{eqn:dg-moments-expansion-integrals}
\Bigl\{ {\rho}_{\left(i, s, \ell\right)}, \, 
  {\rho u}_{\left(i, s, \ell\right)}, \, {\mathcal{E}}_{\left(i, s, \ell\right)} \Bigr\} = 
\frac{\Delta v^{(s)}}{4} \sum_{j=1}^{N_v}\iint_{-1}^{1} \left\{ 1, \, v^{(s)}_j + \eta \frac{\Delta v^{s}}{2}, \, \frac{1}{2} \left( v^{(s)}_j + \eta \frac{\Delta v^{s}}{2} \right)^2 \right\} F_{\left(i,j,s\right)} \, \phi^{\text{1D}}_{\ell} \, d\xi d\eta.
\end{align}
If we define the following sums:
\begin{equation}
\left\{ M_{(i,s,\ell)}^{0}, \, M_{(i,s,\ell)}^{1}, \, M_{(i,s,\ell)}^{2} \right\} := \Delta v^{(s)} \sum_{j=1}^{N_v} \left\{ 1, \, v^{(s)}_j, \, \frac{1}{2} {\left(v^{(s)}_j\right)^2} + \frac{1}{24} {\left(\Delta v^{s}\right)^2}  \right\} Q_{\left(i,j,s,\ell\right)},
\end{equation}
and if we take $\mdeg=3$, then we can reduce \eqref{eqn:dg-moments-expansion-integrals} to the following:
\begin{align}
\vec{{\rho}_{\left(i, s\right)}} &= \left\{ M_{(i,s,1)}^{0}, \, M_{(i,s,2)}^{0}, \, M_{(i,s,4)}^{0}, \, M_{(i,s,7)}^{0} \right\}, \\
\vec{{\rho u}_{\left(i, s\right)}} &= \left\{ M_{(i,s,1)}^{1}, \, M_{(i,s,2)}^{1}, \, M_{(i,s,4)}^{1}, \, M_{(i,s,7)}^{1} \right\}
+ \frac{\Delta v^{(s)}}{2\sqrt{3}} \left\{ M_{(i,s,3)}^{0}, \, M_{(i,s,5)}^{0}, \, M_{(i,s,8)}^{0}, \, 0 \right\}, \\
\begin{split}
\vec{\mathcal{E}_{\left(i, s\right)}} &= \left\{ M_{(i,s,1)}^{2}, \, M_{(i,s,2)}^{2}, \, M_{(i,s,4)}^{2}, \, M_{(i,s,7)}^{2} \right\} + \frac{\Delta v^{(s)}}{2\sqrt{3}} \left\{ M_{(i,s,3)}^{1}, \, M_{(i,s,5)}^{1}, \, M_{(i,s,8)}^{1}, \, 0 \right\} \\
&+ \frac{\left(\Delta v^{(s)}\right)^2}{12\sqrt{5}} \left\{ M_{(i,s,6)}^{0}, \, M_{(i,s,9)}^{0}, \, 0, \, 0 \right\}.
\end{split}
\end{align}

\subsection{Semi-Lagrangian discontinuous Galerkin (SLDG) method}
\label{subsec:sldg_vlasov}
In \S\ref{subsec:splitting} we described how to formulate and exactly solve the two sub-problems, referred to as Problems ${\mathcal A}$ and ${\mathcal B}$. The exact solutions presented in \S\ref{subsec:splitting} assume that we have access to the full solution at the previous time-step (i.e., $f_s\left(t^n,x,v\right)$ for any $(x,v)\in\reals \times \reals$). However, as described in \S\ref{subsec:discretization}, all we have access to is the discontinuous Galerkin representation of the solution as expressed by \eqref{eqn:q_ansatz}. In this subsection, we describe how to utilize the exact solutions presented in \S\ref{subsec:splitting} to produce unconditionally stable, high-order accurate, SLDG methods for both Problems ${\mathcal A}$ and ${\mathcal B}$. The approach described here is based on the method of Rossmanith and Seal \cite{RossmanithSeal2011}.

\subsubsection{SLDG for Problem ${\mathcal A}$}
\label{subsec:sldg_for_problem_A}
In this subsection, we are concerned with solving Problem $\mathcal{A}$ \eqref{eq:free-streaming} given an initial condition at time $t\!=\!t^n$ of the following form:
\begin{equation}
\label{eq:sldg_species_expansion}
F^{n}_{\left(i,j,s\right)}(\xi,\eta)
=\sum_{\ell=1}^{M_\text{basis}} Q^n_{\left(i, j, s, \ell\right)} \,\phi_{\ell}(\xi,\eta), 
\quad i=1,\ldots,N_x, \quad j=1,\ldots,N_v, \quad  s=1,\ldots,\text{S}.
\end{equation}
Our first goal is to rewrite Problem $\mathcal{A}$ \eqref{eq:free-streaming} as a collection of purely one-dimensional advection problems (i.e., instead of a one-dimensional advection problem embedded in two dimensions). To achieve this, we sample \eqref{eq:sldg_species_expansion} at $\morder$ quadrature points in $v$ (i.e., the {\it transverse} direction) using the Gauss-Legendre \cite{article:GolubWelsch1969} points and weights:
\begin{equation}
\label{eqn:sldg_slice_v}
\bigl\{ \eta_k, \, \omega_k \bigr\}_{k=1}^{\morder} \quad \Longrightarrow \quad
v_{\left(j,s,k\right)} := v^{(s)}_j + \eta_k \frac{\Delta v}{2},
\end{equation}
where $\morder = \mdeg + 1$. For the $\mdeg=3$ case, which is widely used in the current work, the relevant values are:
\begin{equation}
\eta_1, \eta_4 = \pm \sqrt{\frac{3}{7} + \frac{2}{7}\sqrt{\frac{6}{5}}}, \quad
\eta_2, \eta_3 = \pm \sqrt{\frac{3}{7} - \frac{2}{7}\sqrt{\frac{6}{5}}}, \quad
\omega_1,\omega_4 = \frac{18-\sqrt{30}}{36}, \quad
\omega_2,\omega_3 = \frac{18+\sqrt{30}}{36}.
\end{equation}
At each of the quadrature points in $v$, we consider a restricted one-dimensional version of Problem $\mathcal{A}$ \eqref{eq:free-streaming} along the line $v\!=\!v_{\left(j,s,k\right)}$:
\begin{equation}
\label{eq:sldg_free_line}
    f_{s,t}+v_{\left(j,s,k\right)} f_{s,x}=0.
\end{equation}

Next, we seek to advance the DG solution as given by \eqref{eq:sldg_species_expansion}, from $t\!=\!t^n$ to $t\!=\!t^n + \Delta t$, at the fixed velocity, $v\!=\!v_{\left(j,s,k\right)}$, by solving \eqref{eq:sldg_free_line} for each $k$ using the following {\it semi-Lagrangian} philosophy:
\begin{description}
\item[\bf Step 1 (shift):] Using the method of characteristics solution \eqref{eqn:char-method-A} as inspiration, shift the current solution, \eqref{eq:sldg_species_expansion}, at $v\!=\!v_{\left(j,s,k\right)}$ by a distance $v_{\left(j,s,k\right)} \Delta t$.
\item[\bf Step 2 (project):] However, in general, a direct shift by a distance $v_{\left(j,s,k\right)} \Delta t$ will result in a solution at $t\!=\!t^n + \Delta t$ that is no longer in the original DG solution space (i.e., the shifted solution will be a piecewise polynomial that is not aligned with the original mesh). Therefore, to return to the original mesh, we apply an $L^2$-projection.
\end{description}

To effectuate the above shift-then-project semi-Lagrangian philosophy over a time step $\Delta t$, we first decompose the ratio of the total distance a particle will travel, $v_{\left(j,s,k\right)} \Delta t $, to the mesh spacing, $\Delta x$, into integer and fractional parts:
\begin{equation}
\label{eq:sldg_x_shift}
    \text{Integer part:} \quad \mathcal{I}_{\left(j,s,k\right)}
    :=
    \left\lfloor \frac{v_{\left(j,s,k\right)} \Delta t}{\Delta x}\right\rfloor,
    \qquad
    \text{Fractional part:} \quad \mu_{\left(j,s,k\right)}
    :=
    \frac{v_{\left(j,s,k\right)}\Delta t}{\Delta x}-\mathcal{I}_{\left(j,s,k\right)},
\end{equation}
where $\lfloor \cdot \rfloor$ is the {\it floor} operation (i.e., the input is rounded down to the greatest integer that is less than or equal to the input) and $0\le \nu_{\left(j,s,k\right)}<1$. Here are two quick examples to illustrate the integer/fraction decomposition:
\begin{align}
  \text{Positive velocity:} & \quad \frac{v_{\left(j,s,k\right)} \Delta t}{\Delta x} = +7.215 \quad \Longrightarrow \quad
     \left\lfloor \frac{v_{\left(j,s,k\right)} \Delta t}{\Delta x}\right\rfloor = +7
     \quad \text{and} \quad 
     \mu_{\left(j,s,k\right)} = 0.215, \\
   \text{Negative velocity:} & \quad   \frac{v_{\left(j,s,k\right)} \Delta t}{\Delta x} = -7.215 \quad \Longrightarrow \quad
     \left\lfloor \frac{v_{\left(j,s,k\right)} \Delta t}{\Delta x}\right\rfloor = -8
     \quad \text{and} \quad 
     \mu_{\left(j,s,k\right)} = 0.785.
\end{align}
The key observations are: (1) the integer part of the shift is just a change of index, and (2) the fractional part of the shift is the part that requires $L^2$-projection back onto the original mesh. 

We define a tensor that can map the original two-dimensional basis coefficients in \eqref{eq:sldg_species_expansion} to one-dimensional coefficients for basis \eqref{eqn:legendre_basis_1d}
at each of the quadrature points, $\eta\!=\!\eta_k$:
\begin{equation}
A^{\text{2D}\rightarrow\text{1D}}_{\left(m, \, k, \, \ell\right)} = \frac{1}{2} \int_{-1}^{1}
  \phi^{\text{1D}}_{m}\!\left(\xi\right) \, \phi_{\ell}\!\left(\xi,\eta_k\right) \, d\xi, \quad
  m=1,\ldots,\morder, \quad k=1,\ldots,\morder, \quad \ell = 1, \ldots, \mbasis.
\end{equation}
The resulting one-dimensional coefficients are then given by
\begin{equation}
\label{eqn:projected-ic-x}
  Q^{\text{1D}}_{(i,j,s,m,k)} = \sum_{\ell=1}^{\mbasis} A^{\text{2D}\rightarrow\text{1D}}_{\left(m, \, k, \, \ell\right)} \, 
   Q^n_{\left(i, j, s, \ell\right)},
\end{equation}
where $m$ is the basis index and $k$ is the $\eta_k$ quadrature point index.

Having constructed several purely one-dimensional problems, we use the shift-then-project philosophy described above to create an approximate solution to \eqref{eq:sldg_free_line} with time step $\Delta t$ and initial data \eqref{eqn:projected-ic-x}: 
\begin{equation}
\label{eq:sldg_line_projection}
\begin{split}
Q^{\text{1D,\text{new}}}_{\left(i, \, j, \, s, \, m, \, k \right)}
&=
\frac{1}{2}
\sum_{\ell=1}^{\morder}
Q_{\bigl(i-1-\mathcal{I}_{\left(j,s,k\right)}, \, j, \, s, \, \ell, \, k\bigr)}^{\text{1D}}
\int_{-1}^{-1+2\mu_{\left(j, s, k\right)}}
\phi^{\text{1D}}_{\ell}\bigl(\xi+2-2\mu_{\left(j, s, k\right)}\bigr) \,
\phi^{\text{1D}}_{m}\bigl(\xi\bigr)\,d\xi
\\
&+
\frac{1}{2}
\sum_{\ell=1}^{\morder}
Q_{\bigl(i-\mathcal{I}_{\left(j,s,k\right)}, \, j, \, s, \, \ell, \, k\bigr)}^{\text{1D}}
\int_{-1+2\mu_{\left(j,s,k\right)}}^{1}
\phi^{\text{1D}}_{\ell}\bigl(\xi-2\mu_{\left(j,s,k\right)} \bigr) \, 
\phi^{\text{1D}}_{m}\bigl(\xi \bigr)\,d\xi.
\end{split}
\end{equation}
Seen in this formula is the index shift due to the integer part of $v_{\left(j,s,k\right)} \Delta t\bigl/\Delta x$, as well as the projection integrals due to the fractional part (see the definitions in \eqref{eq:sldg_x_shift}).

Finally, in order to reconstruct the two-dimensional Legendre coefficients, we define the following tensor that maps from 1D to 2D Legendre coefficients:
\begin{equation} 
\label{eqn:Q1D_to_Qnew_x}
A^{\text{1D}\rightarrow\text{2D}}_{\left(\ell, \, m, \, k\right)} = \frac{\omega_{k}}{4}  \int_{-1}^{1}
  \phi^{\text{1D}}_{m}\!\left(\xi\right) \, \phi_{\ell}\!\left(\xi,\eta_k\right) \, d\xi, \quad \ell = 1, \ldots, \mbasis, \quad
  m=1,\ldots,\morder, \quad k=1,\ldots,\morder.
\end{equation}
The application of this mapping produces the new Legendre coefficients at time $t\!=\!t^{n+1}\!=\!t^n+\Delta t$ that give an approximate solution to Problem $\mathcal{A}$ \eqref{eq:free-streaming}:
\begin{equation}
\label{eq:sldg_x_update}
Q_{\left(i, \, j, \, s, \, \ell \right)}^{n+1}
    =
    \sum_{m=1}^{\morder} \sum_{k=1}^{\morder}
    A^{\text{1D}\rightarrow\text{2D}}_{\left(\ell, \, m, \, k \right)} \, Q^{\text{1D},\text{new}}_{\left(i, \, j, \, s, \, m, \, k \right)},
    \quad \text{where} \quad
    \ell=1,\ldots,\mbasis,
\end{equation}
where again $\morder\!=\!\mdeg + 1$, $\mbasis\!=\!(\mdeg+1)(\mdeg+2)/2$, and $\mdeg$ is the chosen maximum polynomial order for the discontinuous Galerkin finite element space \eqref{eqn:dg-fem-space}.
\subsubsection{SLDG for Problem ${\mathcal B}$}
\label{subsec:sldg_for_problem_B}
In this subsection, we are concerned with solving Problem $\mathcal{B}$ \eqref{eq:acceleration} given an initial condition at time $t\!=\!t^n$ of the form \eqref{eq:sldg_species_expansion}. 
As in the previous subsection, our first goal is to rewrite Problem $\mathcal{B}$ \eqref{eq:acceleration} as a collection of purely one-dimensional advection problems (i.e., instead of a one-dimensional advection problem embedded in two dimensions). To achieve this, we sample \eqref{eq:sldg_species_expansion} at $\morder$ quadrature points in $x$ (i.e., the {\it transverse} direction) using the Gauss-Legendre \cite{article:GolubWelsch1969} points and weights:
\begin{equation}
\label{eqn:sldg_slice_x}
\bigl\{ \xi_k, \, \omega_k \bigr\}_{k=1}^{\morder} \quad \Longrightarrow \quad
x_{\left(i,k\right)} := x_i + \xi_k \frac{\Delta x}{2}.
\end{equation}
At each of the quadrature points in $x$, we consider a restricted one-dimensional version of Problem $\mathcal{B}$ \eqref{eq:acceleration} along the line $x\!=\!x_{\left(i,s,k\right)}$:
\begin{equation}
\label{eq:sldg_acc_line}
    f_{s,t} + E_{\left(i,s,k\right)} f_{s,v}=0, \quad \text{where} \quad
      E_{\left(i,s,k\right)} := \frac{q_s}{m_s} E^{\star}\!\left(x_{\left(i,k\right)} \right)
\end{equation}
and $E^{\star}$ is given by \eqref{eq:E_star_HO}.

Next, we seek to advance the DG solution as given by \eqref{eq:sldg_species_expansion}, from $t\!=\!t^n$ to $t\!=\!t^n + \Delta t$, at the fixed position, $x\!=\!x_{\left(i,s,k\right)}$, by solving \eqref{eq:sldg_acc_line} for each $k$ using the following {\it semi-Lagrangian} philosophy:
\begin{description}
\item[\bf Step 1 (shift):] Using the method of characteristics solution \eqref{eqn:char-method-B} as inspiration, shift the current solution, \eqref{eq:sldg_species_expansion}, at $x\!=\!x_{\left(i,s,k\right)}$ by a distance $E_{\left(i,s,k\right)} \Delta t$.
\item[\bf Step 2 (project):] However, in general, a direct shift by a distance $E_{\left(i,s,k\right)} \Delta t$ will result in a solution at $t\!=\!t^n + \Delta t$ that is no longer in the original DG solution space (i.e., the shifted solution will be a piecewise polynomial that is not aligned with the original mesh). Therefore, to return to the original mesh, we apply an $L^2$-projection.
\end{description}

To effectuate the above shift-then-project semi-Lagrangian philosophy over a time step $\Delta t$, we first decompose the ratio of the total distance a particle will travel, $E_{\left(i,s,k\right)} \Delta t$, to the mesh spacing, $\Delta v$, into integer and fractional parts:
\begin{equation}
\label{eq:sldg_v_shift}
    \text{Integer part:} \quad \mathcal{J}_{\left(i,s,k\right)}
    :=
    \left\lfloor \frac{E_{\left(i,s,k\right)} \Delta t}{\Delta v}\right\rfloor,
    \qquad
    \text{Fractional part:} \quad \nu_{\left(i,s,k\right)}
    :=
    \frac{E_{\left(i,s,k\right)}\Delta t}{\Delta v}-\mathcal{J}_{\left(i,s,k\right)},
\end{equation}
where $0\le \nu_{\left(i,s,k\right)}<1$. The key observations are: (1) the integer part of the shift is just a change of index, and (2) the fractional part of the shift is the part that requires $L^2$-projection back onto the original mesh. 

We introduce a one-dimensional version of the Legendre basis:
\begin{equation}
\label{eqn:legendre_basis_1d_eta}
\mathbb{P}^{\text{1D}}_{\eta}\!\left(\mdeg\right) = \text{Span}\Biggl\{ \phi^{\text{1D}}_1\!\left(\eta\right), \,
\phi^{\text{1D}}_2\!\left(\eta\right), \, \ldots, \, \phi^{\text{1D}}_{\morder}\!\left(\eta\right) \Biggr\}, \quad
\text{where} \quad \morder = \mdeg + 1.
\end{equation}
We also define a tensor that can map the original two-dimensional basis coefficients in \eqref{eq:sldg_species_expansion} to one-dimensional coefficients for basis \eqref{eqn:legendre_basis_1d}
at each of the quadrature points, $\xi\!=\!\xi_k$:
\begin{equation}
B^{\text{2D}\rightarrow\text{1D}}_{\left(m, \, k, \, \ell\right)} = \frac{1}{2} \int_{-1}^{1}
  \phi^{\text{1D}}_{m}\!\left(\eta\right) \, \phi_{\ell}\!\left(\xi_k,\eta\right) \, d\eta, \quad
  m=1,\ldots,\morder, \quad k=1,\ldots,\morder, \quad \ell = 1, \ldots, \mbasis.
\end{equation}
The resulting one-dimensional coefficients are then given by
\begin{equation}
\label{eqn:projected-ic-v}
  Q^{\text{1D}}_{(i,j,s,m,k)} = \sum_{\ell=1}^{\mbasis} B^{\text{2D}\rightarrow\text{1D}}_{\left(m, \, k, \, \ell\right)} \, 
   Q^n_{\left(i, j, s, \ell\right)},
\end{equation}
where $m$ is the basis index and $k$ is the $\xi_k$ quadrature point index.

Having constructed several purely one-dimensional problems, we use the shift-then-project philosophy described above to create an approximate solution to \eqref{eq:sldg_acc_line} with time step $\Delta t$ and initial data \eqref{eqn:projected-ic-v}: 
\begin{equation}
\label{eq:sldg_line_projection_v}
\begin{split}
Q^{\text{1D},\text{new}}_{\left(i, \, j, \, s, \, m, \, k \right)}
&=
\frac{1}{2}
\sum_{\ell=1}^{\morder}
Q_{\bigl(i, \, j-1-\mathcal{J}_{\left(i,s,k\right)}, \, s, \, \ell, \, k\bigr)}^{\text{1D}}
\int_{-1}^{-1+2\nu_{\left(i, s, k\right)}}
\phi^{\text{1D}}_{\ell}\bigl(\eta+2-2\nu_{\left(i, s, k\right)}\bigr) \,
\phi^{\text{1D}}_{m}\bigl(\eta\bigr)\,d\eta
\\
&+
\frac{1}{2}
\sum_{\ell=1}^{\morder}
Q_{\bigl(i, \, j-\mathcal{J}_{\left(i,s,k\right)}, \, s, \, \ell, \, k\bigr)}^{\text{1D}}
\int_{-1+2\nu_{\left(i,s,k\right)}}^{1}
\phi^{\text{1D}}_{\ell}\bigl(\eta-2\nu_{\left(i,s,k\right)} \bigr) \, 
\phi^{\text{1D}}_{m}\bigl(\eta \bigr)\,d\eta.
\end{split}
\end{equation}
Seen in this formula is the index shift due to the integer part of $E_{\left(i,s,k\right)} \Delta t\bigl/\Delta v$, as well as the projection integrals due to the fractional part (see the definitions in \eqref{eq:sldg_v_shift}).

In order to reconstruct the two-dimensional Legendre coefficients, we define the following tensor that maps from 1D to 2D Legendre coefficients:
\begin{equation} 
B^{\text{1D}\rightarrow\text{2D}}_{\left(\ell, \, m, \, k\right)} = \frac{\omega_{k}}{4}  \int_{-1}^{1}
  \phi^{\text{1D}}_{m}\!\left(\eta\right) \, \phi_{\ell}\!\left(\xi_k,\eta\right) \, d\eta, \quad \ell = 1, \ldots, \mbasis, \quad
  m=1,\ldots,\morder, \quad k=1,\ldots,\morder.
\end{equation}
The application of this mapping produces the new Legendre coefficients at time $t\!=\!t^{n+1}\!=\!t^n+\Delta t$ that give an approximate solution to Problem $\mathcal{B}$ \eqref{eq:acceleration}:
\begin{equation}
\label{eq:sldg_v_update}
Q_{\left(i, \, j, \, s, \, \ell \right)}^{n+1}
    =
    \sum_{m=1}^{\morder} \sum_{k=1}^{\morder}
    B^{\text{1D}\rightarrow\text{2D}}_{\left(\ell, \, m, \, k \right)} \, Q^{\text{1D},\text{new}}_{\left(i, \, j, \, s, \, m, \, k \right)},
    \quad \text{where} \quad
    \ell=1,\ldots,\mbasis,
\end{equation}
where again $\morder\!=\!\mdeg + 1$, $\mbasis\!=\!(\mdeg+1)(\mdeg+2)/2$, and $\mdeg$ is the chosen maximum polynomial order for the discontinuous Galerkin finite element space \eqref{eqn:dg-fem-space}. 

\subsection{Positivity-preserving limiter}
\label{subsec:pp_limiter_multi}
Although the numerical updates provided in the previous section for sub-problems $\mathcal{A}$ and $\mathcal{B}$ are each high-order accurate in $x$ and $v$, they do not guarantee positivity of the discrete solution. In this subsection, we briefly review the positivity limiter of Rossmanith and Seal \cite{RossmanithSeal2011} and extend it to the multi-species case.

\subsubsection{Positivity in Problem $\mathcal{A}$}
Positivity-preservation in Problem $\mathcal{A}$ requires that positive-in-the-mean initial conditions, $Q^n$, produce positive-in-the-mean solutions, $Q^{n+1}$, after the full numerical update given by \eqref{eqn:projected-ic-x}, \eqref{eq:sldg_line_projection}, and \eqref{eq:sldg_x_update}:
\begin{equation}
\label{eqn:positivity-preserving}
Q^{n}_{\left(i,j,s,1\right)} \ge \varepsilon > 0 \quad \forall i,j,s \quad \Longrightarrow \quad
Q^{n+1}_{\left(i,j,s,1\right)} \ge \varepsilon > 0 \quad \forall i,j,s.
\end{equation}
We can uncover what happens in the numerical method by setting $m=1$ in \eqref{eq:sldg_line_projection} (i.e., the mean value along $\eta=\eta_k$), combining this expression with \eqref{eqn:projected-ic-x}, and
using definition \eqref{eq:sldg_species_expansion} to obtain:
\begin{equation}
\label{eqn:ave_line_integral_x_version_1}
\begin{split}
Q^{\text{1D}}_{\left(i, \, j, \, s, \, 1, \, k \right)} &=
\frac{1}{2}
\int_{-1}^{-1+2\mu_{\left(j, s, k\right)}}
F^{n}_{\bigl(i-1-\mathcal{I}_{\left(j,s,k\right)}, \, j, \, s\bigr)}\!\Bigl(\xi+2-2\mu_{\left(j, s, k\right)}, \, \eta_k\Bigr) \,d\xi
\\
&+
\frac{1}{2}
\int_{-1+2\mu_{\left(j,s,k\right)}}^{1}
F^{n}_{\bigl(i-\mathcal{I}_{\left(j,s,k\right)}, \, j, \, s\bigr)}\!\Bigl(\xi-2\mu_{\left(j,s,k\right)}, \, \eta_k\Bigr) \,d\xi.
\end{split}
\end{equation}
Via simple affine transformations for each integral:
\begin{align}
\text{First integral:} & \quad \xi = -1 + \mu_{\left(j,s,k\right)}\!\left(\alpha + 1 \right), \quad d\xi = \mu_{\left(j,s,k\right)} \, d\alpha, \\
\text{Second integral:} & \quad \xi = \mu_{\left(j,s,k\right)} + \alpha\!\left(1- \mu_{\left(j,s,k\right)}\right), \quad d\xi = \left(1-\mu_{\left(j,s,k\right)} \right) \, d\alpha,
\end{align}
we can rewrite \eqref{eqn:ave_line_integral_x_version_1} as 
\begin{equation}
\label{eqn:ave_update_x}
\begin{split}
Q^{\text{1D}}_{\left(i, \, j, \, s, \, 1, \, k \right)} &=
\frac{\mu_{\left(j,s,k\right)}}{2}
\int_{-1}^{1}
F^{n}_{\bigl(i-1-\mathcal{I}_{\left(j,s,k\right)}, \, j, \, s\bigr)}\!\Bigl(
1 + \mu_{\left(j,s,k\right)}\!\left(\alpha - 1 \right), \, \eta_k\Bigr) \,d\alpha
\\
&+
\frac{1-\mu_{\left(j,s,k\right)}}{2}
\int_{-1}^{1}
F^{n}_{\bigl(i-\mathcal{I}_{\left(j,s,k\right)}, \, j, \, s\bigr)}\!\Bigl(-\mu_{\left(j,s,k\right)} + \alpha\!\left(1- \mu_{\left(j,s,k\right)}\right), \, \eta_k\Bigr) \,d\alpha.
\end{split}
\end{equation}

The integrand in each integral of \eqref{eqn:ave_update_x} is a polynomial in $\alpha$ of degree strictly less than $\morder$. Therefore, each integral can be computed exactly via the minimum Gauss-Legendre quadrature points of $\lceil\morder/2\rceil$ points\footnote{$\lceil \cdot \rceil$ is the {\it ceiling} operation (i.e., the input is rounded up to the least integer that is greater than or equal to the input).} in $\alpha$ using the Gauss-Legendre \cite{article:GolubWelsch1969} points and weights:
\begin{equation}
\label{eqn:pos_points}
\bigl\{ \alpha_{\ell}, \, \varpi_{\ell} \bigr\}_{\ell=1}^{\lceil\morder/2\rceil},
\end{equation}
which, for example, in the $\morder=4$ case are given by
\begin{equation}
\morder = 4 \quad \Longrightarrow \quad \lceil\morder/2\rceil = 2, \quad
\alpha_1 = -1\bigl/{\sqrt{3}}, \quad
\alpha_2 =  1\bigl/{\sqrt{3}}, \quad
\varpi_1 = \varpi_2 = 1.
\end{equation}
With this quadrature, we can rewrite \eqref{eqn:ave_update_x} as follows:
\begin{equation}
\label{eqn:ave_update_gauss_x}
\begin{split}
Q^{\text{1D}}_{\left(i, \, j, \, s, \, 1, \, k \right)} &=
\frac{\mu_{\left(j,s,k\right)}}{2}
\sum_{\ell=1}^{\lceil\morder/2\rceil} \varpi_{\ell} \,
F^{n}_{\bigl(i-1-\mathcal{I}_{\left(j,s,k\right)}, \, j, \, s\bigr)}\!\Bigl(
1 + \mu_{\left(j,s,k\right)}\!\left(\alpha_{\ell} - 1 \right), \, \eta_k\Bigr)
\\
&+
\frac{1-\mu_{\left(j,s,k\right)}}{2}
\sum_{\ell=1}^{\lceil\morder/2\rceil} \varpi_{\ell} \,
F^{n}_{\bigl(i-\mathcal{I}_{\left(j,s,k\right)}, \, j, \, s\bigr)}\!\Bigl(-\mu_{\left(j,s,k\right)} + \alpha_{\ell}\!\left(1- \mu_{\left(j,s,k\right)}\right), \, \eta_k\Bigr),
\end{split}
\end{equation}
where we note that $0 < \varpi_{\ell} \le 2$ $\forall \ell$ and
$0 \le \mu_{\left(j,s,k\right)} < 1$ $\forall j,s,k$. This leads to the following lemma.

\begin{lemma}[Positivity in the mean for Problem $\mathcal{A}$]
\label{thm:positivity_mean}
The SLDG numerical update for Problem $\mathcal{A}$, as described via equations \eqref{eqn:projected-ic-x}, \eqref{eq:sldg_line_projection}, and \eqref{eq:sldg_x_update}, is positivity-preserving in the sense of
\eqref{eqn:positivity-preserving}, provided that the starting solution, \eqref{eq:sldg_species_expansion},
is positive,
\begin{equation*}
F^{n}_{\left(i,j,s\right)}(\xi,\eta)
=\sum_{\ell=1}^{M_\text{\normalfont basis}} Q^n_{\left(i, j, s, \ell\right)} \,\phi_{\ell}(\xi,\eta) \ge \varepsilon > 0,
\end{equation*}
at all of the following quadrature points:
\begin{align*}
\left(\xi, \, \eta \right) \in \Bigl\{ \Bigl(-1 + \mu_{\left(j,s,k\right)}\!\left(\alpha_{\ell} + 1 \right), \, \eta_k \Bigr) \Bigr\}^{\lceil\morder/2\rceil}_{\ell=1} \, \cup \,
\Bigl\{ \Bigl( \mu_{\left(j,s,k\right)} + \alpha_{\ell}\!\left(1- \mu_{\left(j,s,k\right)}\right), \, \eta_k\Bigr) \Bigr\}^{\lceil\morder/2\rceil}_{\ell=1},
\end{align*}
where $\alpha_{\ell}$ is given by \eqref{eqn:pos_points} and $\mu_{\left(j,s,k\right)}$ is given by \eqref{eq:sldg_x_shift}, for all $i\!=\!1,\ldots,N_x$, $j\!=\!1,\ldots,N_v$, $s\!=\!1,\ldots,\text{S}$, and $k\!=\!1,\ldots,\morder$.
\end{lemma}

\begin{proof}
From the above discussion, we have already established that $Q^{\text{1D}}_{\left(i, \, j, \, s, \, 1, \, k \right)}$ is positive under the assumptions of the lemma (see \eqref{eqn:ave_update_gauss_x}). The final step is to apply definition \eqref{eqn:Q1D_to_Qnew_x} with $\ell\!=\!1$ to obtain:
\begin{equation*}
Q_{\left(i, \, j, \, s, \, 1 \right)}^{n+1}
    =
    \sum_{m=1}^{\morder} \sum_{k=1}^{\morder}
    A^{\text{1D}\rightarrow\text{2D}}_{\left(1, \, m, \, k \right)} \, Q^{\text{1D}}_{\left(i, \, j, \, s, \, m, \, k \right)} = 
    \sum_{k=1}^{\morder} \frac{\omega_k}{2} \, Q^{\text{1D}}_{\left(i, \, j, \, s, \, 1, \, k \right)} \ge \varepsilon > 0.
\end{equation*}
\end{proof}

\subsubsection{Positivity in Problem $\mathcal{B}$}
Positivity-preservation in Problem $\mathcal{B}$ is nearly identical to what was described above for Problem $\mathcal{A}$. We therefore just summarize the main finding in the lemma below.

\begin{lemma}[Positivity in the mean for Problem $\mathcal{B}$]
\label{thm:positivity_mean_v}
The SLDG numerical update for Problem $\mathcal{B}$, as described via equations \eqref{eqn:projected-ic-v}, \eqref{eq:sldg_line_projection_v}, and \eqref{eq:sldg_v_update}, is positivity-preserving in the sense of
\eqref{eqn:positivity-preserving}, provided that the starting solution, \eqref{eq:sldg_species_expansion},
is positive,
\begin{equation*}
F^{n}_{\left(i,j,s\right)}(\xi,\eta)
=\sum_{\ell=1}^{M_\text{\normalfont basis}} Q^n_{\left(i, j, s, \ell\right)} \,\phi_{\ell}(\xi,\eta) \ge \varepsilon > 0,
\end{equation*}
at all of the following quadrature points:
\begin{align*}
\left(\xi, \, \eta \right) \in \Bigl\{ \Bigl(\xi_k, \, -1 + \nu_{\left(i,s,k\right)}\!\left(\alpha_{\ell} + 1 \right) \Bigr) \Bigr\}^{\lceil\morder/2\rceil}_{\ell=1} \, \cup \,
\Bigl\{ \Bigl( \xi_k, \, \nu_{\left(i,s,k\right)} + \alpha_{\ell}\!\left(1- \nu_{\left(i,s,k\right)}\right)\Bigr) \Bigr\}^{\lceil\morder/2\rceil}_{\ell=1},
\end{align*}
where $\alpha_{\ell}$ is given by \eqref{eqn:pos_points} and $\nu_{\left(i,s,k\right)}$ is given by \eqref{eq:sldg_v_shift}, for all $i\!=\!1,\ldots,N_x$, $j\!=\!1,\ldots,N_v$, $s\!=\!1,\ldots,\text{S}$, and $k\!=\!1,\ldots,\morder$.
\end{lemma}

\subsubsection{Enforcing positivity on a collection of points}
In the two previous subsections, we established the conditions needed to rigorously guarantee positivity during an SLDG sub-step (either Problem $\mathcal{A}$ or $\mathcal{B}$). The one missing ingredient is the actual positivity enforcement at the appropriate quadrature points. Our approach is to use the widely celebrated method of Zhang and Shu \cite{ZhangShu2010}, which we briefly describe in this subsection.

At the beginning of either SLDG sub-step (either Problem $\mathcal{A}$ or $\mathcal{B}$), the solution is given by \eqref{eq:sldg_species_expansion}. We introduce a {\it damping parameter}, $\theta$, that pre-multiplies the ``high-order'' terms in \eqref{eq:sldg_species_expansion}:
\begin{equation}
\label{eqn:damped_pdf}
F^{n}_{\left(i,j,s\right)}(\xi,\eta)
=Q^n_{\left(i, j, s, 1\right)} + \theta_{\left(i, j, s \right)} \, \sum_{\ell=2}^{M_\text{basis}} Q^n_{\left(i, j, s, \ell\right)} \,\phi_{\ell}(\xi,\eta), 
\end{equation}
where $0 \le \theta_{\left(i, j, s \right)} \le 1$ and we assume a positive element average: $Q^n_{\left(i, j, s, 1\right)} \ge \varepsilon > 0$. Consider a collection of $\mpos$ points in $(\xi,\eta) \in [-1,1] \times [-1,1]$ where we want to impose that $F^{n}_{\left(i,j,s\right)} \ge \varepsilon > 0$:
\begin{equation}
\label{eqn:pos_points_chi_def}
\chi_{\mpos} = \biggl\{ \Bigl(\xi_p, \, \eta_p \Bigr) \biggr\}_{p=1}^{\mpos} \quad \text{where} \quad
\Bigl(\xi_p, \, \eta_p \Bigr) \in \bigl[-1,1\bigr] \times \bigl[-1,1\bigr] \quad \forall p.
\end{equation}
For each $\left(i, j, s \right)$, the goal is to find the largest possible value of $\theta$ (between 0 and 1) that guarantees positivity at all points in $\chi_{\mpos}$. We note three important mathematical facts relevant here:
\begin{enumerate}
\item There exists some $\theta\in[0,1]$ for which \eqref{eqn:damped_pdf} is positive at all points in $\chi_{\mpos}$;
\item If \eqref{eqn:damped_pdf} is positive at all points in $\chi_{\mpos}$ for some $\theta^{\star}\in[0,1]$, then it will remain positive for any $0 \le \theta \le \theta^{\star}$;
\item The damping will not affect the cell average (i.e., the first Legendre coefficient), so it will not alter the mass conservation properties of the overall scheme.
\end{enumerate}

We start with positivity for Problem $\mathcal{A}$,  and define the following points where we need to enforce positivity:
\begin{equation}
\begin{gathered}
\chi_{\mpos} :=  \biggl\{ \Bigl(-1 + \mu_{\left(j,s,k\right)}\!\left(\alpha_{\ell} + 1 \right), \, \eta_k \Bigr) \quad \text{and} \quad
 \Bigl( \mu_{\left(j,s,k\right)} + \alpha_{\ell}\!\left(1- \mu_{\left(j,s,k\right)}\right), \, \eta_k\Bigr)\!: \\ \ell=1,\ldots,\lceil\morder/2\rceil \quad \text{and} \quad \, k=1,\ldots\morder \biggr\}, \quad
 \mpos = 2 \morder \cdot \lceil\morder/2\rceil,
 \end{gathered}
\end{equation}
where $\alpha_{\ell}$ is given by \eqref{eqn:pos_points} and $\mu_{\left(j,s,k\right)}$ is given by \eqref{eq:sldg_x_shift}. For each $(i,j,s)$, we find the minimum value of \eqref{eq:sldg_species_expansion} across all these positivity points:
\begin{equation}
    F^{\text{min}}_{(i,j,s)} := \min_{\left(\xi, \, \eta\right) \in \chi_{\mpos}}
    F^n_{(i,j,s)}\!\left(\xi,\eta\right).
\end{equation}
From this minimum value, we define the damping parameter:
\begin{equation}
\label{eqn:theta_val}
\theta_{(i,j,s)} = \min\left\{ 1, \, \frac{Q^n_{\left(i, j, s, 1\right)}}{Q^n_{\left(i, j, s, 1\right)} - F^{\text{min}}_{(i,j,s)}} \right\},
\end{equation}
and use it to damp the high-order corrections via \eqref{eqn:damped_pdf}.

A similar approach is applied in Problem $\mathcal{B}$, where we again define the set of points where we need to enforce positivity:
\begin{equation}
\begin{gathered}
\chi_{\mpos} := \biggl\{ \Bigl(\xi_k, \, -1 + \nu_{\left(i,s,k\right)}\!\left(\alpha_{\ell} + 1 \right) \Bigr) \quad \text{and} \quad
 \Bigl( \xi_k, \, \mu_{\left(i,s,k\right)} + \alpha_{\ell}\!\left(1- \nu_{\left(i,s,k\right)}\right) \Bigr)\!: \\ \ell=1,\ldots,\lceil\morder/2\rceil \quad \text{and} \quad \, k=1,\ldots\morder \biggr\},
  \quad \mpos = 2 \morder \cdot \lceil\morder/2\rceil,
 \end{gathered}
\end{equation}
where $\alpha_{\ell}$ is given by \eqref{eqn:pos_points} and $\nu_{\left(i,s,k\right)}$ is given by \eqref{eq:sldg_v_shift}. For each $(i,j,s)$, we find the minimum value of \eqref{eq:sldg_species_expansion} across all these positivity points:
\begin{equation}
    F^{\text{min}}_{(i,j,s)} := \min_{\left(\xi, \, \eta\right) \in \chi_{\mpos}}
    F^n_{(i,j,s)}\!\left(\xi,\eta\right).
\end{equation}
Using this minimum value, $\theta$ is again computed via \eqref{eqn:theta_val}, and then used to damp the high-order corrections via \eqref{eqn:damped_pdf}.

\subsubsection{Positive kinetic energy}
All of the discussion in \S\ref{subsec:pp_limiter_multi} has so far focused on guaranteeing positivity of the element average of the particle density function (PDF). However, at the end of each time step, we want to guarantee that the distribution function produces a non-negative kinetic energy. To achieve this, we define another set of points where we will enforce positivity:
\begin{equation}
\chi_{\mpos} = \biggl\{ \Bigl(\xi_{a}, \, \xi_{b} \Bigr) \biggr\}_{a,b=1}^{\morder}, \quad
\mpos = M_{\text{o}}^2,
\end{equation}
where $\xi_{a}$ are the one-dimensional $\morder$-point Gauss-Legendre points. For each $(i,j,s)$, we find the minimum value of \eqref{eq:sldg_species_expansion} across all these positivity points:
\begin{equation}
    F^{\text{min}}_{(i,j,s)} = \min_{\left(\xi, \, \eta\right) \in \chi_{\mpos}}
    F^n_{(i,j,s)}\!\left(\xi,\eta\right).
\end{equation}
Using this minimum value, $\theta$ is again computed via \eqref{eqn:theta_val}, and then used to damp the high-order corrections via \eqref{eqn:damped_pdf}.
Kinetic energy is then guaranteed to be positive, since the integrals in \eqref{eqn:dg-moments-expansion} and \eqref{eqn:dg-moments-expansion-integrals} can all be computed exactly with quadrature based on the points in $\chi_{\mpos}$. 

\begin{algorithm}[!t]
\caption{Second-order operator splitting algorithm for the Vlasov-Amp\`ere system}
\label{alg:va-strang-ho}
    \begin{algorithmic}[1]
   \State $\frac{\Delta t}{2}$ step on Problem $\mathcal{A}$: \quad
$f_{s,t} + v f_{s,x} = 0 \quad  \forall s$: \quad
$f^{n}_{s} \rightarrow f_{s}^{(1)}$
 \State From $f_{s}^{(1)}$ compute moments $J^{(1)}$ and $\omega^{(1)}$, then update electric field:

\vspace{-9mm}

\begin{equation*}
 E^{\,n+1}
= E^n \cos\Bigl(\Delta t \, \omega^{(1)} \Bigr)
-\frac{J^{(1)}}{\omega^{(1)}}\sin\Bigl(\Delta t \, \omega^{(1)}  \Bigr), \quad
    E^{\star} 
    = E^n \left( \frac{\sin\bigl(\Delta t \, \omega^{(1)} \bigr)}{\Delta t \, \omega^{(1)}} \right)
    - \frac{J^{(1)}}{\omega^{(1)}} \left(\frac{1-\cos\bigl(\Delta t \, \omega^{(1)}\bigr)}{\Delta t \, \omega^{(1)}}\right)
\end{equation*}

\vspace{-4mm}

\State \(\Delta t\) step on Problem $\mathcal{B}$: \quad
$f_{s,t} + \frac{q_s}{m_s}  E^{\star} f_{s,v} = 0 \quad  \forall s$: \quad
$f_{s}^{(1)} \rightarrow f_{s}^{(2)}$
\State $\frac{\Delta t}{2}$ step on Problem $\mathcal{A}$: \quad
$f_{s,t} + v f_{s,x} = 0 \quad  \forall s$: \quad
$f^{(2)}_{s} \rightarrow f_{s}^{n+1}$
 \end{algorithmic}
\end{algorithm}

\subsection{Full operator split SLDG method for Vlasov-Amp\`ere}
\label{subsec:strang_algorithm}
In the previous subsections, we have treated the SLDG solutions of Problems $\mathcal{A}$ and $\mathcal{B}$ as independent problems. To achieve a full algorithm for solving the Vlasov-Amp\`ere, we need to take the independent Problem $\mathcal{A}$ and $\mathcal{B}$ solvers and place them inside an operator splitting time-stepping method. We briefly explain how to do this in this subsection.

\subsubsection{Second-order time splitting}
The classical operator splitting method for Vlasov, as pioneered by Cheng and Knorr \cite{ChengKnorr1976}, is the second-order Strang splitting method \cite{Strang1968}. We summarize a single time step of our proposed method in Algorithm~\ref{alg:va-strang-ho}. This approach combines all of the previously described algorithmic elements, including the harmonic-oscillator field update, which exactly evolves the electric field on the acceleration branch and supplies the time-averaged electric field for the velocity translation.

\begin{table}[!h]
\begin{tabular}{|c|r||c|r|}
        \hline
        $\zeta_{1}$  &  {\tt  0.0829844064174052}   & $\zeta_{2}$  &  {\tt  0.2452989571842710} \\ \hline
        $\zeta_{3}$  &  {\tt  0.3963098014983680}   & $\zeta_{4}$  &  {\tt  0.6048726657110800} \\ \hline
        $\zeta_{5}$  &  {\tt -0.0390563049223486}   & $\zeta_{6}$  &  {\tt -0.3501716228953510} \\ \hline
        $\zeta_{7}$  &  {\tt  0.1195241940131508}   & $\zeta_{8}$  &  {\tt -0.3501716228953510} \\ \hline
        $\zeta_{9}$  &  {\tt -0.0390563049223486}   & $\zeta_{10}$ &  {\tt  0.6048726657110800} \\ \hline
        $\zeta_{11}$ &  {\tt  0.3963098014983680}   & $\zeta_{12}$ &  {\tt  0.2452989571842710} \\ \hline
        $\zeta_{13}$ &  {\tt  0.0829844064174052}   & & \\
        \hline
\end{tabular}
\caption{Time-step coefficients for the fourth-order operator splitting approach of Blanes and Moan \cite{BlanesMoan2002}. \label{table:blanes-moan}}
\end{table}

\subsubsection{Fourth-order time splitting}
High-order accuracy in time is obtained via the optimized Runge–Kutta–Nystr\"om coefficients from Blanes and Moan \cite{BlanesMoan2002}, which were also used by Crouseilles, Faou, and Mehrenberger \cite{article:Crouseilles2011} and Güçlü, Christlieb, and Hitchon \cite{article:GucluChristliebHitchon2014} in their Vlasov-Poisson solvers. This approach differs from the approach of Rossmanith and Seal \cite{RossmanithSeal2011}, which used the fourth-order Yoshida splitting \cite{article:ForestRuth1990,article:Yoshida1990,article:Yoshida1993} (which is referred to as the ``triple jump'' method in \cite{article:Crouseilles2011}). The advantage of the fourth-order Runge–Kutta–Nystr\"om approach of Blanes and Moan \cite{BlanesMoan2002} over fourth-order Yoshida splitting \cite{article:Yoshida1990} is that all of the intermediate stages are strictly within the time interval $\left[t^n, t^n + \Delta t\right]$; this property was shown by Crouseilles et al. \cite{article:Crouseilles2011} to produce significantly improved total energy conservation in their Vlasov-Poisson solver, and that is why we have chosen it here.

We summarize a single time step of our proposed method in Algorithm~\ref{alg:split-4-va-ho}. We again highlight all the algorithmic elements, including the harmonic-oscillator field update, which exactly evolves the electric field on the acceleration branch and supplies the time-averaged electric field for the velocity translation. Note that the full time step, $\Delta t$, is divided into 13 stages -- 7 corresponding to SLDG solutions of Problem $\mathcal{A}$ and 6 corresponding to SLDG solutions of Problem $\mathcal{B}$. The step size for each stage is given by $\zeta_k \Delta t$, such that odd $k$ correspond to solutions of Problem $\mathcal{A}$ and even $k$ correspond to solutions of Problem $\mathcal{B}$. The time coefficients, $\zeta_k$, of the Runge–Kutta–Nystr\"om splitting are given in Table \ref{table:blanes-moan}; the left two columns are the time coefficients used in each of the 7 stages where Problem $\mathcal{A}$ is solved, while the right two columns are the time coefficients used in each of the 6 stages where Problem $\mathcal{B}$ is solved. 

\begin{algorithm}[!t]
\caption{Fourth-order operator splitting algorithm for the Vlasov-Amp\`ere system}
\label{alg:split-4-va-ho}
\begin{algorithmic}[1]
\State ({\bf Stage $1$}): \, $\zeta_{1}\Delta t$ step on Problem $\mathcal{A}$: \quad
$f_{s,t} + v f_{s,x} = 0 \quad  \forall s$: \quad
$f^{n}_{s} \rightarrow f^{(1)}_{s}$, \quad $E^{(1)}=E^n$
\For{$m=1,\ldots,6$}
\State From $f^{(2m-1)}_{s}$ compute moments $J$ and $\omega$, then update electric field: 

\vspace{-9mm}

\begin{equation*}
 E^{(m+1)}
= E^{(m)} \cos\Bigl(\zeta_{2m}  \Delta t \, \omega \Bigr)
-\frac{J}{\omega}\sin\Bigl(\zeta_{2m} \Delta t \, \omega  \Bigr), \quad
    E^{\star} 
    = E^{(m)} \left( \frac{\sin\bigl(\zeta_{2m} \Delta t \, \omega \bigr)}{\zeta_{2m} \Delta t \, \omega} \right)
    - \frac{J}{\omega} \left(\frac{1-\cos\bigl(\zeta_{2m}  \Delta t \, \omega\bigr)}{\zeta_{2m}  \Delta t \, \omega}\right)
\end{equation*}

\vspace{-4mm}

\State ({\bf Stage $2m$}): \, $\zeta_{2m} \Delta t$ step on Problem $\mathcal{B}$: \quad
$f_{s,t} + \frac{q_s}{m_s} E^{\star} f_{s,v} = 0 \quad  \forall s$: \quad
$f_{s}^{(2m-1)} \rightarrow f_{s}^{(2m)}$
\State ({\bf Stage $2m+1$}): \, $\zeta_{2m+1}\Delta t$ step on Problem $\mathcal{A}$: \quad
$f_{s,t} + v f_{s,x} = 0 \quad  \forall s:$ \quad
$f_{s}^{(2m)} \rightarrow f_{s}^{(2m+1)}$
\EndFor
\State Set \, $f^{n+1}_s \leftarrow f^{(13)}_s$ $\forall s$ \, and \, $E^{n+1} \leftarrow E^{(7)}$
\end{algorithmic}
\end{algorithm}

\section{Numerical results}\label{sec:results}

In this section, we validate the accuracy and structure-preserving properties of the proposed semi-Lagrangian DG scheme.  The numerical tests cover both the single and two-species 1D1V Vlasov-Amp\`ere system. For each case, we first use the method of manufactured solutions (MMS) to verify the expected order of accuracy. We then consider standard kinetic benchmark tests from the literature. Unless otherwise stated, periodic boundary conditions are imposed in \(x\), and the velocity domain is truncated to a finite interval chosen large enough that the distribution functions remain negligible near the velocity boundaries over the reported simulation time. For each simulation, time steps are chosen to maintain a constant Courant-Friedrichs-Lewy (CFL) number \cite{article:CFL1928}:
\begin{flalign}
 \mathrm{CFL}
:= \Delta t\;
\max
\left\{
\frac{\max\limits_{s\in\{e,i\}}\left(\max\left\{\left|v^{(s)}_{\text{min}}\right|, \left|v^{(s)}_{\text{max}}\right|\right\}\right)}{\Delta x},\;
\max\limits_{s\in\{e,i\}}\left|\frac{q_s}{m_s}\right| \cdot \frac{\max\limits_{x\in\Omega_x} \bigl|E(x)\bigr|}{\, \Delta v}
\right\}. 
\end{flalign}
   
The proposed numerical methods are implemented in the open-source DoGPack code \cite{dogpack}. The main code can be found by first cloning the Git repository and then setting the environment variable via the command:
\begin{tcolorbox}
\begin{verbatim}
make install
\end{verbatim}
\end{tcolorbox}
All of the examples presented in the upcoming subsections can be found in the following three subdirectories:
\begin{tcolorbox}
\begin{verbatim}
cd $DOGPACK/apps/2d/one_species_va_1d1v/
cd $DOGPACK/apps/2d/two_species_va_1d1v/
cd $DOGPACK/apps/2d/keen_wave/
\end{verbatim}
\end{tcolorbox}

\subsection{Single-species model}\label{subsec:single_species_limit}

We first consider the 1D1V single-species Vlasov-Amp\`ere system \eqref{eq:va_single_vlasov} and \eqref{eq:va_single_ampere}. In this setting, only the electron distribution is evolved, while the ions are modeled as a fixed, spatially uniform neutralizing background.

\subsubsection{Accuracy test: method of manufactured solution}
\label{subsec:mms_single}
We first verify accuracy in the single-species limit using a forced 1D1V Vlasov-Amp\`ere system, in a setup similar to Rossmanith and Seal \cite{RossmanithSeal2011} and Seal \cite{thesis:Seal2012}:
\begin{equation}
\label{eq:mms_vlasov_single}
    f_{,t}+v f_{,x}-E f_{,v}=\psi(t,x,v), \quad
    E_{,x}=\sqrt{\pi}-\rho, \quad
    E_{,t}-\rho u=S_1(t,x),
\end{equation}
where
\begin{equation}
\rho = \int_{-\infty}^{\infty} f \, dv \quad \text{and} \quad 
\rho u = \int_{-\infty}^{\infty} v f \, dv.
\end{equation}
The computation is performed over the domain:
\begin{equation}
(t,x,v)\in[0,1]\times[-\pi,\pi]\times[-\pi,\pi],
\end{equation}
with periodic boundary conditions in $x$ and extrapolation boundary conditions in $v$. The exact solution is taken as
\begin{equation}
\label{eq:mms_exact_f_single}
    f(t,x,v)
    =
    e^{-\frac{(4v-1)^2}{4}}
    \bigl(2-\cos(2x-2\pi t)\bigr), \quad 
    E(t,x)
    =
    \frac{\sqrt{\pi}}{4}\sin(2x-2\pi t).
\end{equation}
Substitution of \eqref{eq:mms_exact_f_single} into the forced Vlasov-Amp\`ere equation \eqref{eq:mms_vlasov_single} yields the following artificial source terms:
\begin{align}
\label{eq:mms_source_psi_single}
\psi(t,x,v)
&=
\frac12 \sin\left(2x-2\pi t\right) \,
e^{-\frac{\left(4v-1\right)^2}{4}}
\left[ 4(v-\pi) + \sqrt{\pi}\left(4v-1\right) \, \Bigl(2-\cos\bigl(2x-2\pi t\bigr)\Bigr)
\right], \\
\label{eq:mms_source_S1_single}
    S_1(t,x)
    &=
    \frac{\sqrt{\pi}}{8}
    \Bigl[\left(1-4\pi\right) \cos\left(2x-2\pi t\right) - 2 \Bigr].
\end{align}
We split the forced problem into the following two sub-problems:
\begin{align}
\label{eq:mms_problem_A_single}
\textnormal{Problem $\mathcal{A}$:}& \quad
    f_{,t}+v f_{,x}=\psi(t,x,v), \quad E_{,t}=0,
\\
\label{eq:mms_problem_B_single}
\textnormal{Problem $\mathcal{B}$:}& \quad
    f_{,t}-E f_{,v}=0, \quad E_{,t}- \rho u=S_1(t,x).
\end{align}
Problem \(\mathcal A\) can still be solved using the method of characteristics, but in this setting, the distribution is no longer constant along each characteristic; instead, we need to compute the following integral:
\begin{equation}
\label{eq:mms_SL_update_A_single}
    f\left(t^{n+1},x,v\right)
    =
    f\left(t^n,x-v\Delta t,v\right)
    +
    \int_{t^n}^{t^{n+1}}
    \psi\Bigl(\tau,x-v\left(t^{n+1}-\tau\right),v\Bigr)\,d\tau .
\end{equation}
For the source in \eqref{eq:mms_source_psi_single}, this integral is evaluated in closed form;
otherwise a consistent high-order quadrature may be used. Problem \(\mathcal B\) is advanced by
the same semi-Lagrangian velocity tracing used in the unforced scheme, but with the field equation
modified by the source \(S_1\). 

The error at some fixed time, $t$, is measured using the relative $L_2$-norm error:
\begin{equation}
\text{relative} \, L^2 \, \text{error} :=
\sqrt{\frac{\displaystyle\int_{v^{(s)}_{\text{min}}}^{v^{(s)}_{\text{max}}}\!\!\!\int_{x_{\text{min}}}^{x_{\text{max}}} \Bigl( f^{\left(\Delta x, \Delta v\right)}_s\left(t,x,v\right) - f_s\left(t,x,v\right) \Bigr)^2 \, d\!x \, d\!v}{\displaystyle\int_{v^{(s)}_{\text{min}}}^{v^{(s)}_{\text{max}}}\!\!\!\int_{x_{\text{min}}}^{x_{\text{max}}} \Bigl( f_s\left(t,x,v\right) \Bigr)^2 \, d\!x \, d\!v}},
\end{equation}
where $f^{\left(\Delta x, \Delta v\right)}_s$ is the numerical solution and $f_s$ is the exact solution.
In practice, we do not compute these integrals exactly; instead, we use the following approximation to the relative $L^2$ error:
\begin{equation}
\label{eqn:error}
\text{error}\!\left(s,\morder,N_x,N_v\right) := 
\sqrt{\frac{\displaystyle\sum_{i=1}^{N_x} \displaystyle\sum_{j=1}^{N_v} \left[\sum_{\ell=1}^{\morder(\morder+1)/2}\Bigl(Q_{\left(i,j,\ell,s\right)}-Q^{\text{exact}}_{\left(i,j,\ell,s\right)}\Bigr)^2+\sum_{\ell=\morder(\morder+1)/2+1}^{(\morder+1)(\morder+2)/2} \Bigl(Q^{\text{exact}}_{\left(i,j,\ell,s\right)}\Bigr)^2\right]}{\displaystyle\displaystyle\sum_{i=1}^{N_x} \displaystyle\sum_{j=1}^{N_v}\left[\sum_{\ell=1}^{(\morder+1)(\morder+2)/2}\Bigl(Q^{\text{exact}}_{\left(i,j,\ell,s\right)}\Bigr)^2\right]}},
\end{equation}
where $\morder$ is the theoretical order of accuracy of the numerical method (see \eqref{eqn:theoretical_order_accuracy}) and $Q^{\text{exact}}$ are the Legendre coefficients computed from projecting the exact PDF onto the finite element mesh. Note that the numerical solution has $\morder(\morder+1)/2$ Legendre coefficients, while the exact solution is represented with $(\morder+1)(\morder+2)/2$ Legendre coefficients that are computed via Gauss-Legendre quadrature with $\left(\morder+1\right)^2$ points. In other words, the exact solution is approximated on the finite element mesh with higher order accuracy than the numerical solution (e.g., see \cite{article:JohnsonRossmanithVaughan2023,article:RossmanithVaughan2026} where a similar approach is used). 

A convergence study for this problem is shown in Table~\ref{va:mms-single species}, which confirms that the scheme attains its expected order accuracy. In these simulations, the final time is $t\!=\!1$, and the CFL number is fixed at $\text{CFL}\!=\!2$. We note that the errors decrease at the correct rates for both the second-order splitting with the $\morder\!=\!2$ SLDG method (Columns 2 and 3) and for the fourth-order splitting with the $\morder\!=\!4$ SLDG method (Columns 4 and 5). The convergence results reported in this table were generated using the DoGPack code~\cite{dogpack}. The results presented here are consistent with those of Rossmanith and Seal \cite{RossmanithSeal2011} and Seal \cite{thesis:Seal2012}. Finally, we note that Table~\ref{va:mms-single species} can be generated from the code via the following commands:
\begin{tcolorbox}
\begin{verbatim}
cd $DOGPACK/apps/2d/one_species_va_1d1v/manufactured_soln_convergence_test; 
make errortable;
\end{verbatim}
\end{tcolorbox}

\begingroup
\setlength{\tabcolsep}{10pt} %
\renewcommand{\arraystretch}{1.5} %
\begin{table}[!th]
\begin{center}
\begin{tabular}{|c||c|c||c|c|}
    \hline
    \textbf{$N_{x} \times N_{v}$} & \textbf{error: $2^{\text{nd}}$ order split} & $\log_{2}(\text{ratio})$ &
                    \textbf{error: $4^{\text{th}}$ order split} & $\log_{2}(\text{ratio})$ \\ \hline\hline
    $10\times10$   & $1.5922\times10^{-1}$  & --      & $1.0252\times10^{-2}$ & --      \\ \hline
    $20\times20$   & $3.4915\times10^{-2}$  & $2.189$ & $6.9058\times10^{-4}$ & $3.892$ \\ \hline
    $40\times40$   & $8.4027\times10^{-3}$  & $2.055$ & $4.2036\times10^{-5}$ & $4.038$ \\ \hline
    $80\times80$   & $2.1312\times10^{-3}$  & $1.979$ & $2.4912\times10^{-6}$ & $4.077$ \\ \hline
    $160\times160$ & $5.4175\times10^{-4}$  & $1.976$ & $1.4329\times10^{-7}$ & $4.120$ \\ \hline
    $320\times320$ & $1.3740\times10^{-4}$  & $1.979$ & $8.6479\times10^{-9}$ & $4.050$ \\ \hline
\end{tabular}
\end{center}
\caption{(\S\ref{subsec:mms_single}: Single-Species Manufactured Solution) Manufactured solution convergence study for the single-species semi-Lagrangian DG method. Reported in this table are the relative $L^2$-errors at increasing mesh resolutions along with the log of the error ratios for both the second-order splitting with $\morder\!=\!2$ SLDG in phase space and fourth-order splitting with $\morder\!=\!4$ SLDG in phase space. Note that $\log_2$ of the error ratio gives an estimate for the order of convergence, since each row in the table represents a doubling of mesh resolution compared to the previous row. All runs are computed to a final time of $t=1$ with $\text{CFL}\!=\!2$ over the domain $(x,v)\in[-\pi,\pi]\times[-\pi,\pi]$ with periodic boundary conditions in $x$.}
\label{va:mms-single species}
\end{table}
\endgroup

\subsubsection{Two-stream instability}
\label{subsec:twostream}
Next we consider the classical two-stream instability in the single-species limit for the 1D1V
Vlasov-Amp\`ere system, which has been considered numerous times in the literature (e.g., see 
\cite{BanksHittinger2010,ChengKnorr1976,FilbetSonnendrucker2003,HeathGambaMorrisonMichler2012,qiu2010,QiuShu2011}). This benchmark tests the ability of the scheme to resolve nonlinear
phase space filamentation. The initial distribution and electric field are
\begin{equation}
\label{eq:initial_dist_two_stream}
    f(t\!=\!0,x,v)
    =
    \frac{v^2}{\sqrt{8\pi}}
    \biggl(2-\cos\left(\frac{x}{2}\right)\biggr)e^{-\frac{v^2}{2}},
    \quad E(t\!=\!0,x)=\sin\left(\frac{x}{2}\right), \quad
    \rho_0 = 1, \quad J_0 = 0.
\end{equation}
The computation is performed on $(x,v)\in[-2\pi,2\pi]\times[-2\pi,2\pi]$,
with periodic boundary conditions in $x$, the solution is advanced to $t\!=\!60$, 
$N_x\times N_v\!=\!128\times128$, and $\mathrm{CFL}\!=\!5$.

Figure~\ref{fig:two-stream-dist} shows the phase space distribution, $f\left(t\!=\!45,x,v\right)$, in Panel (a) and a slice of this solution, $f\left(t\!=\!45,x\!=\!0,v\right)$, in Panel (b). These results are obtained using fourth-order splitting with the $\morder\!=\!4$ SLDG method. The fourth-order scheme sharply resolves the fine-scale filamentation in the distribution function. Panel (b) illustrates how the positivity-preserving limiter eliminates the undershoots and produces a non-negative distribution. Panels (c) and (d) show the time history of $\|E(t,\cdot)\|_{2}$ and particle number drift over the simulation time, $t\in[0,60]$, respectively. These images can be generated using the DoGPack code~\cite{dogpack} by going to the code directory and typing:

\begin{tcolorbox}
\begin{verbatim}
cd $DOGPACK/apps/2d/one_species_va_1d1v/two_stream_instability/; make;
dog.exe; make plotfull; make plotefield; make plotcon;
\end{verbatim}
\end{tcolorbox}

\begin{figure}[!th]
\begin{center}
\begin{tabular}{cc}
   (a)\includegraphics[width=0.44\textwidth]{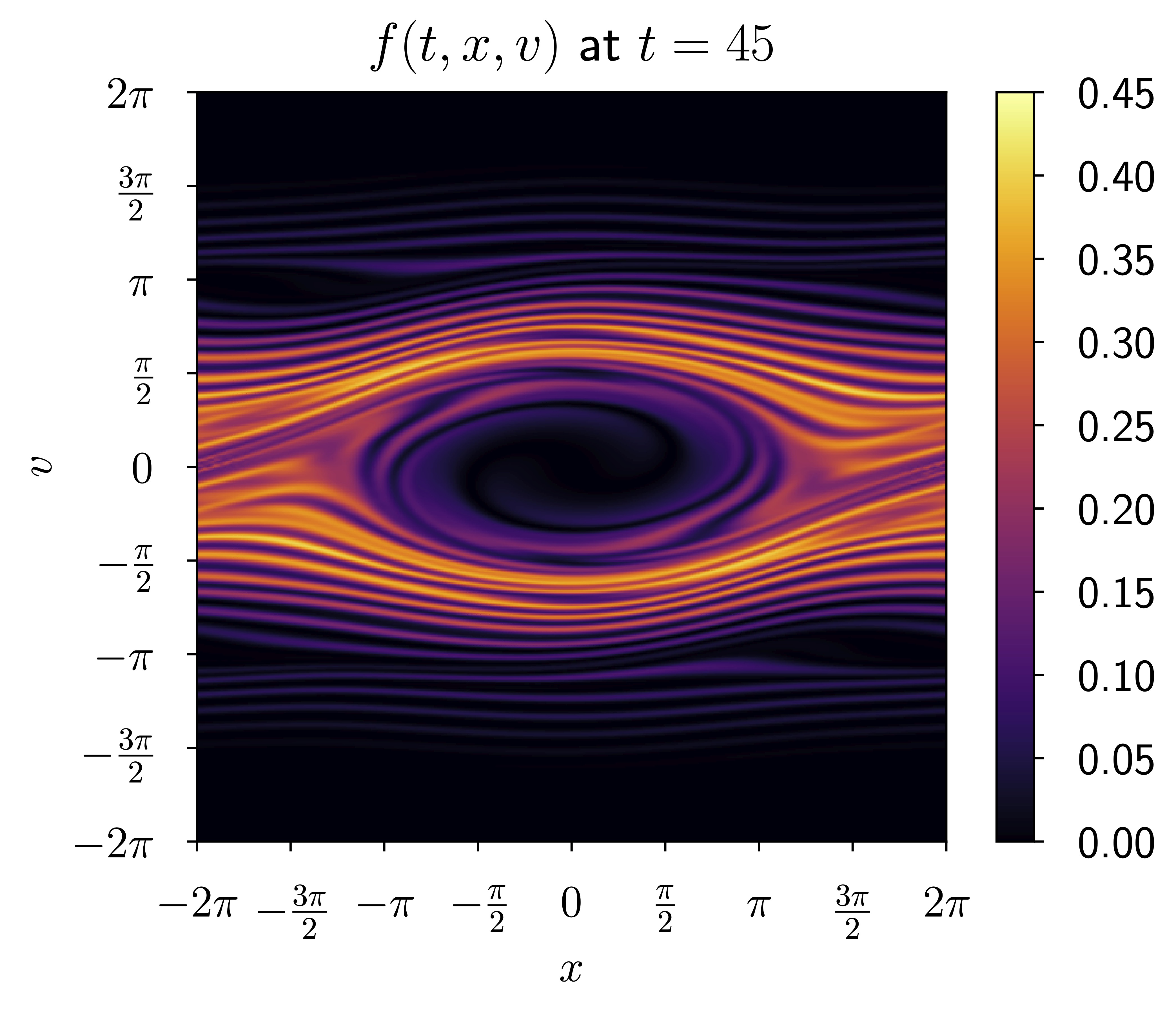} &
   (b)\includegraphics[width=0.44\textwidth]{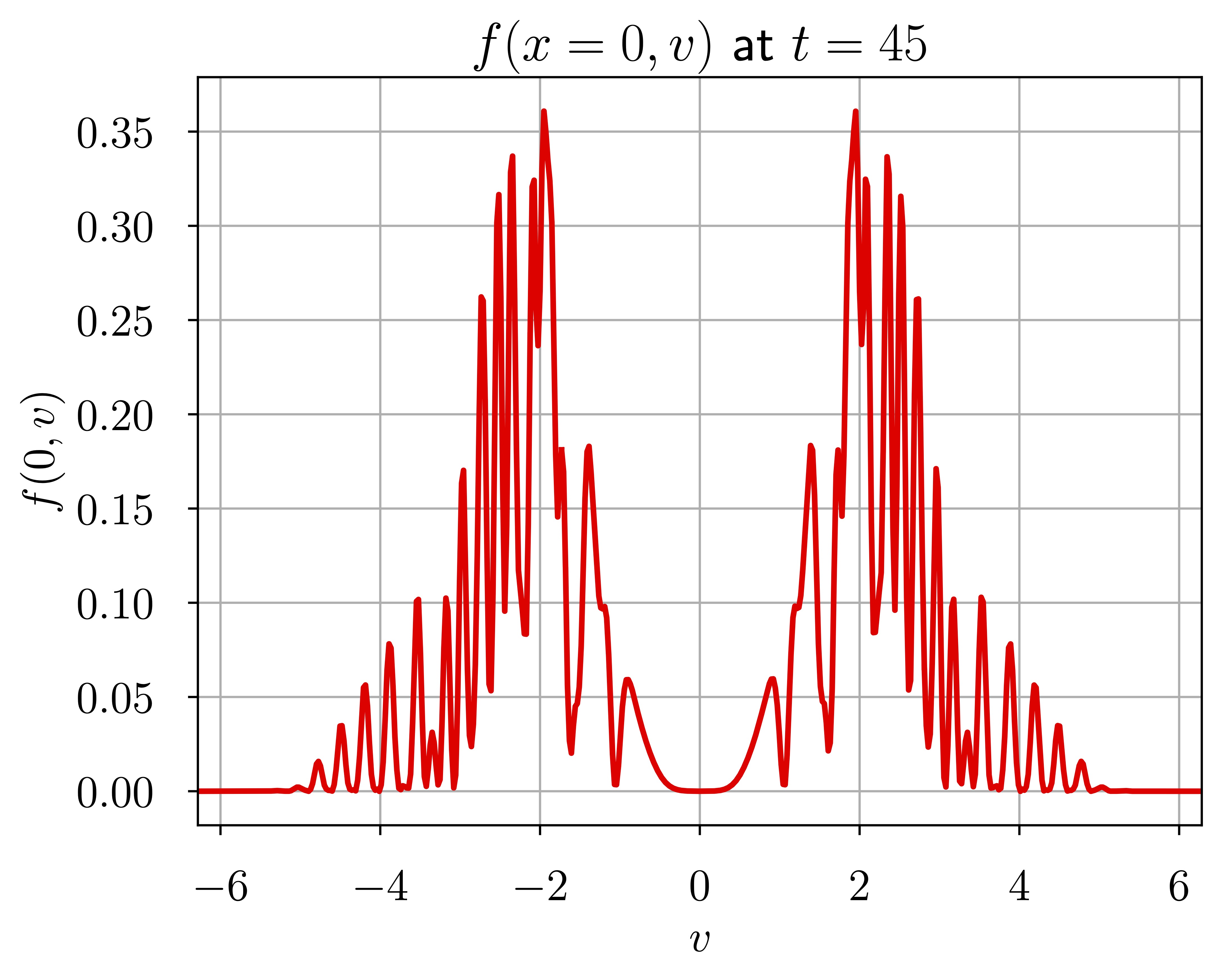} \\
   (c)\includegraphics[width=0.44\textwidth]{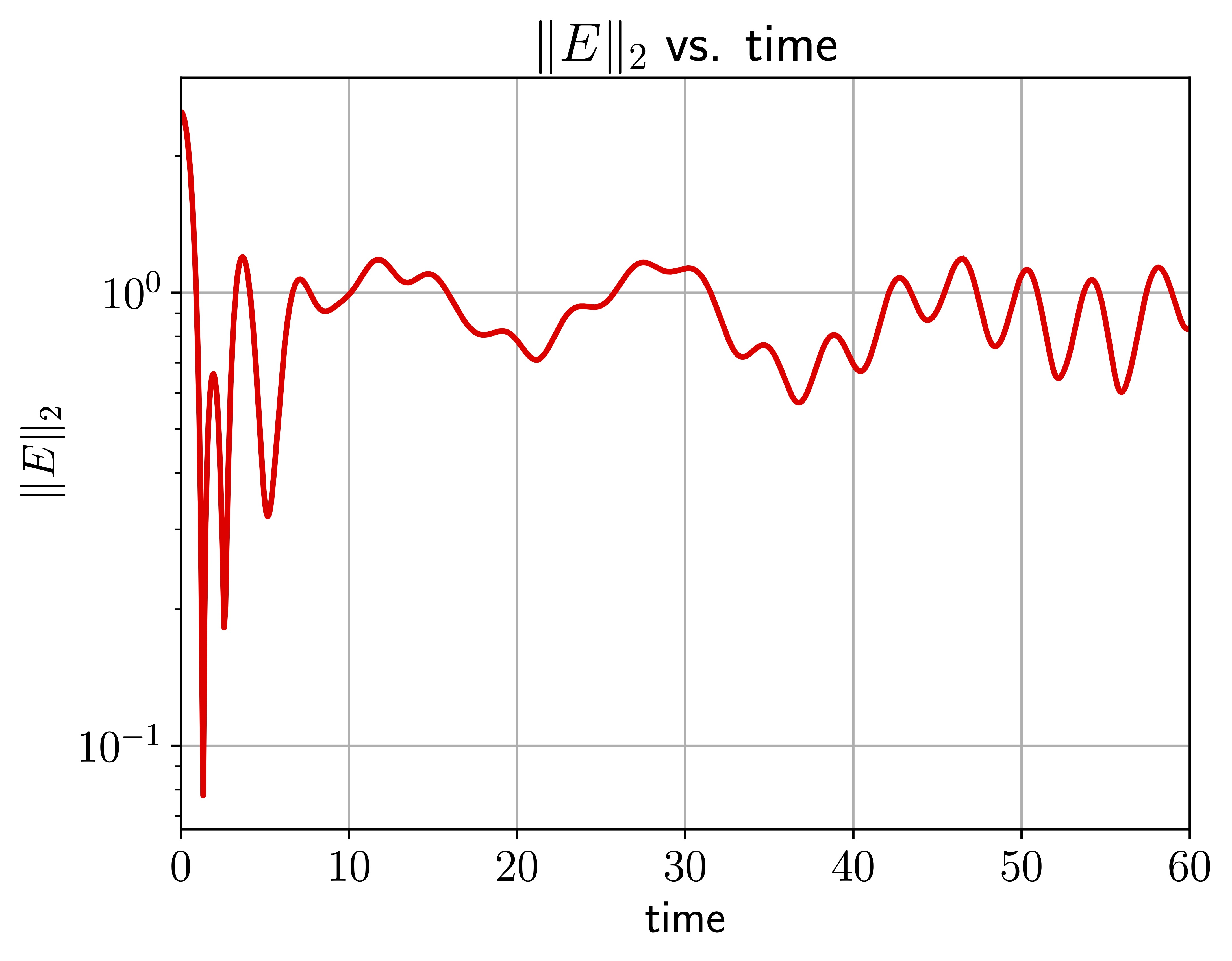} &
   (d)\includegraphics[width=0.44\textwidth]{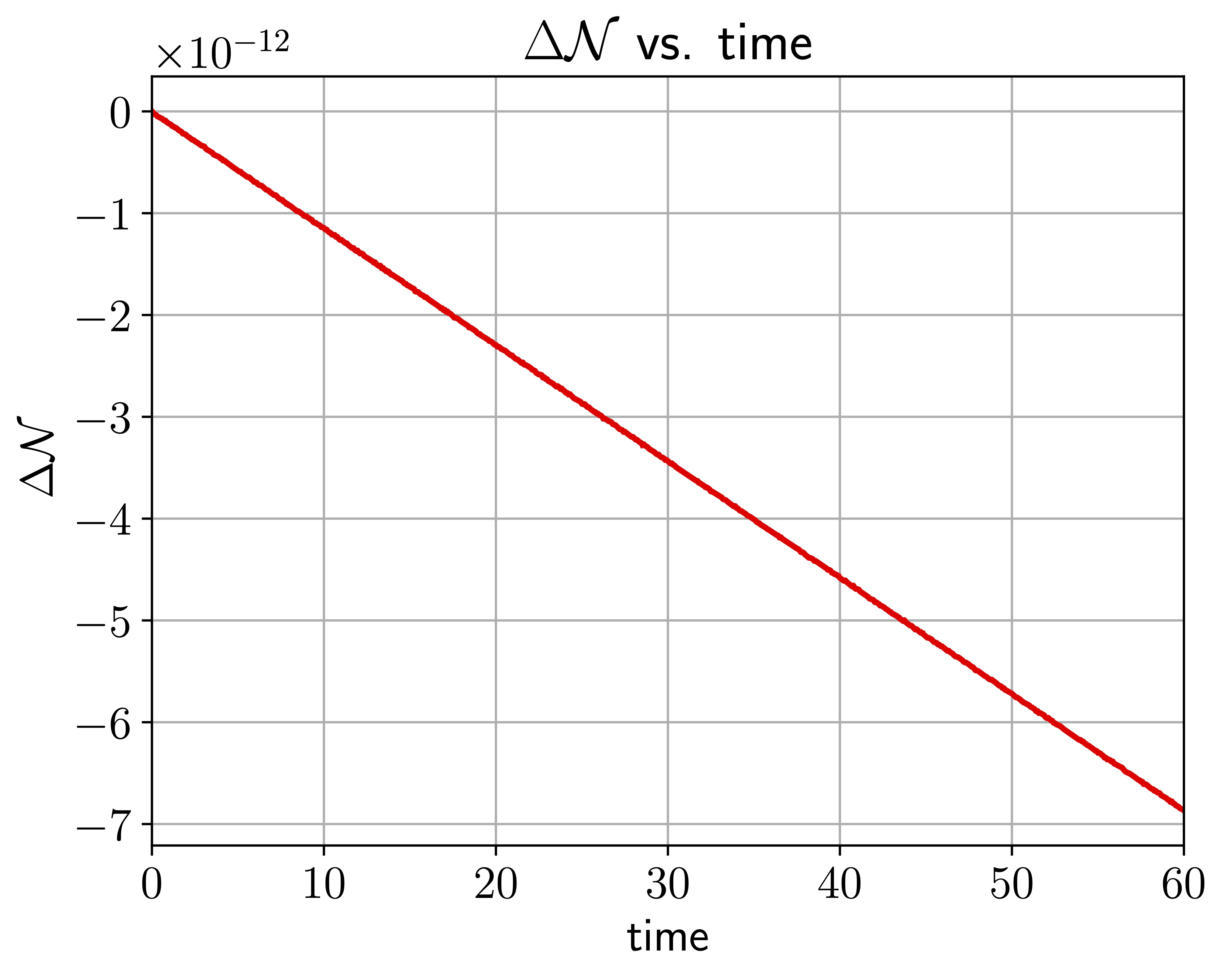}
   \end{tabular}
    \caption{(\S\ref{subsec:twostream}: Two-stream instability): Panel (a) shows the phase space plot of the distribution function, $f(t\!=\!45,x,v)$. Panel (b) shows the slice $f\left(t\!=\!45,x\!=\!0,v\right)$, which demonstrates that the positivity-preserving limiter maintains a non-negative distribution while retaining the resolved fine-scale structure.
    Panel (c) shows the time history of the $L^{2}$-norm of the electric field, $\|E(t,\cdot)\|_{2}$, over $t\in[0,60]$, while Panel (d) shows the particle number drift over the same time window. These results are computed using fourth-order operator-splitting with the $\morder\!=\!4$ SLDG method, $N_x\times N_v\!=\!128\times128$ elements, and $\mathrm{CFL}\!=\!5$.}
\label{fig:two-stream-dist}
\end{center}
\end{figure}

\subsubsection{Weak Landau damping}
\label{subsec:weak_landau}
Next we consider weak Landau damping, which has been extensively studied both numerically (e.g., Zhou, Guo, and Shu \cite{article:ZhouGuoShu2001} and Finn et al. \cite{article:Finn2023}) and analytically (e.g., Landau \cite{article:Landau1946} and Mouhot and Villani \cite{article:MouhotVillani2011}). In the weakly perturbed regime, the electric field is expected to exhibit an exponentially decaying envelope, making this test useful for assessing both accuracy and long-time conservation.

The initial distribution is a small perturbation of a Maxwellian
\begin{equation}
\label{eq:ld_ic}
f(t\!=\!0,x,v)=\frac{1}{\sqrt{2\pi}}\bigl(1+\alpha\cos(kx)\bigr)e^{-\frac{v^2}{2}} \quad \text{and} \quad
E(t\!=\!0,x)=-\frac{\alpha}{k}\sin(kx),
\end{equation}
with $\alpha\!=\!0.01$, $k\!=\!0.5$, $\rho_0\!=\!1$, and $J_0\!=\!0$. The computation is performed on $(x,v)\in[-2\pi,2\pi]\times[-2\pi,2\pi]$ with periodic boundary conditions in $x$, and is advanced to a final time of $t\!=\!60$. The grid resolution is $N_x \times N_v\!=\!128 \times 128$, and the CFL number is $\text{CFL}\!=\!5$. We test the fourth-order split scheme with the $\morder\!=\!4$ SLDG method.

As a diagnostic, we monitor the electric field strength as a function of time:
\begin{equation}
\label{eq:ld_E_L2}
    \|E(t,\cdot)\|_{2}
    :=
    \sqrt{
    \int_{-2\pi}^{2\pi} \bigl|E(t,x)\bigr|^2\,d\!x},
\end{equation}
which is shown in Panel (a) of Figure~\ref{fig:weak-landau-E}. The fourth-order scheme accurately captures the exponential decay that is predicted by linear theory \cite{article:Landau1946}. Panel (b) of Figure~\ref{fig:weak-landau-E} shows the particle number drift over the course of the simulation, $t\in[0,60]$.

All of the provided images can be recreated from the DoGPack \cite{dogpack} source code via the following commands:
\begin{tcolorbox}
\begin{verbatim}
cd $DOGPACK/apps/2d/one_species_va_1d1v/landau_damping_weak/; make;
dog.exe; make plotefield; make plotcon;
\end{verbatim}
\end{tcolorbox}

\begin{figure}[!th]
\begin{center}
\begin{tabular}{cc}
   (a)\includegraphics[width=0.44\textwidth]{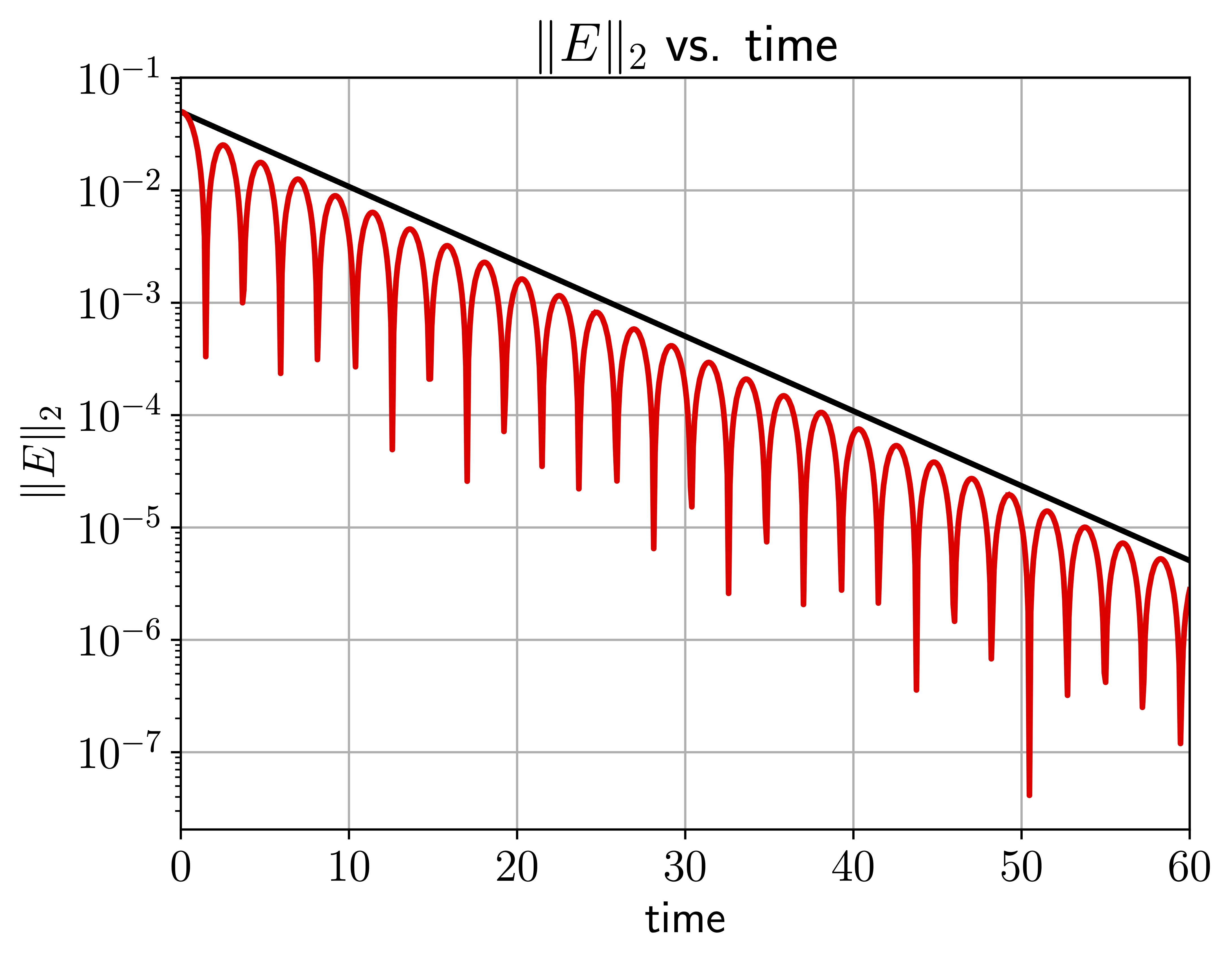} 
   (b)\includegraphics[width=0.44\textwidth]{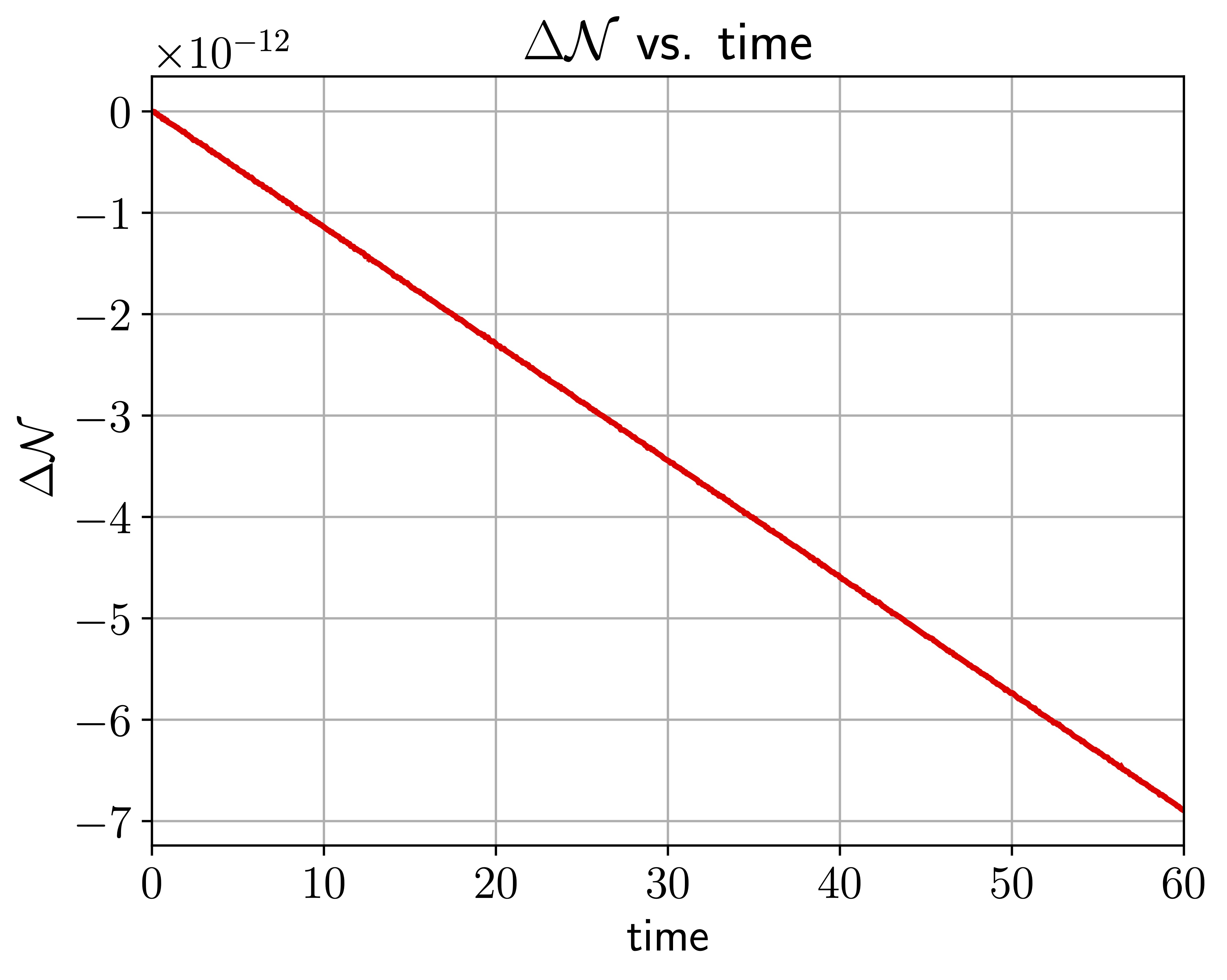}
\end{tabular}
    \caption{(\S\ref{subsec:weak_landau}: Weak Landau damping) Panel (a) shows time history of the $L^{2}$-norm of the electric field, $\|E(t,\cdot)\|_{2}$, over $t\in[0,60]$, while Panel (b) shows the conservation diagnostics for particle number. Panel (a) shows the expected linear decay rate of the electric field in the weakly perturbed regime. These results are obtained using fourth-order operator-splitting with the $\morder\!=\!4$ SLDG method, $N_x\times N_v\!=128\times128$ elements, and $\mathrm{CFL}\!=\!5$. }
\label{fig:weak-landau-E}
\end{center}
\end{figure}

\subsubsection{Strong Landau damping}
\label{subsec:strong_landau}
Strong Landau damping is the same as weak Landau damping, but with a larger perturbation in the initial condition \eqref{eq:ld_ic}: $\alpha\!=\!0.5$. Compared with the weakly perturbed case, the strong initial perturbation produces nonlinear phase mixing, particle trapping, and fine-scale filamentation in phase space. The computation is performed on $(x,v)\in\left[-2\pi,2\pi\right]\times\left[-2\pi,2\pi\right]$ with periodic boundary conditions in $x$, a mesh resolution of $N_x \times N_v\!=\!128 \times 128$, a CFL number of $\mathrm{CFL}\!=\!5$, and a final time of $t\!=\!60$. We test the fourth-order split scheme with the $\morder\!=\!4$ SLDG method.

Figure~\ref{fig:strong-landau-dist} shows the phase space distribution, $f\left(t\!=\!45,x,v\right)$, in Panel (a) and a slice of this solution, $f\left(t\!=\!45,x\!=\!0,v\right)$, in Panel (b). The solution develops the expected nonlinear structures, including trapped-particle vortices, pronounced filamentation, and clear resolution of the finer structures as depicted in Panel (a). As shown in Panel (b), applying the positivity-preserving limiter eliminates negative values and ensures the distribution remains non-negative. Panels (c) and (d) show the time history of $\|E(t,\cdot)\|_{2}$ and particle number drift over the simulation time, $t\in[0,60]$, respectively. The electric field strength exhibits the characteristic strong Landau damping behavior: (1) an initial decay that is followed by (2) nonlinear growth, and (3) long-time oscillations as shown in Panel (c). As depicted in Panel (d), the particle number drift remains very small throughout the simulation, demonstrating that the positivity-preserving limiter can be applied without compromising discrete conservation.

All of the provided images can be recreated from the DoGPack \cite{dogpack} source code via the following commands:
\begin{tcolorbox}
\begin{verbatim}
cd $DOGPACK/apps/2d/one_species_va_1d1v/landau_damping_strong/;
make; dog.exe; make plotfull; make plotefield; make plotcon;
\end{verbatim}
\end{tcolorbox}

\begin{figure}
\begin{center}
\begin{tabular}{cc}
   (a)\includegraphics[width=0.44\textwidth]{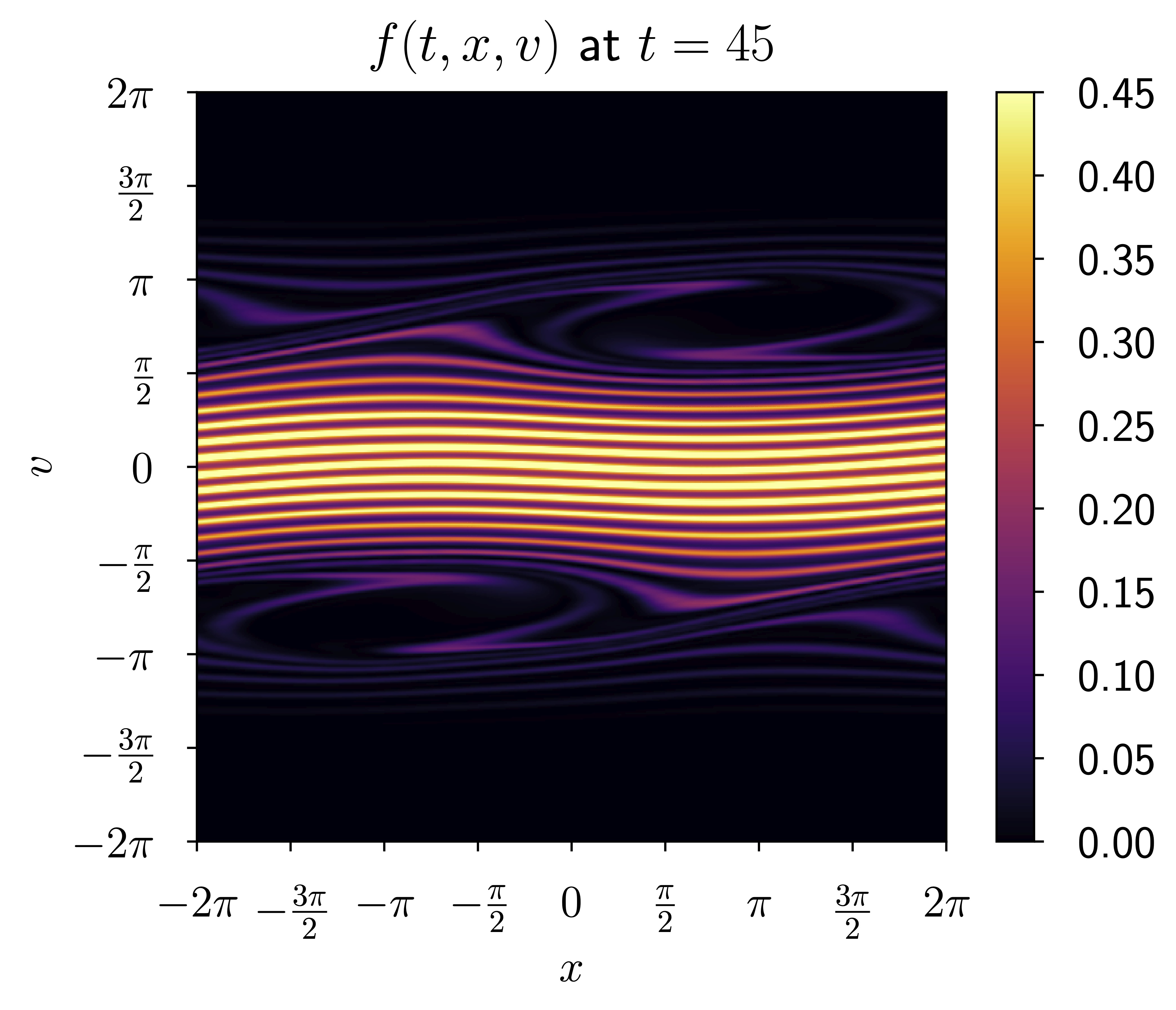} &
   (b)\includegraphics[width=0.44\textwidth]{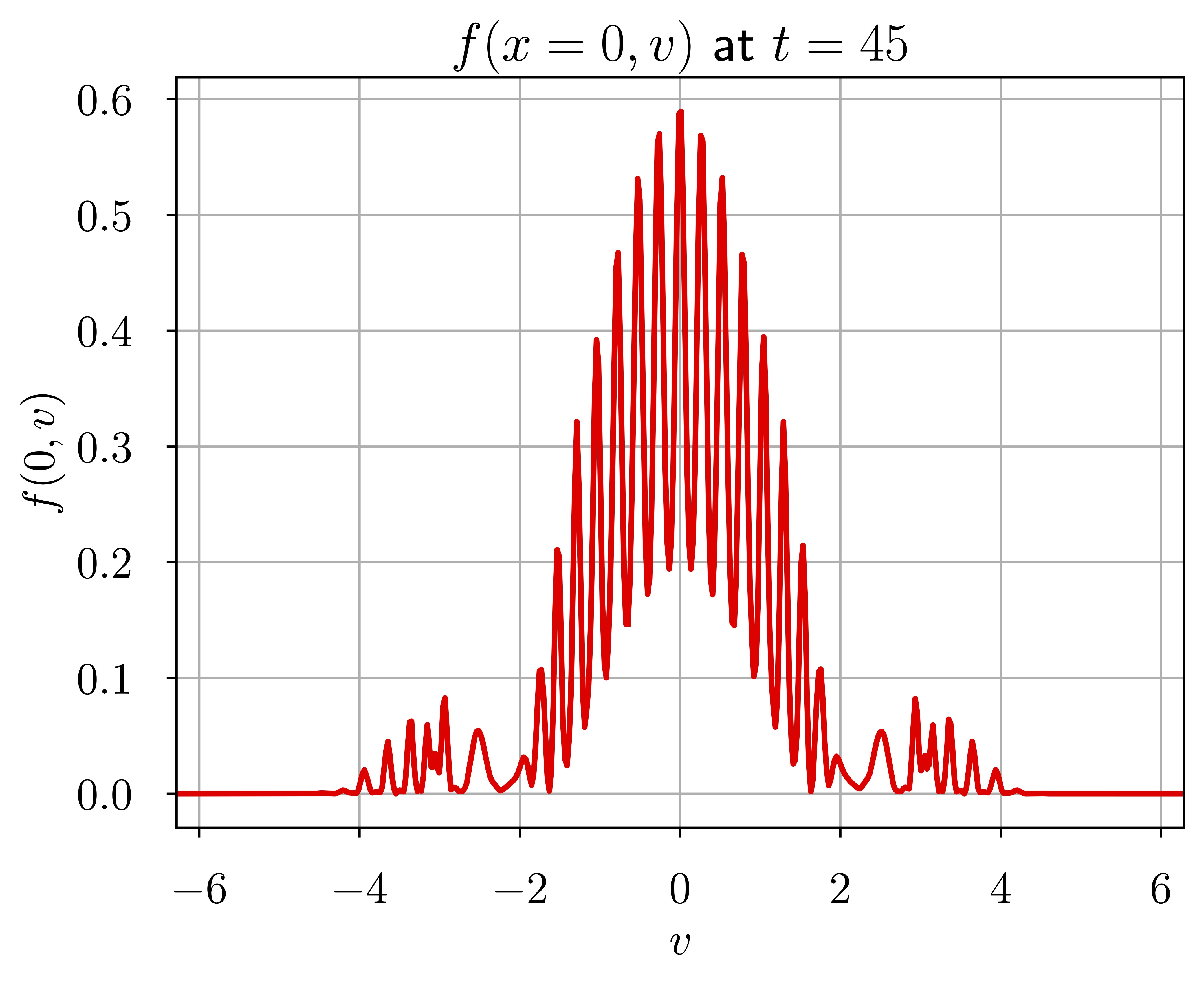} \\
   (c)\includegraphics[width=0.44\textwidth]{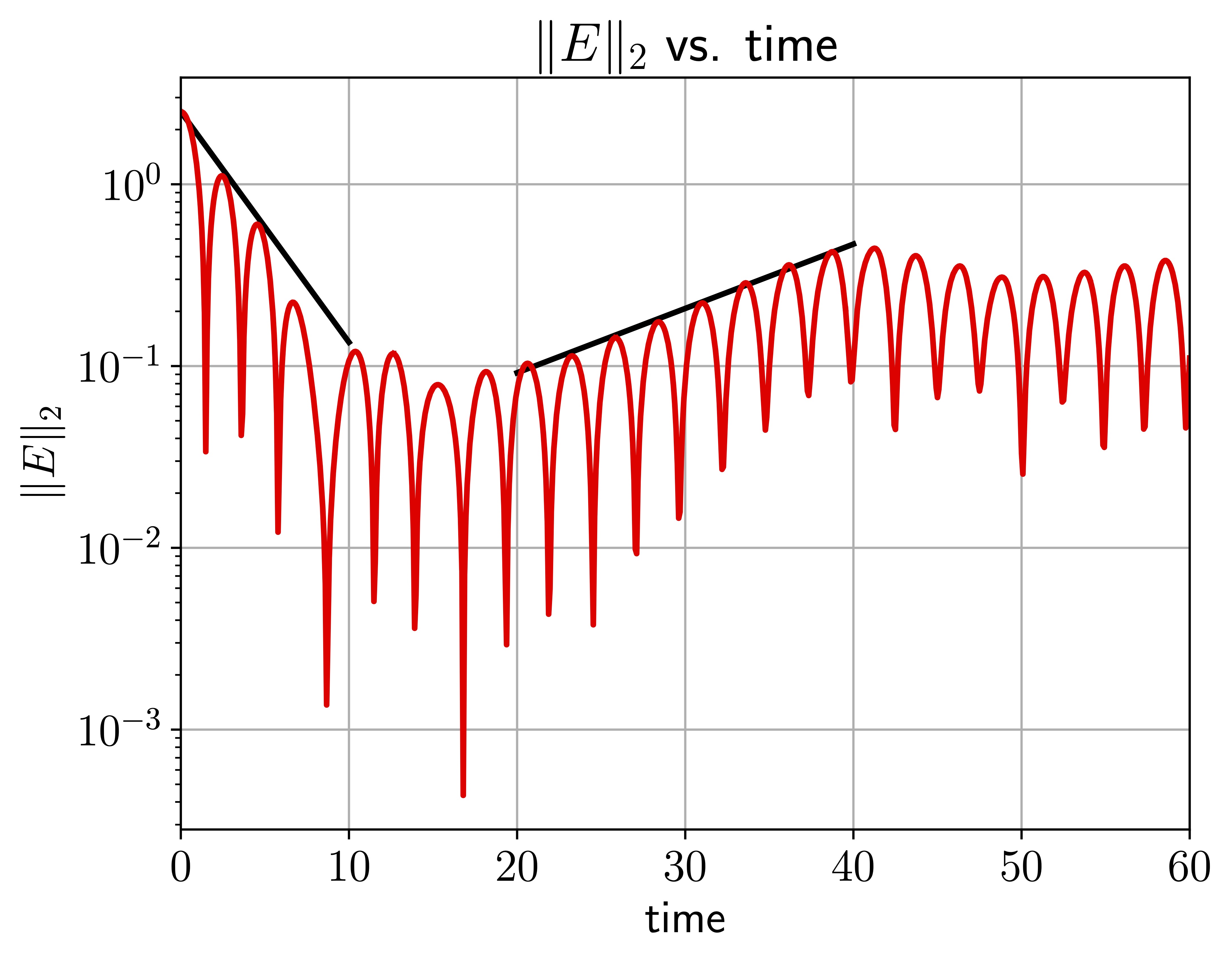} &
   (d)\includegraphics[width=0.44\textwidth]{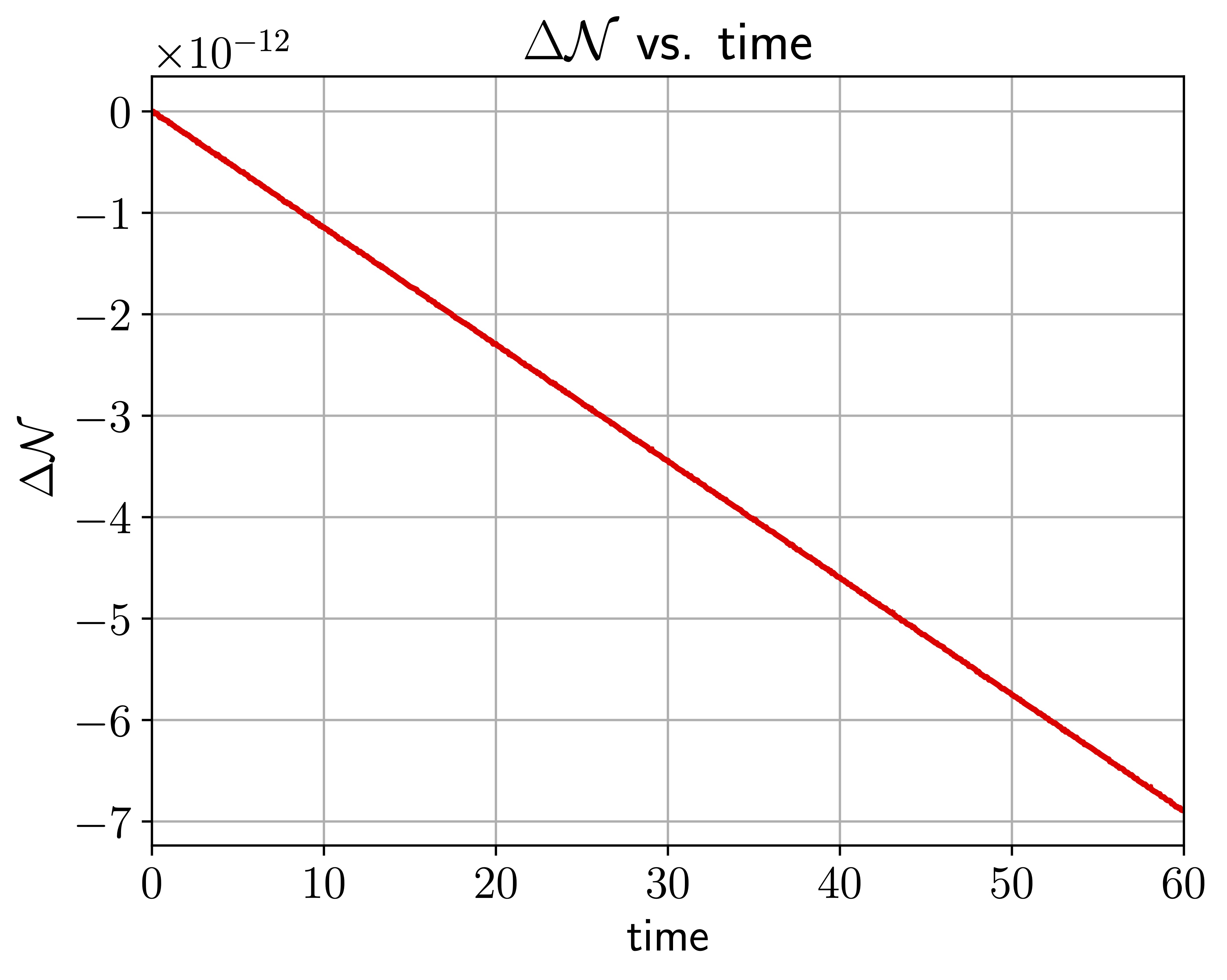}
\end{tabular}
    \caption{(\S\ref{subsec:strong_landau}: Strong Landau damping) 
       Panel (a) shows the phase space plot of the distribution function, $f(t\!=\!45,x,v)$. Panel (b) shows the slice $f\left(t\!=\!45,x\!=\!0,v\right)$, which demonstrates that the positivity-preserving limiter maintains a non-negative distribution while retaining the resolved fine-scale structure.
Panel (c) shows the time history of the $L^{2}$-norm of the electric field, $\|E(t,\cdot)\|_{2}$, over $t\in[0,60]$, while Panel (d) shows the particle number drift over the same time window. The electric field history exhibits the characteristic damping followed by nonlinear growth in the strongly perturbed regime.
These results were computed using fourth-order operator-splitting with the $\morder\!=\!4$ SLDG method, $N_x\times N_v\!=\!128\times128$ elements, and $\mathrm{CFL}\!=\!5$.}
\label{fig:strong-landau-dist}
\end{center}
\end{figure}

\subsection{Two-species model}\label{subsec:two_species}
We next consider the two-species 1D1V Vlasov-Amp\`ere model, in which both ions and electrons are
advanced self-consistently and are coupled through the electric field. In contrast to the
single-species case, the field is determined by the combined charge response of the two
distributions, and the dynamics can exhibit pronounced multi-scale behavior due to disparate masses and/or
temperatures. The tests below are designed to verify accuracy in this coupled setting and to
assess long-time robustness under stiff electron time scales and comparatively slow ion evolution.

\subsubsection{Accuracy test: method of manufactured solution}
\label{subsec:mms_two_species}
We begin by again verifying the order accuracy for the two-species 1D1V Vlasov-Amp\`ere solvers using a manufactured solution. The forced system is
\begin{flalign}
\label{eq:mms_vlasov_two}
&f_{s,t}+v f_{s,x}+\frac{q_s}{m_s}E f_{s,v}=\psi_s, \quad E_{,x}=\sigma, \quad \sigma=\sum_{s}q_{s}\rho_{s}, \quad E_{,t}+J=S_2.
\end{flalign}
The problem is posed on $ (t,x,v)\in[0,1]\times[-\pi,\pi]\times[-\pi,\pi],$ with periodic boundary conditions in $x$. We take
\begin{equation}
m_i=1,\quad m_e=1,\quad q_i=1,\quad q_e=-1.
\end{equation}
The exact ion and electron distributions and the Gauss-consistent electric field are chosen as follows:
\begin{equation}
\label{eq:mms_exact_f_i}
f_i(t,x,v)=e^{-\frac{(4v-1)^2}{4}}
    \left(2-A\right), \quad
    f_e(t,x,v)=e^{-\frac{(4v-1)^2}{4}}
    \left(2-B\right), \quad \text{and} \quad
    E(t,x)=\frac{\sqrt{\pi}}{4}
    \left(A+B\right),
\end{equation}
where
\begin{equation}
A:=\sin(2x-2\pi t) \quad \text{and} \quad B:=\cos(2x-2\pi t).
\end{equation}
Substitution of these ``exact'' solutions into the Vlasov system yields the following source terms:
\begin{align}
\label{eq:mms_source_psi_i_two}
\psi_i(t,x,v)
=
e^{-\frac{(4v-1)^2}{4}}
\bigg[
(2\pi-2v)B-\frac{\sqrt{\pi}}{2m_i}(4v-1)
\big(2-A\big)
\big(A+B\big)
\bigg], \\
\label{eq:mms_source_psi_e_two}
\psi_e(t,x,v)
=
e^{-\frac{(4v-1)^2}{4}}
\bigg[
(2v-2\pi)A+\frac{\sqrt{\pi}}{2m_e}(4v-1)
\big(2-B\big)
\big(A+B\big)
\bigg],
\end{align}
and 
\begin{equation}
\label{eq:mms_source_S2_two}
S_2(t,x)=\sqrt{\pi}
    \left(\frac{4\pi-1}{8}\right)
    \bigl(A-B\bigr).
\end{equation}
This example is a direct extension of the single-species manufactured solution test case of \cite{RossmanithSeal2011,thesis:Seal2012} that was describein in \S\ref{subsec:mms_single}.

We split the forced problem into the following two sub-problems:
\begin{align}
\label{eq:mms_problem_A_two}
\textnormal{Problem $\mathcal{A}$:}& \quad
    f_{s,t}+v f_{s,x}=\psi_s(t,x,v),
    \quad
    E_{,t}=0, \quad s\in\{e,i\}, \\
\label{eq:mms_problem_B_two}
\textnormal{Problem $\mathcal{B}$:}& \quad
    f_{s,t}+\frac{q_s}{m_s}E f_{s,v}=0,
    \quad
    E_{,t}+J=S_2(t,x), \quad s\in\{e,i\}.
\end{align}
Problem \(\mathcal A\) can still be solved using the method of characteristics, but in this setting, the distribution is no longer constant along each characteristic; instead, we need to compute the following integral:
\begin{equation}
\label{eq:mms_SL_update_A_two}
    f_s\left(t^{n+1},x,v\right)
    =
    f_s\left(t^n,x-v\Delta t,v\right)
    +
    \int_{t^n}^{t^{n+1}}
    \psi_s\Bigl(\tau,x-v\left(t^{n+1}-\tau\right),v\Bigr)\,d\tau.
\end{equation}
Problem \(\mathcal B\) is advanced by
the same semi-Lagrangian velocity tracing used in the unforced scheme, but with the field equation
modified by the source \(S_2\). 

Errors are again computed via the approximate relative $L^2$ errors defined via \eqref{eqn:error}.
A convergence study for this problem is shown in Table~\ref{va:mms-fi_two-species}, which confirms that the scheme attains its expected order accuracy. In these simulations, the final time is $t\!=\!1$, and the CFL number is fixed at $\text{CFL}\!=\!2$. We note that the errors decrease at the correct rates for both the second-order splitting with the $\morder\!=\!2$ SLDG method (Columns 2 and 3) and for the fourth-order splitting with the $\morder\!=\!4$ SLDG method (Columns 4 and 5). 

The convergence results reported in these tables were generated using the DoGPack code~\cite{dogpack} and can be reproduced from the code via the following commands:
\begin{tcolorbox}
\begin{verbatim}
cd $DOGPACK/apps/2d/two_species_va_1d1v/manufactured_soln_convergence_test;
make errortable;
\end{verbatim}
\end{tcolorbox}

\begingroup
\setlength{\tabcolsep}{10pt} %
\renewcommand{\arraystretch}{1.5} %
\begin{table}
\begin{center}
\begin{tabular}{|c||c|c||c|c|}
    \hline
    \textbf{$N_{x} \times N_{v}$} & \textbf{error: $2^{\text{nd}}$ order split} & $\log_{2}(\text{ratio})$ &
                    \textbf{error: $4^{\text{th}}$ order split} & $\log_{2}(\text{ratio})$ \\ \hline\hline
     $5\times5$     & $4.5376\times10^{-1}$ & --      & $6.3249\times10^{-2}$ & --      \\ \hline
    $10\times10$   & $1.7683\times10^{-1}$ & $1.360$ & $1.1019\times10^{-2}$ & $2.521$ \\ \hline
    $20\times10$   & $3.4290\times10^{-2}$ & $2.367$ & $6.5239\times10^{-4}$ & $4.078$ \\ \hline
    $40\times40$   & $8.3346\times10^{-3}$ & $2.041$ & $3.9336\times10^{-5}$ & $4.052$ \\ \hline
    $80\times80$   & $2.1014\times10^{-3}$ & $1.988$ & $2.3877\times10^{-6}$ & $4.042$ \\ \hline
    $160\times160$ & $5.3166\times10^{-4}$ & $1.983$ & $1.4242\times10^{-7}$ & $4.067$ \\ \hline
    $320\times320$ & $1.3438\times10^{-4}$ & $1.984$ & $8.6445\times10^{-9}$  & $4.042$ \\ \hline
    \hline\hline
    \textbf{$N_{x} \times N_{v}$} & \textbf{error: $2^{\text{nd}}$ order split} & $\log_{2}(\text{ratio})$ &
                    \textbf{error: $4^{\text{th}}$ order split} & $\log_{2}(\text{ratio})$ \\ \hline\hline
    $5\times5$     & $4.8193\times10^{-1}$ & --      & $6.8822\times10^{-2}$ & --      \\ \hline
    $10\times10$   & $1.7967\times10^{-1}$ & $1.423$ & $1.1313\times10^{-2}$ & $2.605$ \\ \hline
    $20\times20$   & $3.3933\times10^{-2}$ & $2.405$ & $6.5213\times10^{-4}$ & $4.117$ \\ \hline
    $40\times40$   & $8.2723\times10^{-3}$ & $2.036$ & $3.8877\times10^{-5}$ & $4.068$ \\ \hline
    $80\times80$   & $2.0944\times10^{-3}$ & $1.982$ & $2.3784\times10^{-6}$ & $4.031$ \\ \hline
    $160\times160$ & $5.3113\times10^{-4}$ & $1.979$ & $1.4185\times10^{-7}$ & $4.068$ \\ \hline
    $320\times320$ & $1.3436\times10^{-4}$ & $1.983$ & $8.6162\times10^{-9}$ & $4.041$ \\ \hline
\end{tabular}
\end{center}
\caption{(\S\ref{subsec:mms_two_species}: Two-Species Manufactured Solution) Manufactured solution convergence study for the ion and electron PDFs for the two-species semi-Lagrangian DG method. The top half of the table is for the ion PDF, the bottom half is for the electron PDF. Reported in this table are the relative $L^2$-errors at increasing mesh resolutions along with the log of the error ratios for both the second-order splitting with $\morder\!=\!2$ SLDG in phase space and fourth-order splitting with $\morder\!=\!4$ SLDG in phase space. Note that $\log_2$ of the error ratio gives an estimate for the order of convergence, since each row in the table represents a doubling of mesh resolution compared to the previous row. All runs are computed to a final time of $t=1$ with $\text{CFL}\!=\!2$ over the domain $(x,v)\in[-\pi,\pi]\times[-\pi,\pi]$ with periodic boundary conditions in $x$.}
\label{va:mms-fi_two-species}
\end{table}
\endgroup

\subsubsection{Ion-acoustic wave}
\label{subsec:iaw}
We assess long-time robustness of the proposed scheme on the ion-acoustic wave benchmark, a standard electrostatic test case from the computational plasma physics community (e.g., see Chen, Chac{\'o}n, and Barnes~\cite{ChenChaconBarnes2011}). In this example, a wave is triggered by an ion density perturbation and propagates through ion compression/expansion driven by ion thermal motion, while the electron thermal motion provides only incomplete Debye shielding~\cite{Stix1992}. In linear kinetic theory, the ion-acoustic wave branch is characterized by the dielectric condition
\begin{equation}
\varepsilon(\omega,k)=1+\sum_{s\in\{e,i\}}\chi_{s}(\omega,k)=0,
\end{equation}
where \(\chi_s\) are species susceptibilities, \(\omega\) is the wave frequency, and \(k\) is the wavenumber. In the asymptotic regime:
\begin{equation}
    r_T:=\frac{T_e}{T_i}\gg 1,\quad r_m:=\frac{m_i}{m_e}\gg 1,
\end{equation}
one has \(\omega\ll\omega_{pe}\) and \(k\lambda_{De}\ll 1\), so the problem is genuinely multi-scale; ions set the acoustic time scale, while fast electrons dominate shielding and numerical stiffness.

In the specific example considered from Chen, Chac{\'o}n, and Barnes~\cite{ChenChaconBarnes2011}, each species is initialized with a Maxwellian perturbed by a small sinusoidal density mode with zero initial electric field: 
\begin{equation}
f_{s}(t\!=\!0,x,v)=\sqrt{\frac{m_{s}}{2\pi T_{s}}}
\exp\!\Big(-\frac{m_{s} v^2}{2T_{s}}\Big)\bigl(1+a\cos(kx)\bigr), \quad
E(t\!=\!0,x) \equiv 0.
\end{equation}
The various parameters are chosen as follows:
\begin{flalign}
k=\frac{2\pi}{L},\quad m_i=200,\quad T_i=10^{-4},\quad q_i=1,\quad m_e=1,\quad T_e=1,\quad q_e=-1,\quad a=0.2,\quad L=10,
\end{flalign}
so that \(r_m=200\) and \(r_T=10^4\). The computational domain is $(t,x,v)\in[0,2000]\times[0,L]\times I_{s}$
with periodic boundary conditions in \(x\). To resolve the disparate kinetic scales, we evolve each
species on a species-dependent velocity interval $I_e=\left[-7.5,\,7.5\right]$ and $I_i=\left[-0.2,\,0.2\right]$.

In Figures \ref{fig:iaw-dist} and \ref{fig:iaw-ke}, we report on results from a run with the fourth-order operator-splitting with the $\morder\!=\!4$ SLDG method, $N_x\times N_v\!=\!128\times416$, and $\mathrm{CFL}\!=\!10$. Figure \ref{fig:iaw-dist} shows phase space plots of the (a) initial ion, (b) initial electron, (c) final ion, and (d) final electron distribution functions. These results show that the electron distribution remains close to a Maxwellian with a weak spatial modulation, while the ion distribution stays narrowly supported in velocity and develops a clear \(x\)-dependent drift, consistent with acoustic compression/expansion and wave propagation over many periods.

To assess long-time behavior, Figure \ref{fig:iaw-ke} reports the species kinetic energies in Panels (a) and (b). The electron kinetic energy undergoes small-amplitude oscillations about its mean level, while the ion kinetic energy exhibits bounded oscillations on a much smaller scale, reflecting sustained energy exchange in this multi-scale two-species setting without spurious growth. Panels (c) and (d) of Figure \ref{fig:iaw-ke} report the total particle number drift, which remains on the order of \(\mathcal{O}(10^{-10})\) for both species throughout the simulation, demonstrating the robust long-time stability of the discretization for this coupled Vlasov-Amp\`ere test. 

The figures shown here were generated using the DoGPack code~\cite{dogpack} and can be reproduced via the following commands:
\begin{tcolorbox}
\begin{verbatim}
cd $DOGPACK/apps/2d/two_species_va_1d1v/ion_acoustic_wave; make; dog.exe; 
make plotfull; make plotcon;
\end{verbatim}
\end{tcolorbox}

\begin{figure}
\begin{center}
\begin{tabular}{cc}
   (a)\includegraphics[width=0.44\textwidth]{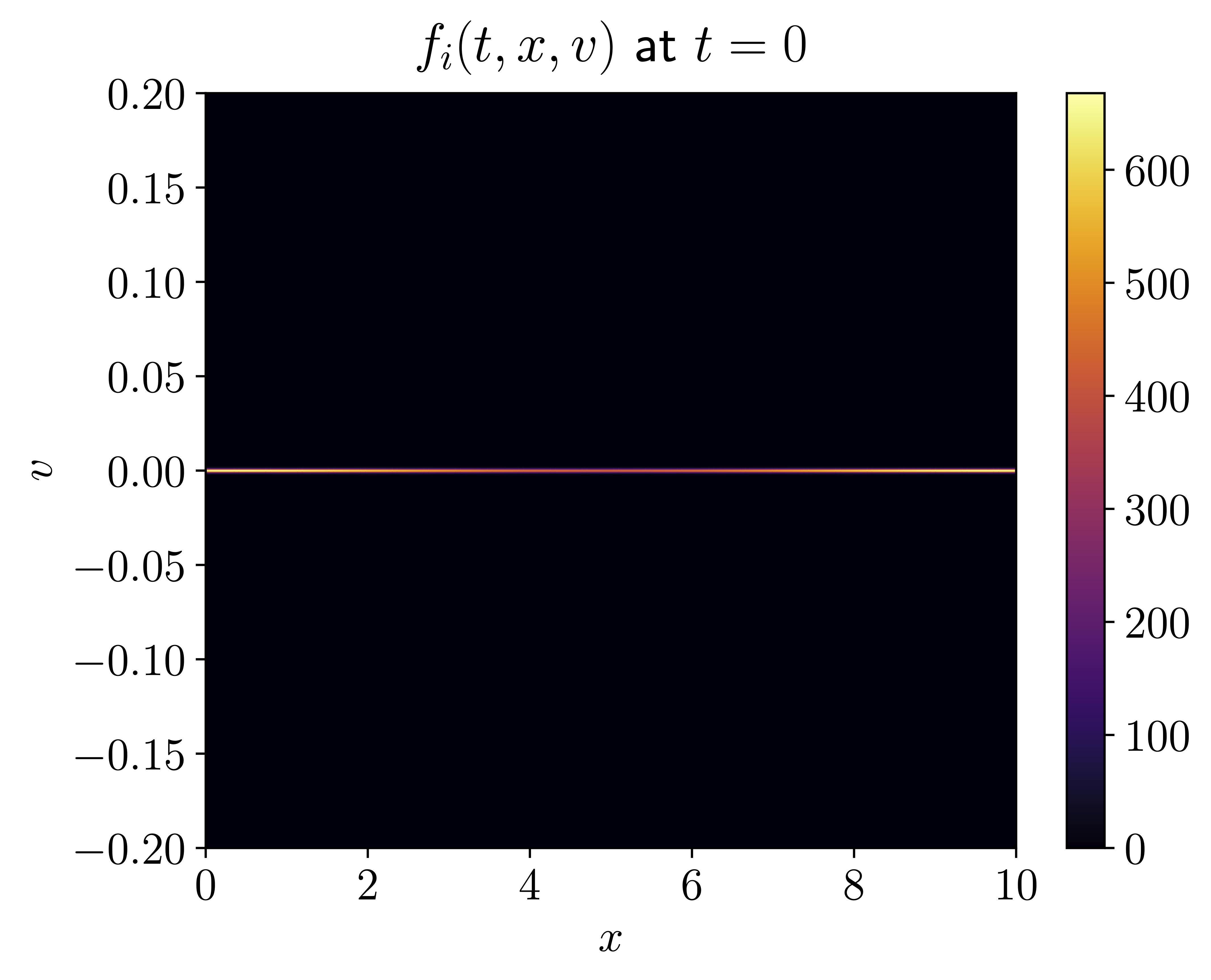} &
   (b)\includegraphics[width=0.44\textwidth]{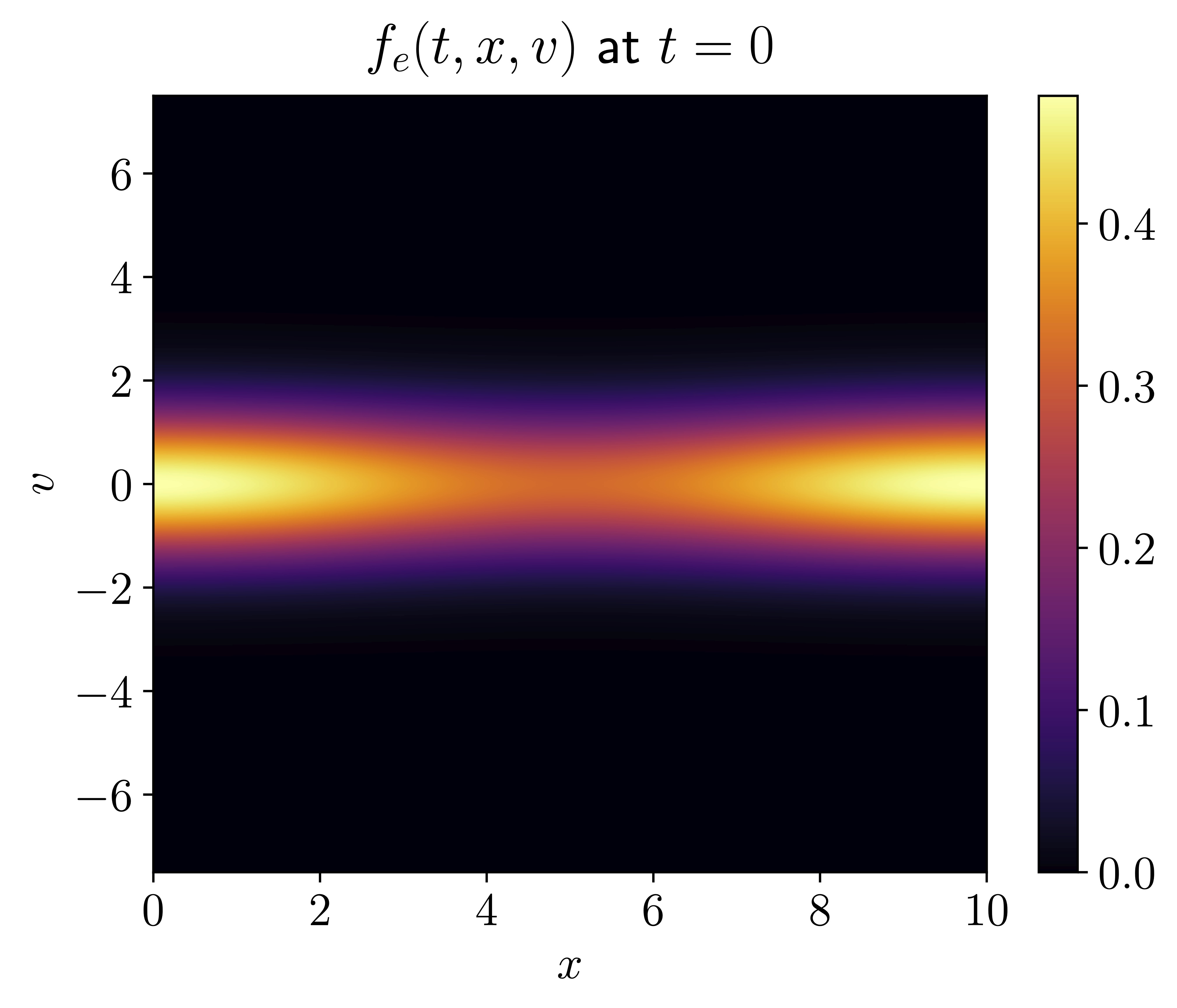} \\
   (c)\includegraphics[width=0.44\textwidth]{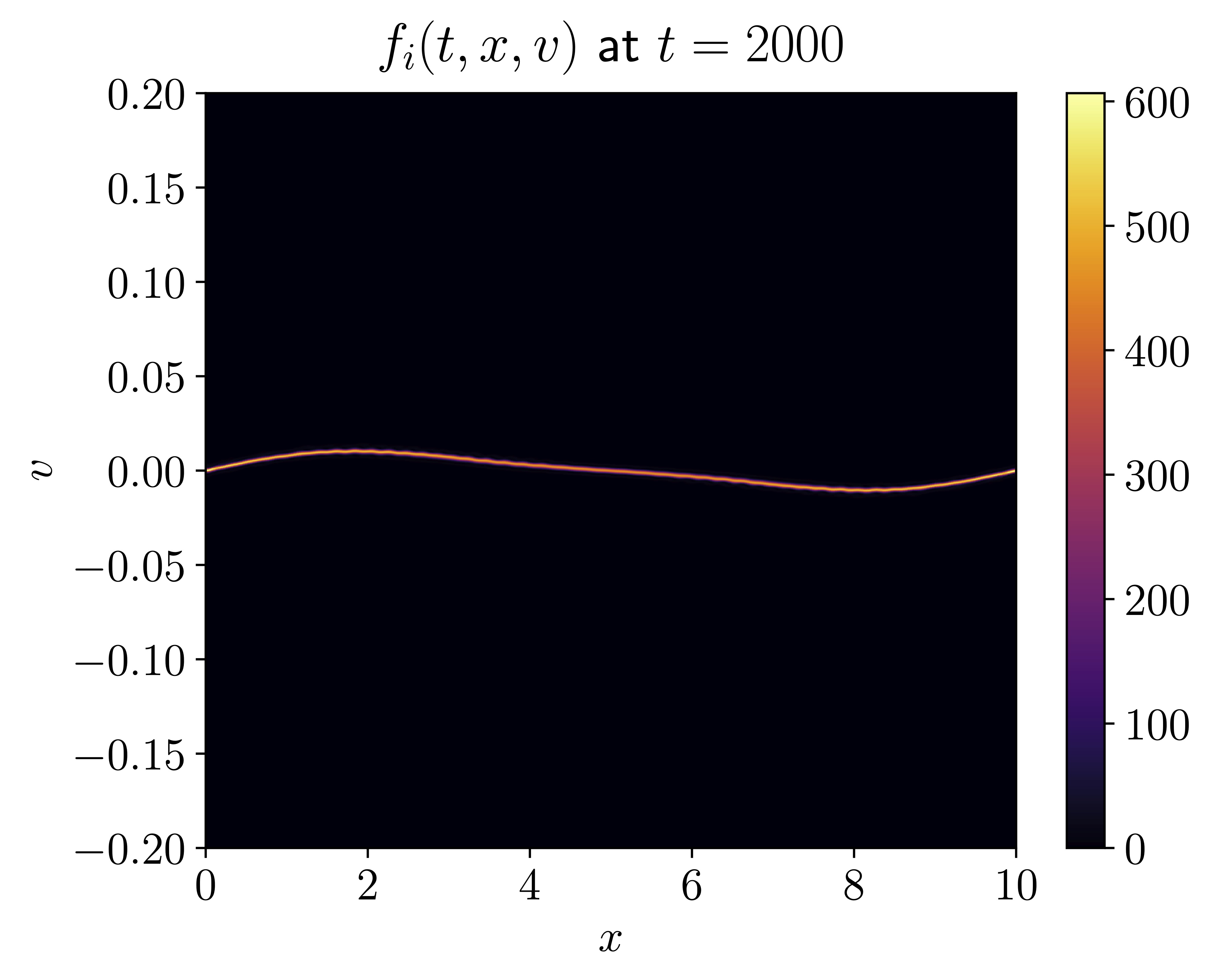} &
   (d)\includegraphics[width=0.44\textwidth]{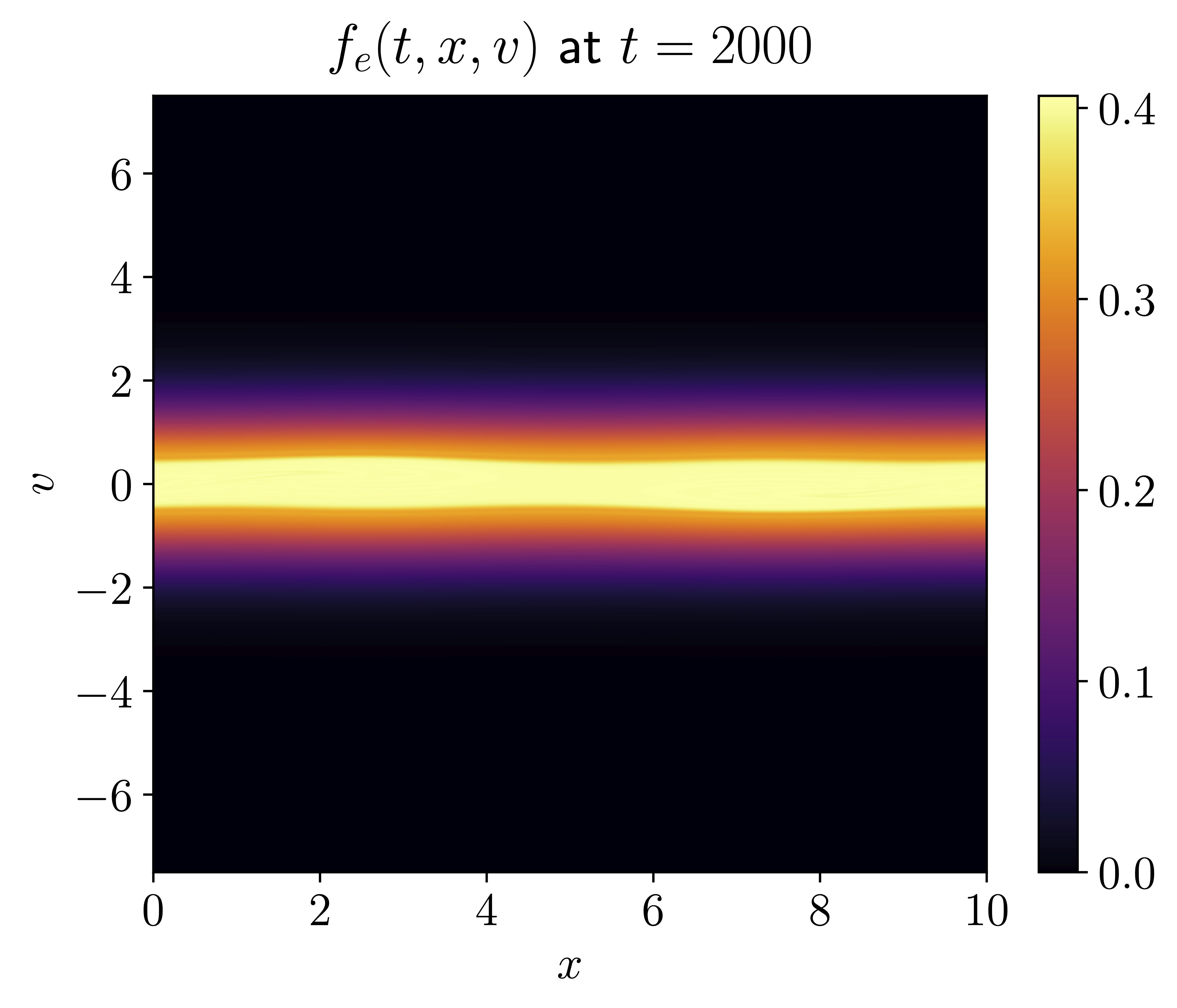}
\end{tabular}
    \caption{(\S\ref{subsec:iaw}: Ion-acoustic wave) Phase space plots of the ion and electron distribution functions: $f_i(t,x,v)$ and $f_e(t,x,v)$. Panels (a) and (b) show the ion and electron distribution functions at $t\!=\!0$, respectively; Panels (c) and (d) show the ion and electron distribution functions at $t=2000$, respectively. This simulation uses fourth-order operator-splitting with the $\morder\!=\!4$ SLDG method, $N_x\times N_v\!=\!128\times416$, and $\mathrm{CFL}\!=\!10$.}
\label{fig:iaw-dist}
\end{center}
\end{figure}

\begin{figure}
\begin{center}
\begin{tabular}{cc}
   (a)\includegraphics[width=0.44\textwidth]{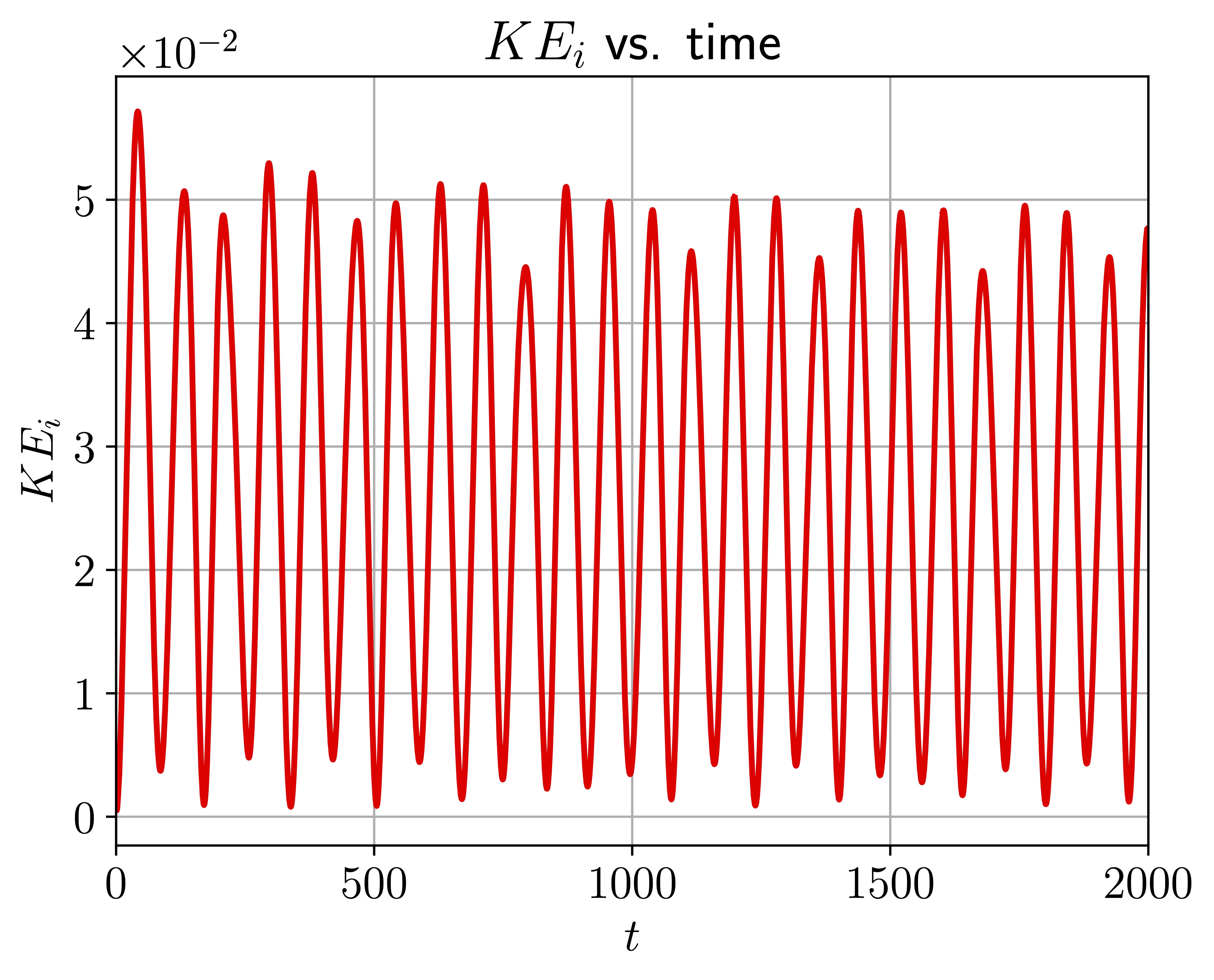} &
   (b)\includegraphics[width=0.44\textwidth]{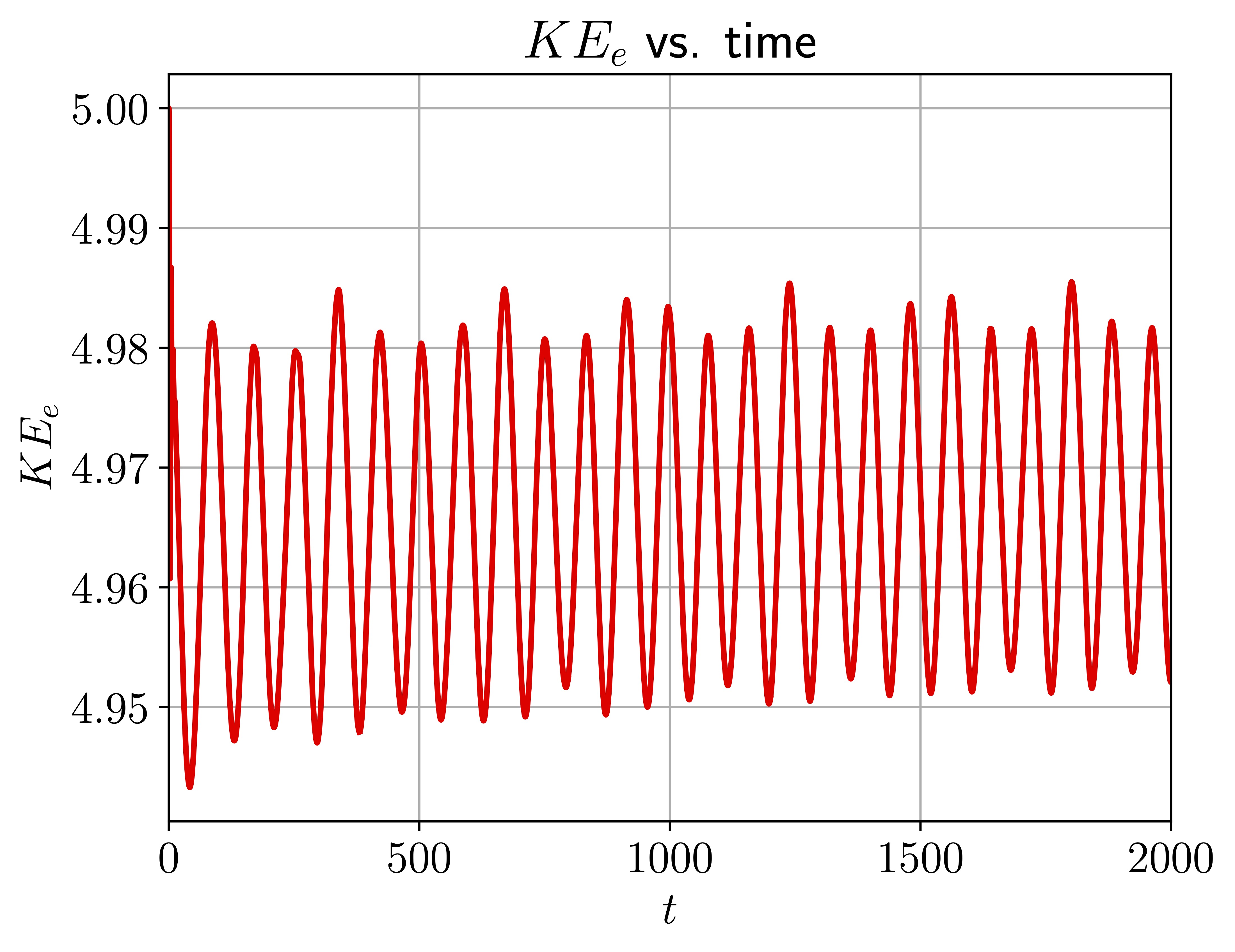} \\
   (c)\includegraphics[width=0.44\textwidth]{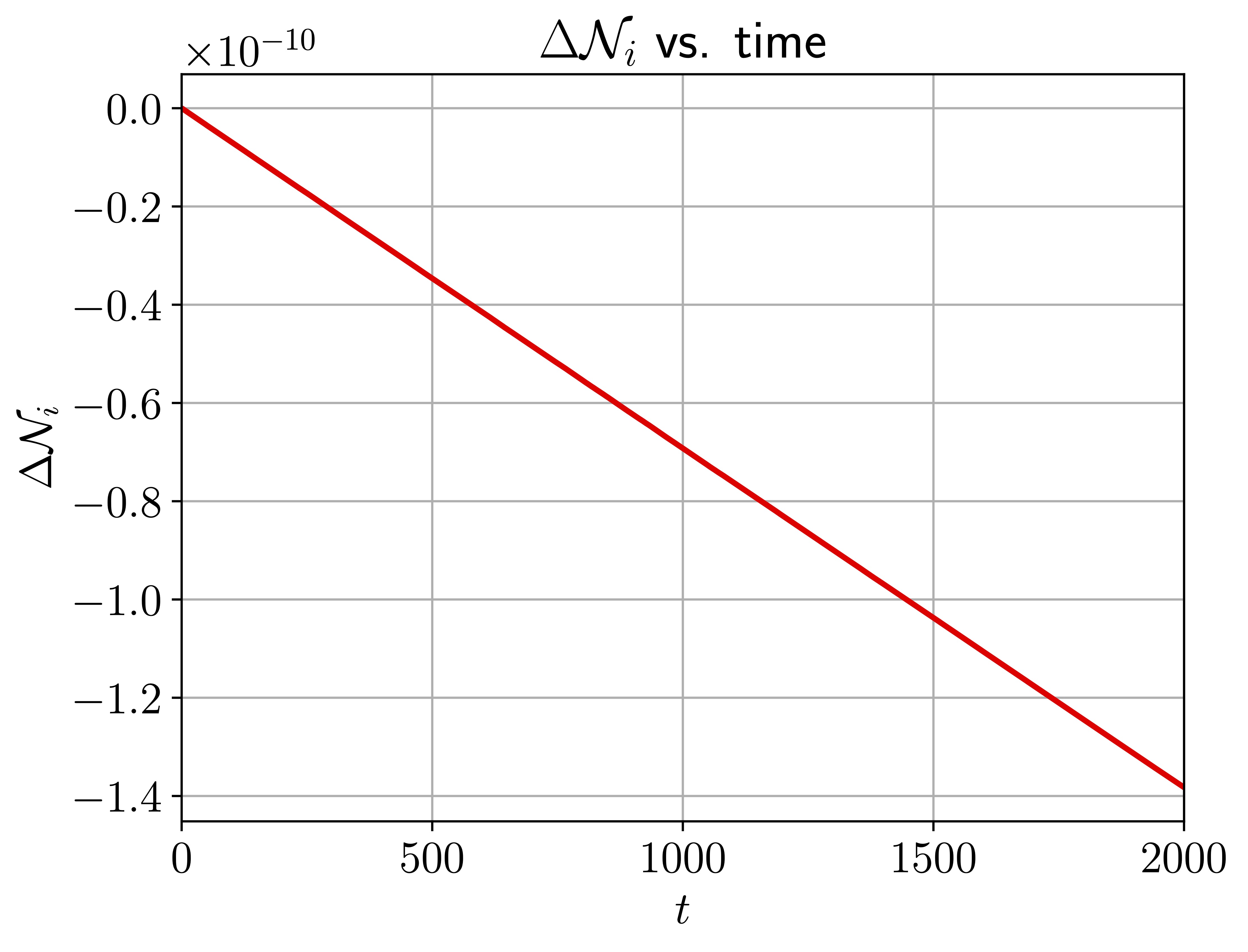} &
   (d)\includegraphics[width=0.44\textwidth]{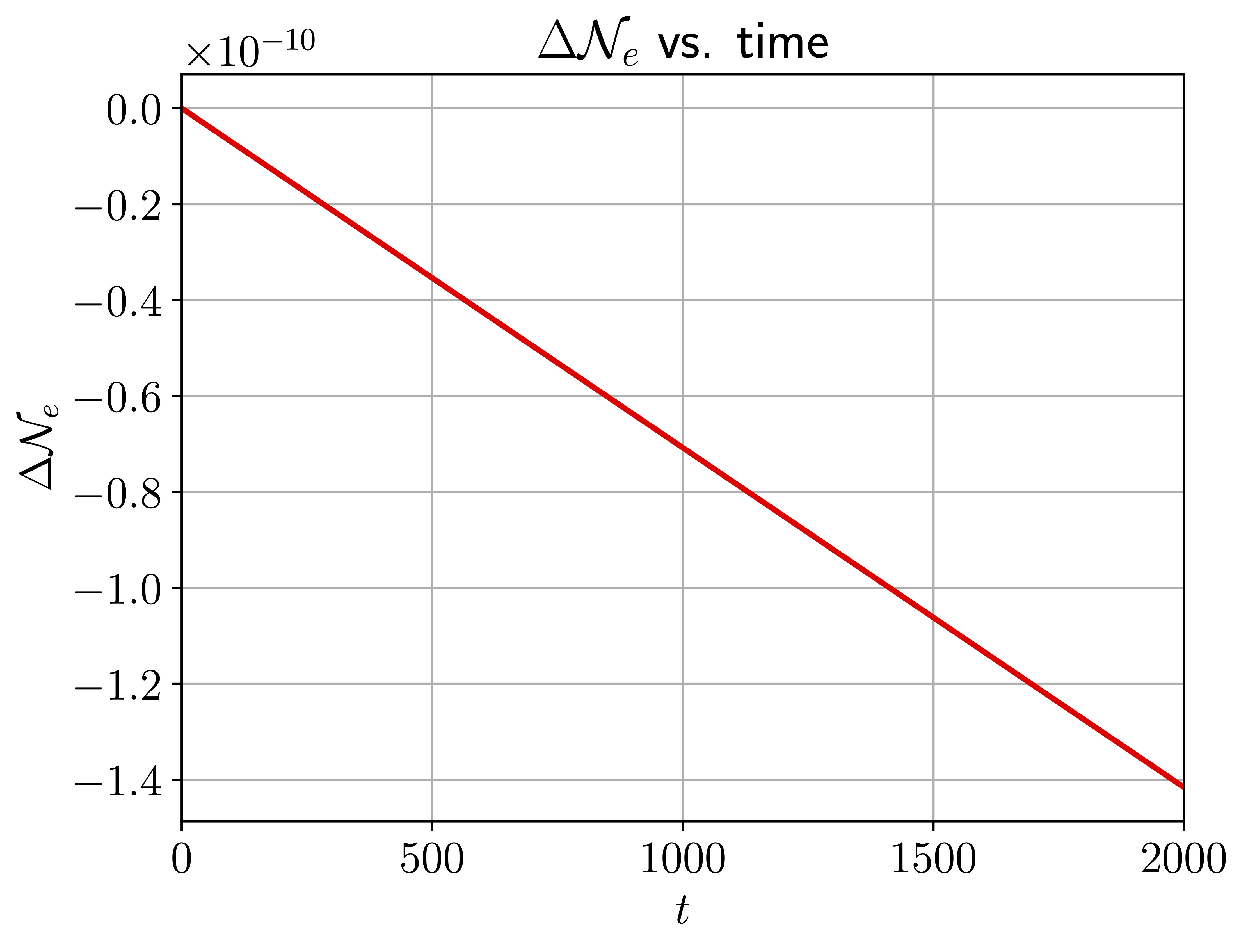}
\end{tabular}
    \caption{(\S\ref{subsec:iaw}: Ion-acoustic wave) Time histories over $t\in[0,2000]$ of the (a) ion kinetic energy, (b) electron kinetic energy, (c) ion particle number drift, and (d) electron particle number drift.}
\label{fig:iaw-ke}
\end{center}
\end{figure}

\subsubsection{Ion-acoustic shock}
\label{subsec:iasw}
The ion-acoustic shock test is a stringent, nonlinear, multiscale benchmark considered by several papers in the literature (e.g., see Parker, Friedman, Ray, and Birdsall \cite{article:Parker1993}, Shay and Drake \cite{article:Shay2007}, Chen, Chac{\'o}n, and Barnes~\cite{ChenChaconBarnes2011}, and Taitano, Burby, and Alekseenko~\cite{TaitanoBurbyAlekseenko2026}). In this example, an ion-acoustic (dispersive) shock forms when a propagating ion-acoustic wave steepens through nonlinear effects until dispersion and effective dissipation balance the steepening, yielding a persistent shock front \cite{SagdeevUsikovZaslavsky1988}. The steepening occurs on a time scale much longer than the plasma period, while the front width is comparable to the Debye length and therefore much smaller than the ion-acoustic wavelength. These disparate scales make ion-acoustic shock formation a demanding long-time accuracy test for kinetic solvers.

The initial setup consists of a periodic domain that contains a spatially uniform background plasma of hot electrons and cool ions, onto which a traveling wave perturbation is superimposed:
\begin{equation}\label{eq:iasw_ic}
\begin{gathered}
f_s(t\!=\!0,x,v)=\sqrt{\frac{m_s}{2\pi T_s}}\,
\exp\left(-\frac{m_s}{2T_s}\left[v+\sqrt{\frac{T_e}{m_i}}\,\bigl(1-a\sin(kx)\bigr)\right]^{2}\right)\,
n_{s}^0(x), \\ E(t\!=\!0,x)=-ak\cos(kx), \quad
n_{i}^0(x)=1+a\sin(kx), \quad
n_{e}^0(x)=1+a\left(1-k^2\right)\sin(kx),
\end{gathered}
\end{equation}
where $k={2\pi}/{L}$,  is the perturbation wavenumber. The parameters are chosen as follows
\begin{equation}\label{eq:iasw_params}
m_i=1836,\, \, \, T_i=0.05,\, \, \,  q_i=1,\, \, \, 
m_e=1,\, \, \,  T_e=1,\, \, \,  q_e=-1,\, \, \, 
a=0.2,\, \, \,  L=144,\, \, \,
\end{equation}
so that $r_m=1836$ and $r_T=20$. It should be noted, the ion temperature is chosen much lower than the electron temperature to suppress ion Landau damping and the associated wave dissipation, consistent with the setup in \cite{TaitanoBurbyAlekseenko2026}.

To efficiently handle the strong scale separation, the two species are evolved on different velocity ranges. The computational domain for each species, $s$, is $(t,x,v)\in[0,5000]\times[0,L]\times I_{s}$ with periodic boundary conditions in $x$, and $I_e=[-12.1,\,12.1]$ and $I_i=[-0.12,\,0.12]$.
This configuration resolves the Debye-scale shock transition layer while tracking the long-time nonlinear steepening and subsequent shock formation in a regime with widely separated ion/electron kinetic scales.
A numerical simulation run to a final time of $t=5000$, using fourth-order operator-splitting with the $\morder\!=\!4$ SLDG method, $N_x\times N_v\!=\!256\times512$, and $\mathrm{CFL}\!=\!10$.

Phase space plots are reported in Figure \ref{fig:ias-dist}; these illustrate the multi-scale response of the two species. The electron distribution remains close to a drifting Maxwellian away from the front, while localized distortions develop near the shock region at late times. In contrast, the ion distribution, develops a pronounced \(x\)-dependent drift with sharp gradients and fine-scale structure near the shock, reflecting nonlinear steepening and compression. 

Panels (a) and (b) of Figure~\ref{fig:ias-density} show the ion and electron number densities at $t=5000$, respectively. Both profiles exhibit a sharp transition between regions of enhanced and depleted density, consistent with nonlinear steepening and ion-acoustic shock formation. Panels (c) and (d) of Figure~\ref{fig:ias-density} show snapshots of the electric field at varying times, illustrating the formation and development of an ion-acoustic shock. The initially smooth wave steepens in time, a sharp front forms, and the transition layer progressively compresses while localized oscillations appear near the front. By $t=5000$, the shock structure is fully developed, with a Debye-scale transition region.
Finally, Panels (e) and (f) of Figure~\ref{fig:ias-density} show the particle number drift for the ion and electron species over $t\in[0,5000]$. These drifts remain very small, again demonstrating the robust long-time stability of the discretization for this coupled Vlasov-Amp\`ere test. 

 The figures presented here were generated using the DoGPack code~\cite{dogpack}. The results can be reproduced via the following commands:
\begin{tcolorbox}
\begin{verbatim}
cd $DOGPACK/apps/2d/two_species_va_1d1v/ion_acoustic_shock; make; dog.exe; 
make plotfull; make plotmom; make plotcon;
\end{verbatim}
\end{tcolorbox}

\begin{figure}
\begin{center}
\begin{tabular}{cc}
   (a)\includegraphics[width=0.44\textwidth]{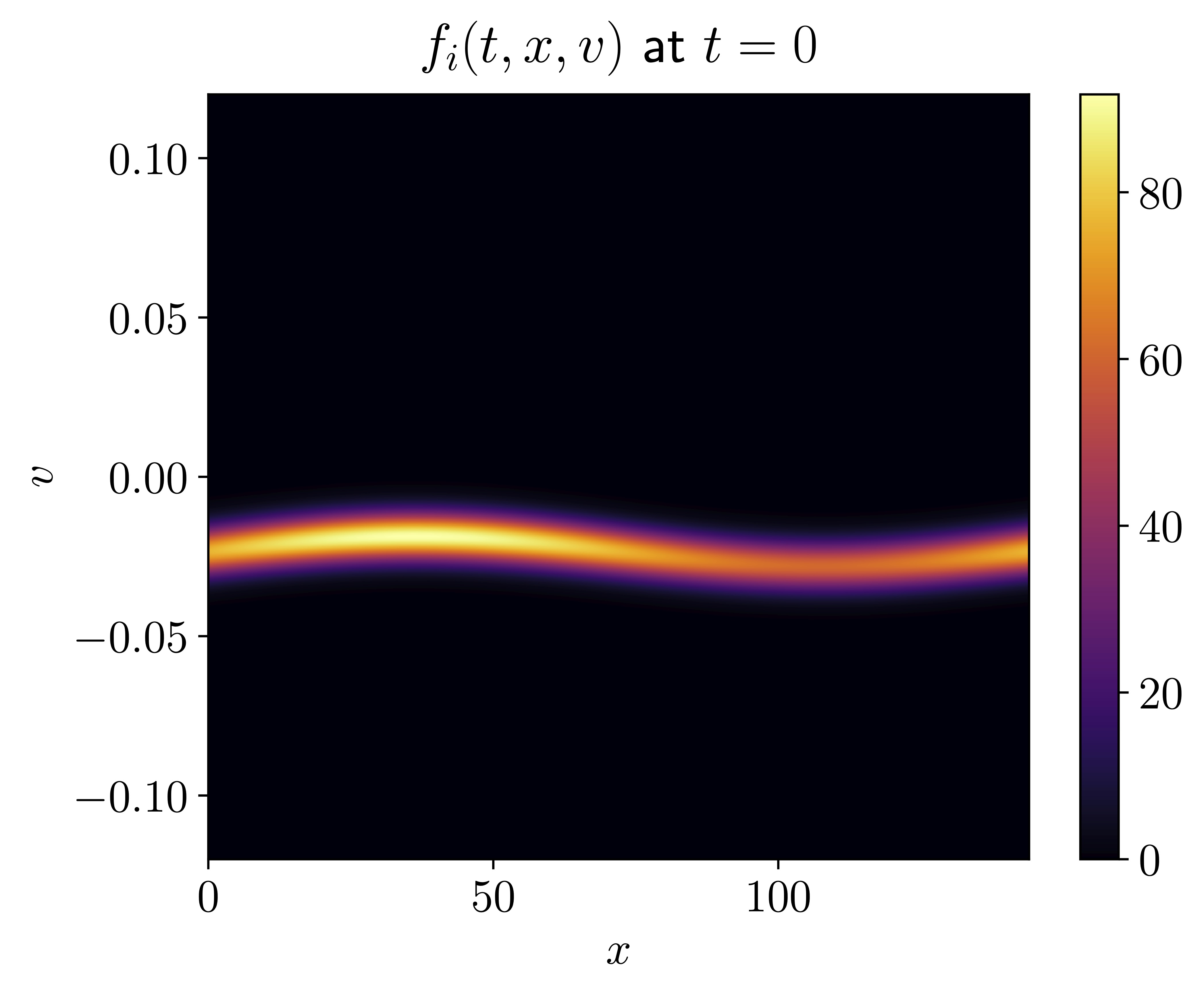} &
   (b)\includegraphics[width=0.44\textwidth]{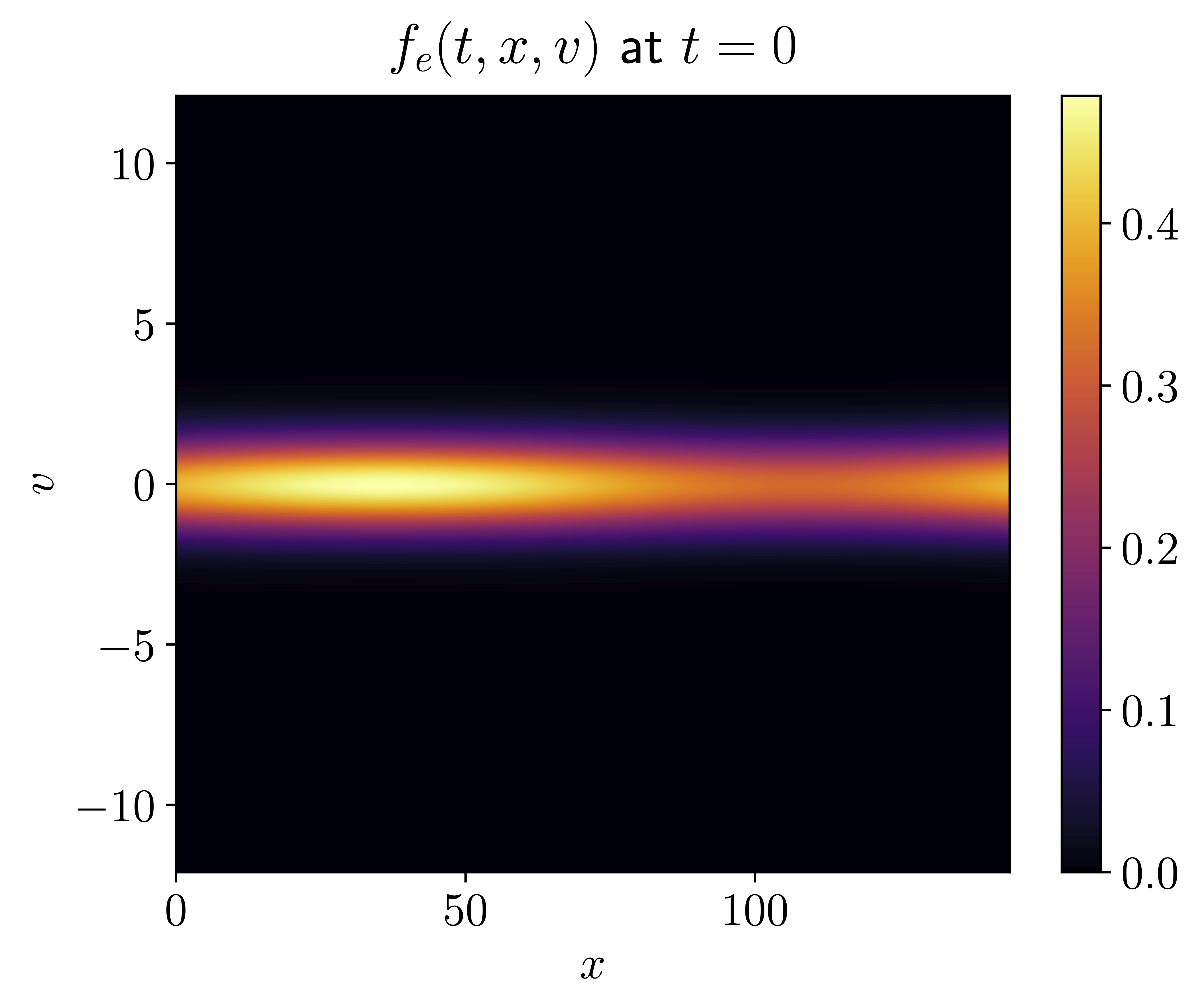} \\
   (c)\includegraphics[width=0.44\textwidth]{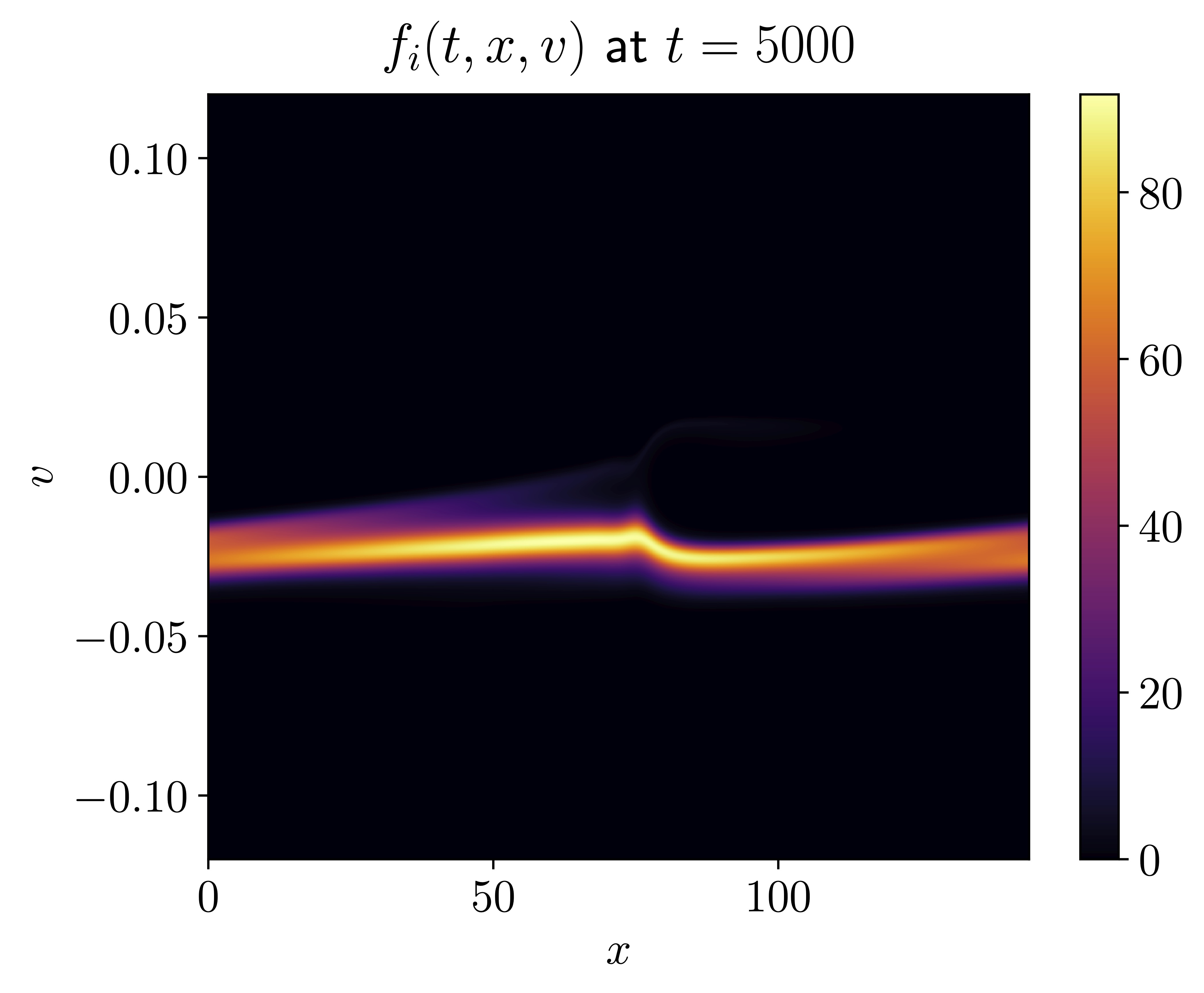} &
   (d)\includegraphics[width=0.44\textwidth]{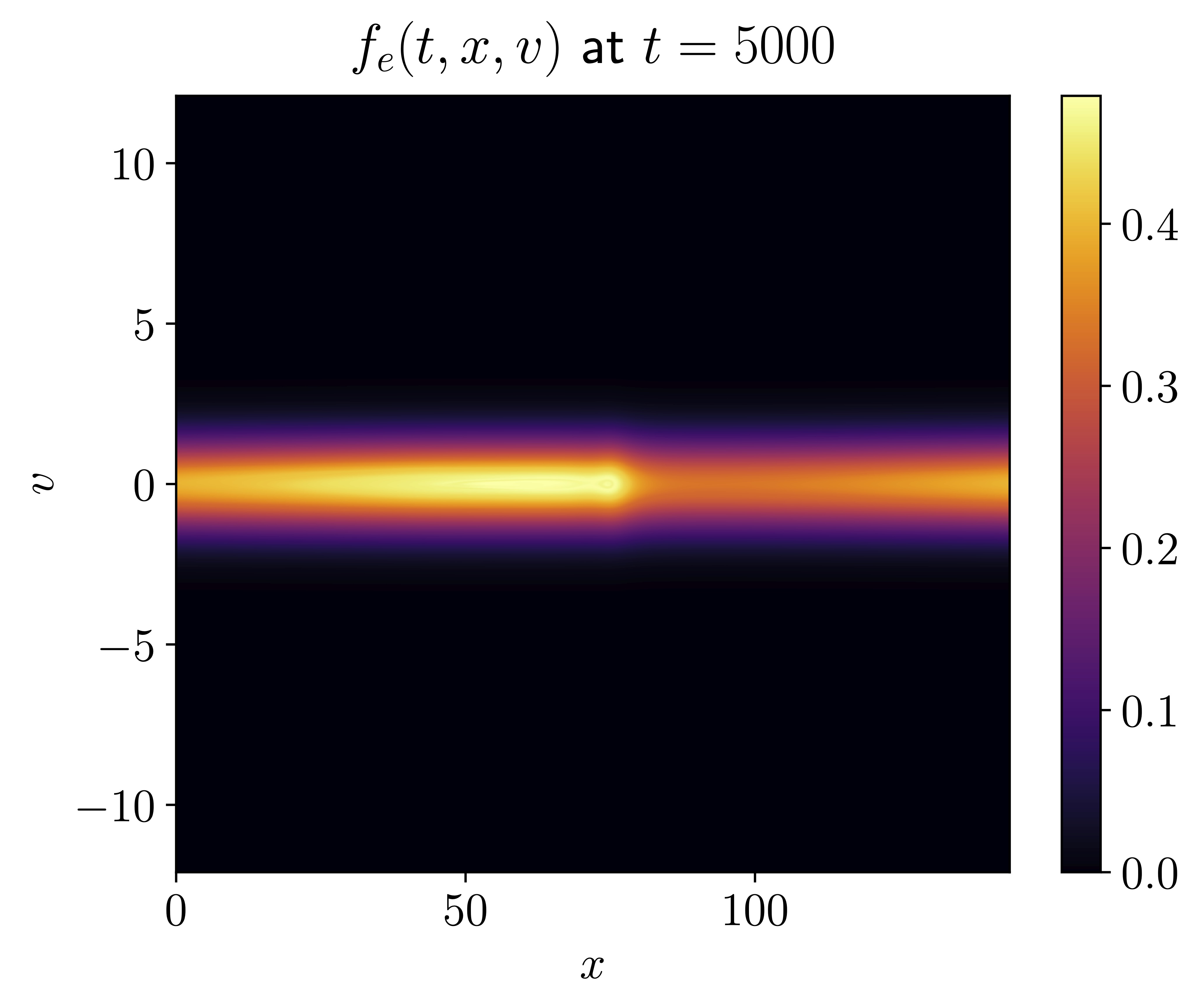}
\end{tabular}
\caption{(\S\ref{subsec:iasw}: Ion-acoustic shock) Phase space plots of the ion and electron distribution functions: $f_i(t,x,v)$ and $f_e(t,x,v)$. Panels (a) and (b) show the ion and electron distribution functions at $t\!=\!0$, respectively; Panels (c) and (d) show the ion and electron distribution functions at $t=5000$, respectively. This simulation uses fourth-order operator-splitting with the $\morder\!=\!4$ SLDG method, $N_x\times N_v\!=\!256\times512$, and $\mathrm{CFL}\!=\!10$. The positivity-preserving limiter maintains the non-negativity of the ion distribution during shock formation.}
\label{fig:ias-dist}
\end{center}
\end{figure}

\begin{figure}
\begin{center}
\begin{tabular}{cc}
   (a)\includegraphics[width=0.44\textwidth]{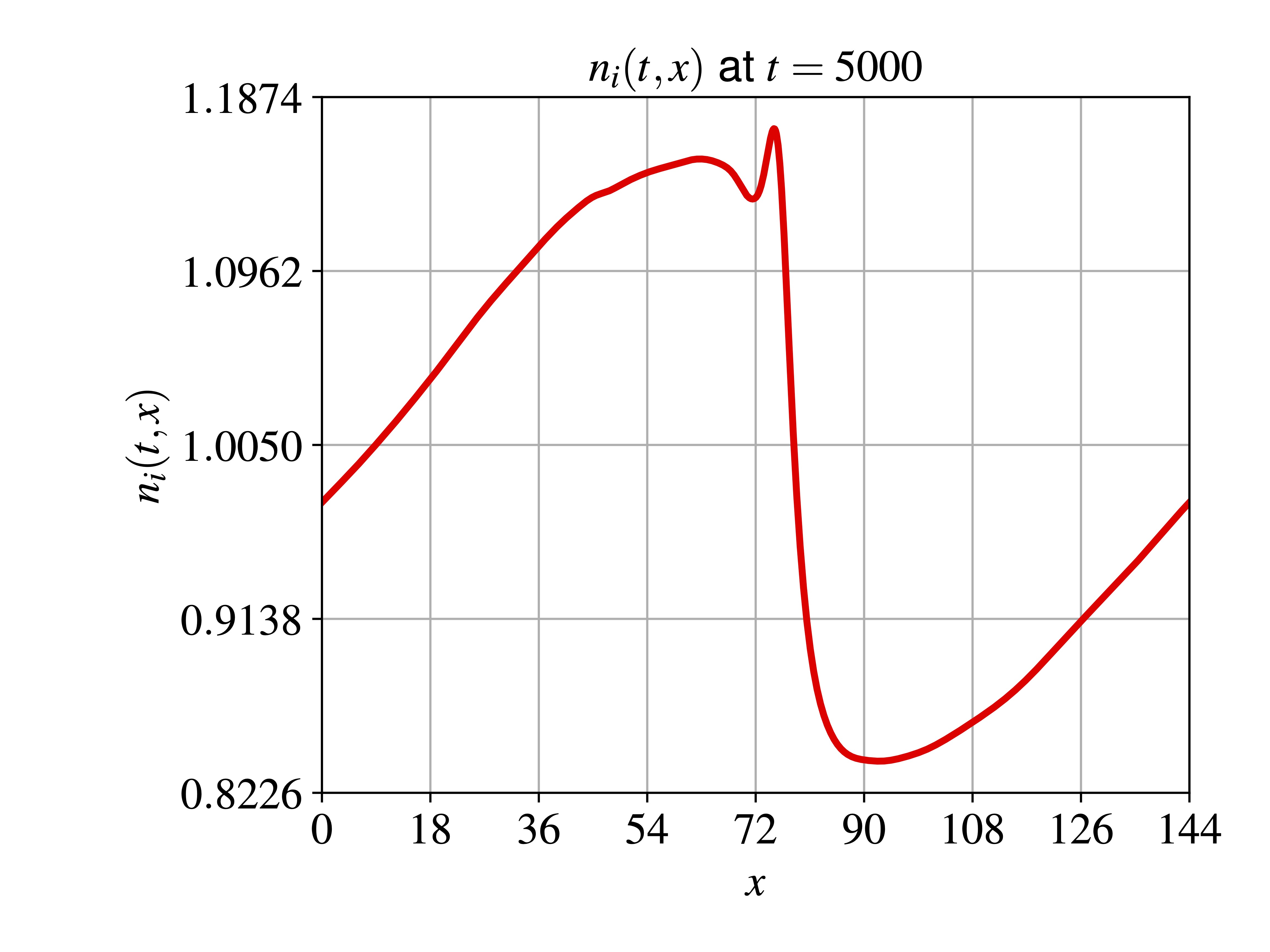} &
   (b)\includegraphics[width=0.44\textwidth]{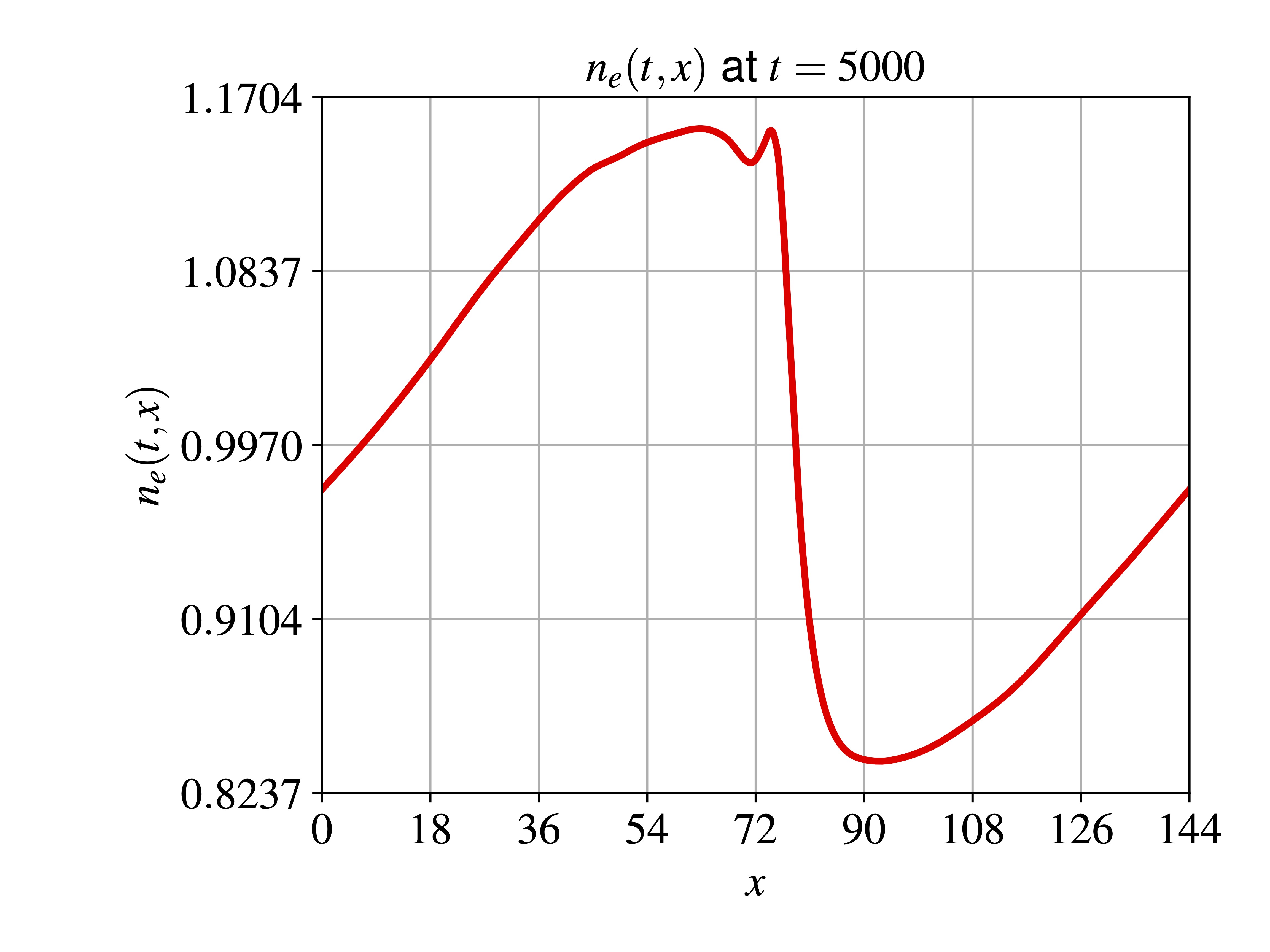} \\
   (c)\includegraphics[width=0.44\textwidth]{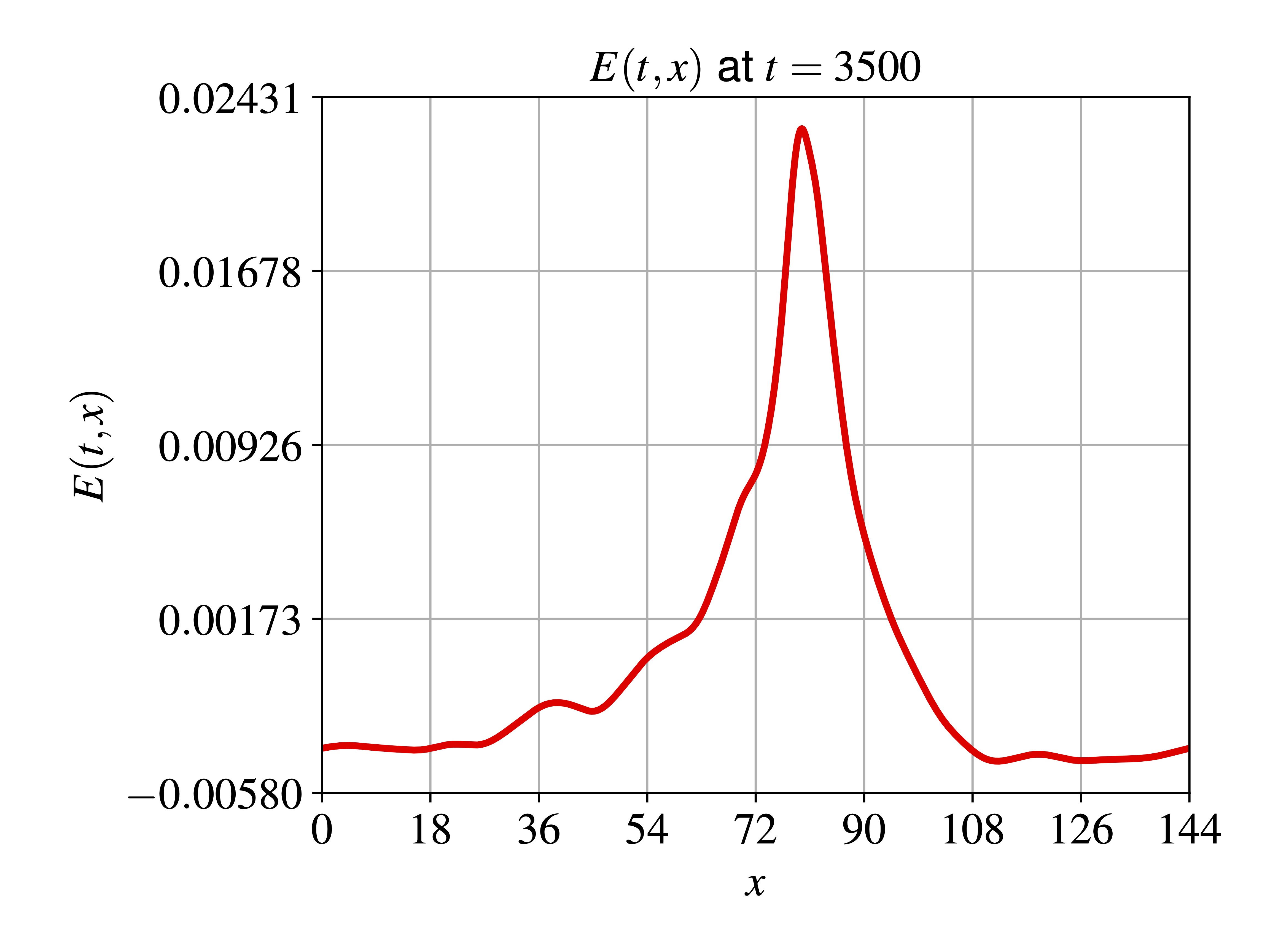} &
   (d)\includegraphics[width=0.44\textwidth]{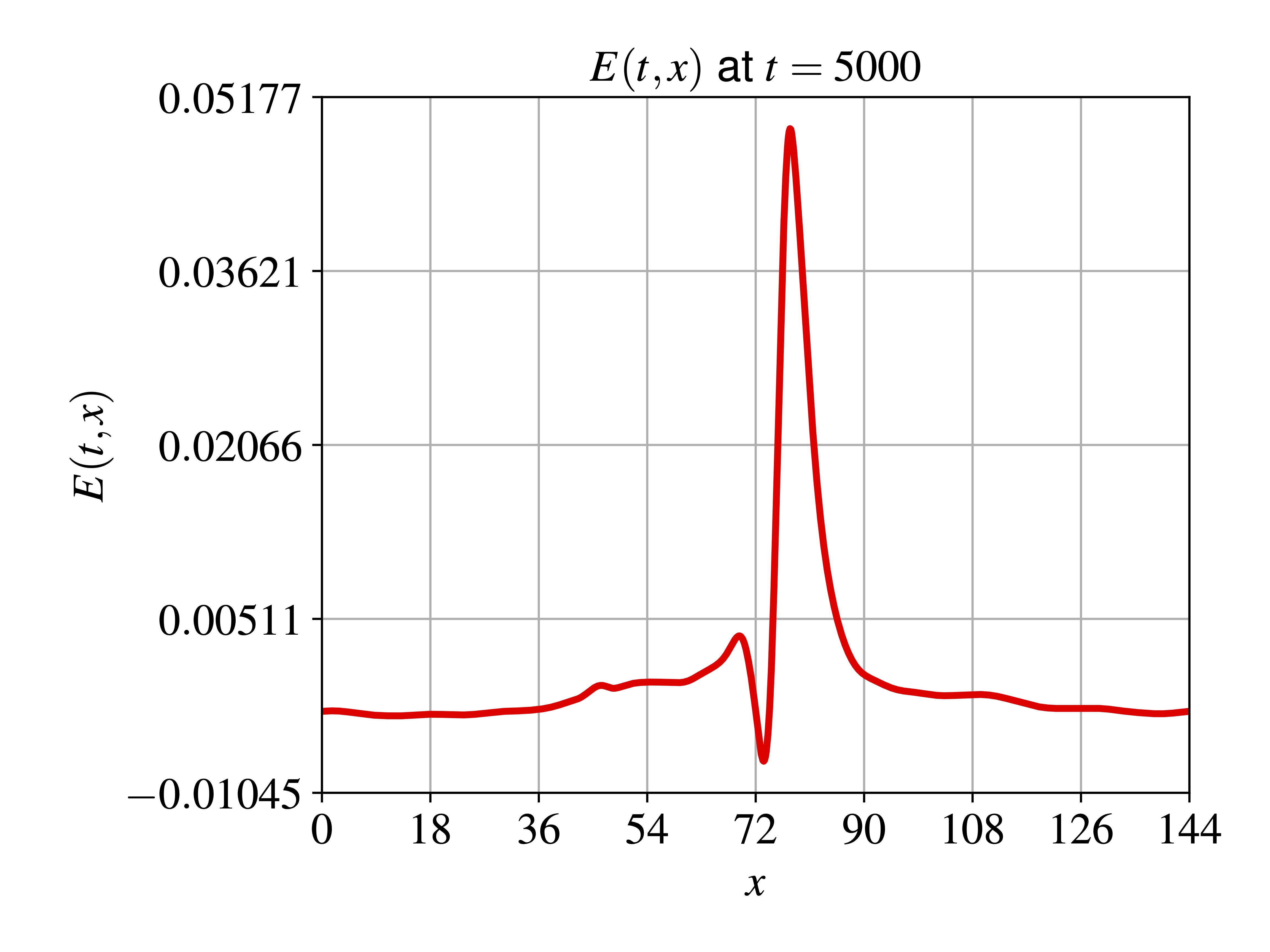} \\
   (e)\includegraphics[width=0.44\textwidth]{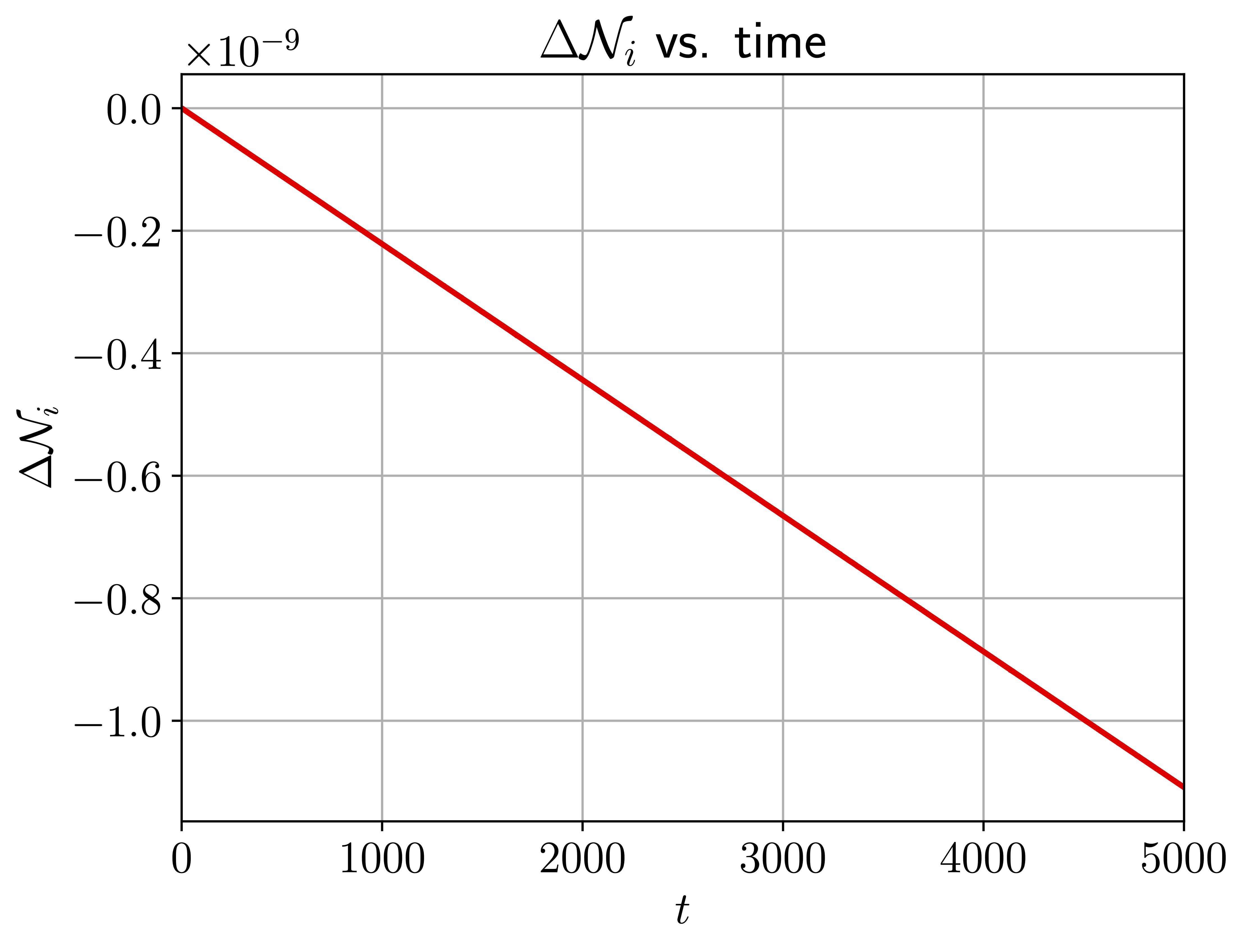} &
   (f)\includegraphics[width=0.44\textwidth]{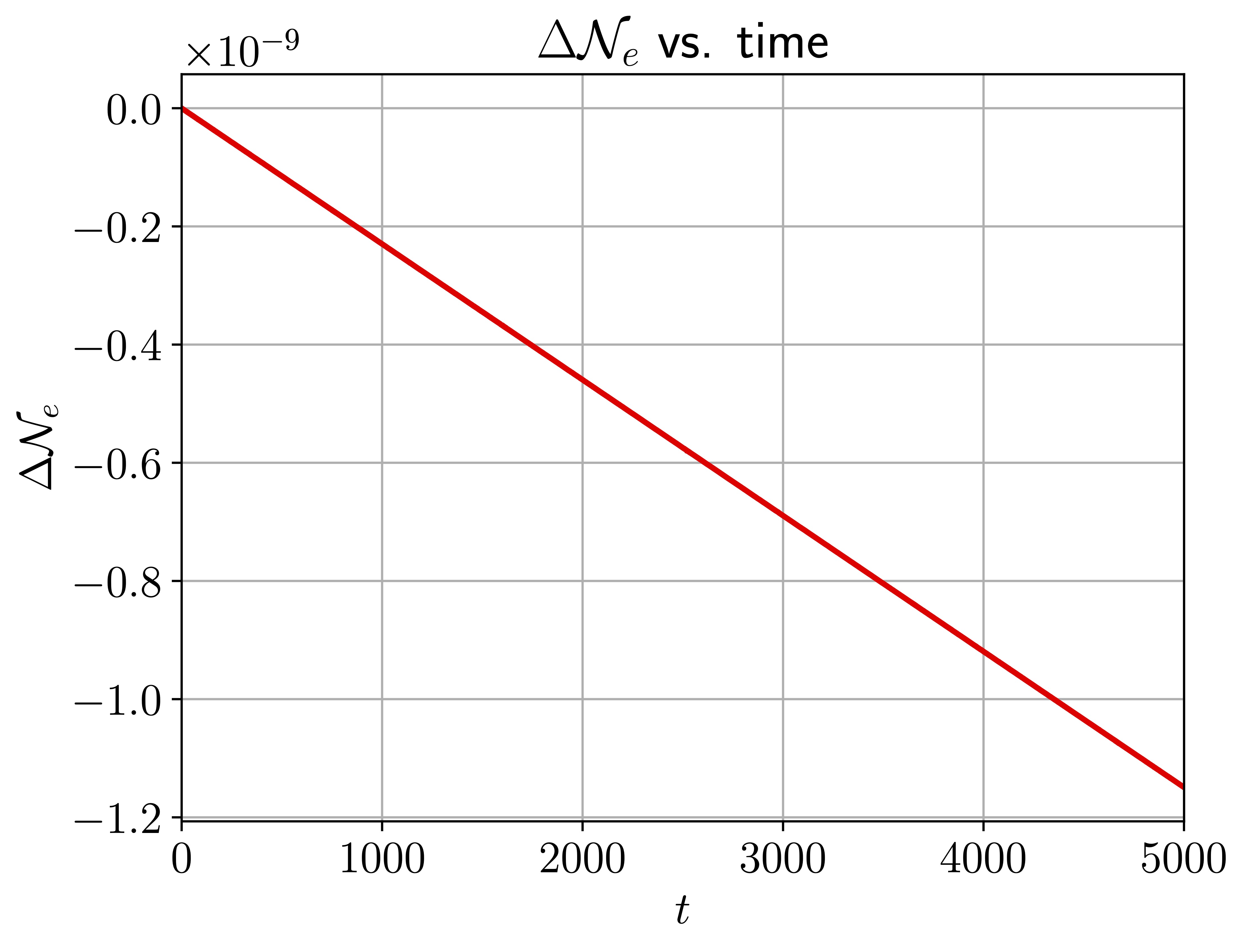}
\end{tabular}
    \caption{(\S\ref{subsec:iasw}: Ion-acoustic shock) 
    Shown in these panels are the (a) ion number density at $t\!=\!5000$, (b) electron number density at $t\!=\!5000$, (c) 
    electric field at $t\!=\!3500$, (d) electric field at $t\!=\!5000$, (e) time history of the ion particle number drift over $t\in[0,5000]$, and (f) time history of the electron particle number drift over $t\in[0,5000]$. The results show the formation of steep electric field gradients associated with the ion-acoustic shock.}
\label{fig:ias-density}
\end{center}
\end{figure}

\subsubsection{Kinetic electrostatic electron nonlinear (KEEN) wave}
\label{subsec:keen}
Kinetic electrostatic electron nonlinear (KEEN) waves serve as a demanding benchmark for assessing the accuracy of numerical methods over long time scales in nonlinear plasma simulations (e.g., see \cite{AfeyanEtAl2014}, \cite{WilhelmEtAl2026}). A traveling ponderomotive drive, representing the forcing from crossing laser beams, induces electron trapping and interactions among wave harmonics. The resulting fine phase space structures require high resolution over long times, while dynamical ions introduce a slower response that modifies the electron evolution.

The initial distribution for each species is a zero-drift Maxwellian, with zero initial electric field:
\begin{equation}\label{eq:keen_ic}
f_s(t\!=\!0,x,v)=\sqrt{\frac{m_s}{2\pi T_s}}\,
\exp\!\left(-\frac{m_s v^2}{2T_s}\right),
\qquad E(t\!=\!0,x)=0.
\end{equation}
Adopting the mass and temperature ratios of the dynamical ion setup in \cite{WilhelmEtAl2026}, and using our charge convention, we choose
\begin{equation}\label{eq:keen_params}
m_i=1836,\quad T_i=1,\quad q_i=1,\qquad
m_e=1,\quad T_e=1,\quad q_e=-1,
\end{equation}
corresponding to $r_m=1836$ and $r_T=1$. The prescribed drive field represents the laser-induced ponderomotive forcing:
\begin{gather}\label{eq:keen_drive}
E_{\mathrm{pond}}(x,t)
=a_{\mathrm{Dr}}k_{\mathrm{Dr}}a(t)
\sin\!\left(k_{\mathrm{Dr}}x-\omega_{\mathrm{Dr}}t\right), \quad
a(t)=\frac{g(t)-g(0)}{1-g(0)}, \\ \text{and} \quad
g(t)=\frac{1}{2}\left[
\tanh\!\left(\frac{t-t_L}{t_{wL}}\right)
-\tanh\!\left(\frac{t-t_R}{t_{wR}}\right)\right].
\end{gather}
The drive parameters are
\begin{equation}\label{eq:keen_drive_params}
a_{\mathrm{Dr}}=0.2,\quad
k_{\mathrm{Dr}}=0.26,\quad
\omega_{\mathrm{Dr}}=0.37,\quad
T_{\mathrm{Dr}}=100,\quad
t_L=69,\quad
t_{wL}=t_{wR}=20,\quad
t_R=307.
\end{equation}

For this model, we incorporate the ponderomotive forcing into the existing multi-species model by replacing $E$ with $\widetilde{E}$ in the force terms of the Vlasov equations in \eqref{eq:va_two_species}, where the effective field is defined by
\begin{equation}\label{eq:keen_effective_field}
\widetilde{E}=E-E_{\mathrm{pond}}.
\end{equation}

The existing splitting formulation accommodates this forcing without changing the splitting structure or introducing additional subproblems. 
To accommodate the separation between electron and ion velocity scales, the two species are evolved on different velocity intervals. We consider
the KEEN wave simulation with dynamically evolving ions on the computational domains $(t,x,v)\in[0,2000]\times[0,L]\times I_s$, with periodic boundary conditions in $x$, where $L=24.1$, $I_e=[-6,\,6]$, and $I_i=[-0.140028,\,0.140028]$. This setup allows electron trapping and phase space filamentation to be studied alongside the slower ion response and its influence on the long-time wave evolution. The simulation is ran to a final time of $t=2000$ using fourth-order operator-splitting with the $\morder\!=\!4$ SLDG method, $N_x\times N_v\!=\!512\times256$ elements, and $\mathrm{CFL}\!=\!20$.%

To track the development of spatial density variations, we compute the first five Fourier harmonics of each species. The amplitude of the
$m$th harmonic is defined by
\begin{equation}
\label{eq:keen_density_harmonics}
    \left|\widehat{n}_{s,m}(t)\right|
    :=
    \left|
    \frac{1}{L}\int_0^L n_s(t,x)
    e^{-\mathrm{i}m k_{\mathrm{Dr}}x}\,dx
    \right|,
    \quad m=1,\ldots,5,
\end{equation}
where $n_s(t,x)$ is the species number density
and $k_{\mathrm{Dr}}=2\pi/L$ is the fundamental wavenumber. 

Figure~\ref{fig:keen-wave-dist} shows the electron distribution deviation $f_e(t,x,v)-f_e(0,x,v)$ at varying times; Panel (a) depicts the pronounced filamented structures formed at $t=300$, and Panel (b), on the other hand, shows a coherent vortex at $t=1000$. These features are consistent with electron trapping and the nonlinear structures described in \cite{AfeyanEtAl2014} and \cite{WilhelmEtAl2026}. 
The ion and electron density harmonics are shown in Panels (c) and (d) of Figure~\ref{fig:keen-wave-dist}, respectively. Panels (c) and (d) show initial growth and oscillations of the density harmonics, followed by persistent, approximately steady amplitudes. 

These figures can be reproduced by running the commands below \cite{dogpack}:

\begin{tcolorbox}
\begin{verbatim}
cd $DOGPACK/apps/2d/keen_wave/dynamic_ion_test/; make; dog.exe; 
make plotfull; make plotmodes;
\end{verbatim}
\end{tcolorbox}

\begin{figure}[!th]
\begin{center}
\begin{tabular}{cc}
   (a)\includegraphics[width=0.44\textwidth]{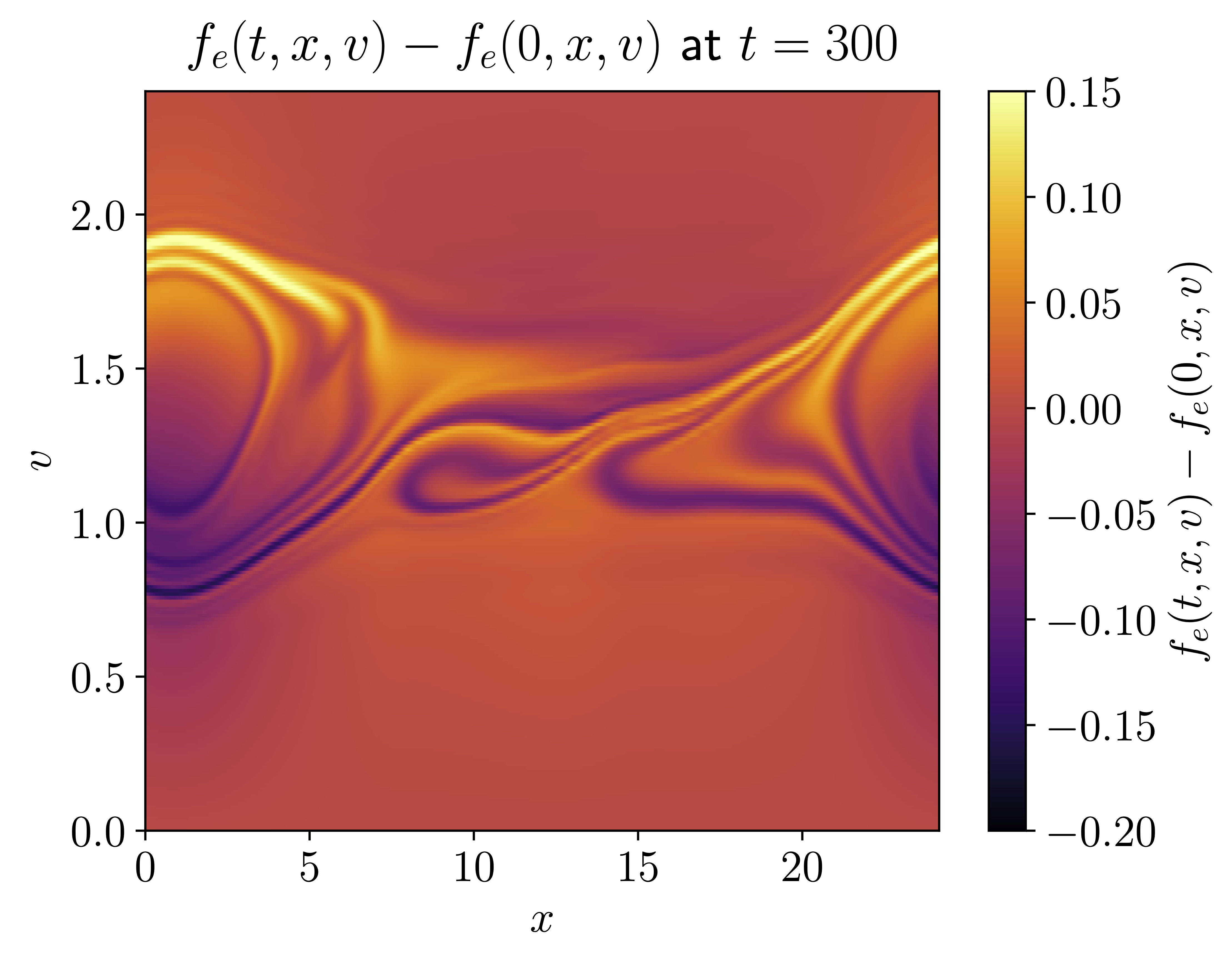}&
   (b)\includegraphics[width=0.44\textwidth]{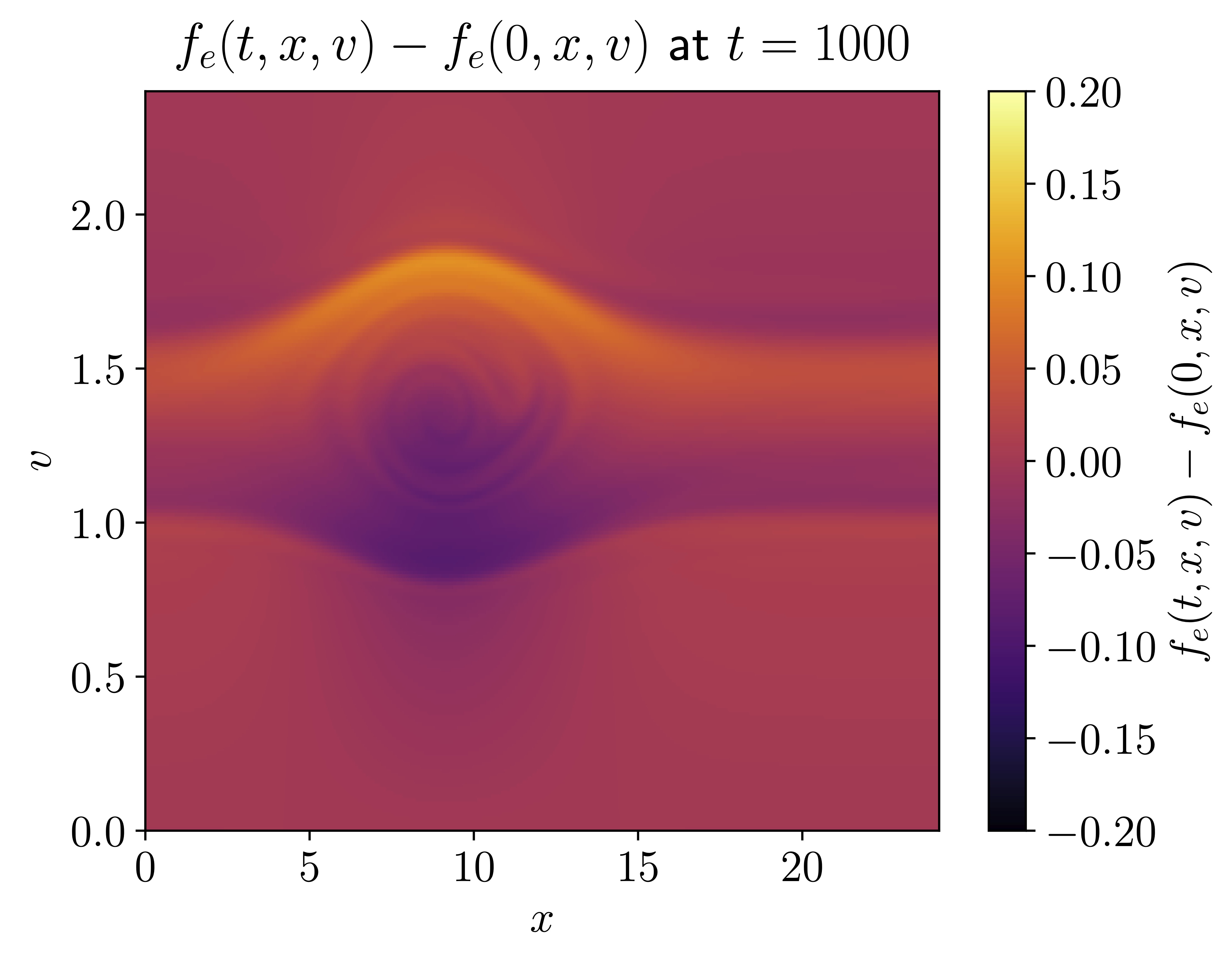} \\
   (c)\includegraphics[width=0.44\textwidth]{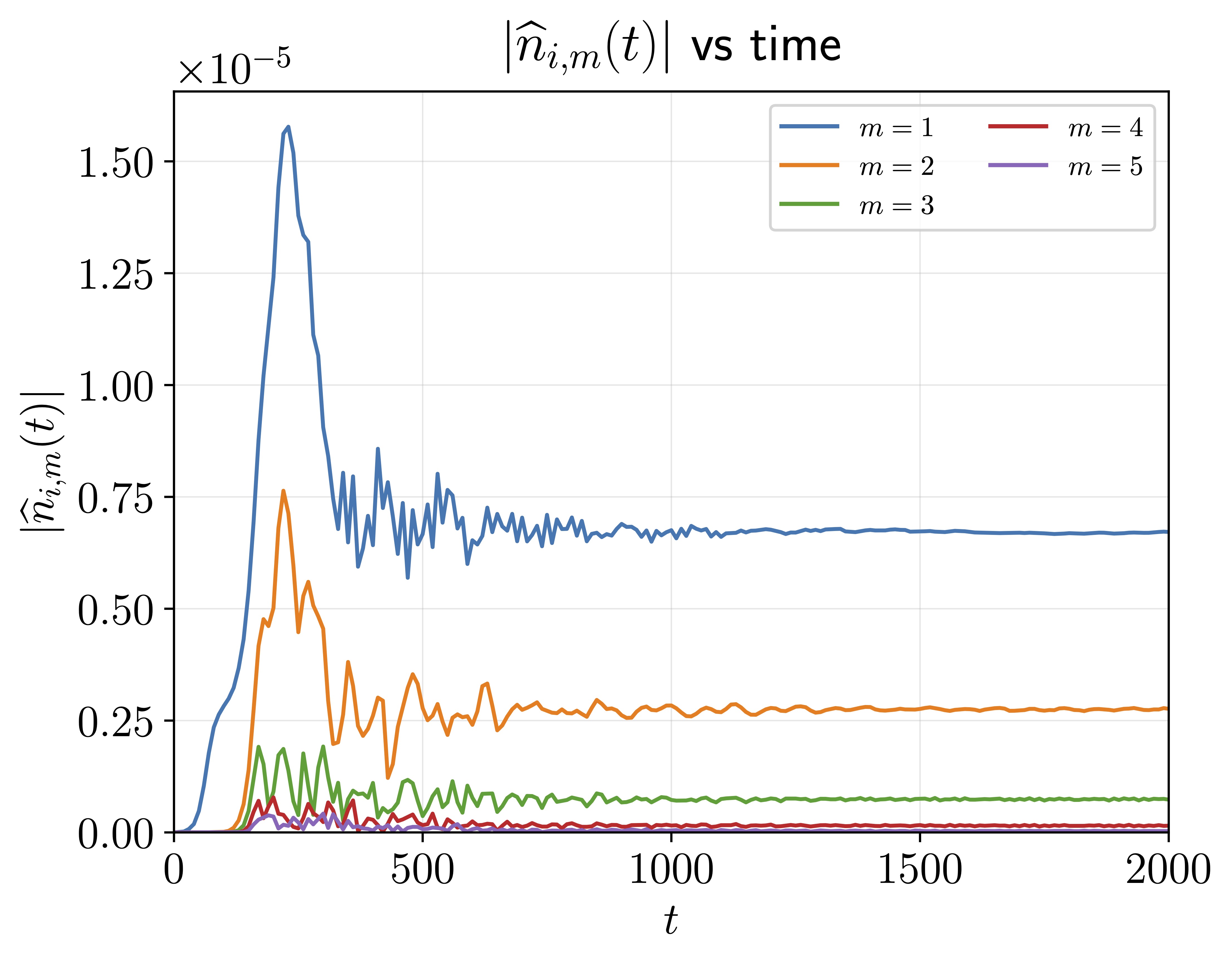}&
   (d)\includegraphics[width=0.44\textwidth]{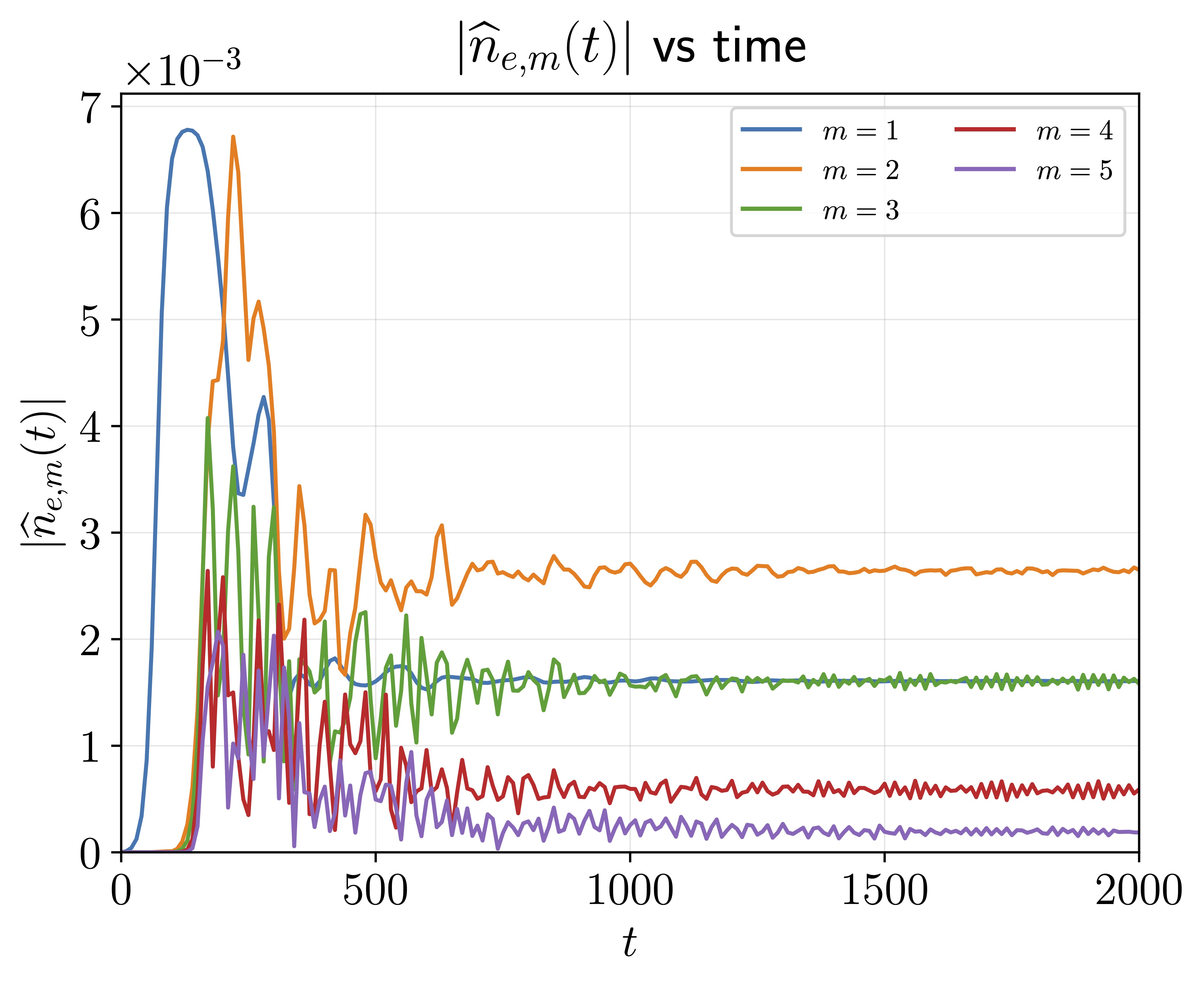}
   \end{tabular}
    \caption{(\S\ref{subsec:keen}: KEEN wave): Phase space plots of deviation of electron distribution $f_e(t,x,v)-f_e(0,x,v)$ at (a) $t=300$ and (b) $t=1000$. 
    Also shown are time series of the first five spatial Fourier modes over $t\in[0,2000]$ of the (c) ion number density and (d) electron number density. 
    The simulation uses fourth-order operator splitting and the $\morder\!=\!4$ SLDG method, $N_x\times N_v\!=\!512\times256$ elements, and $\mathrm{CFL}\!=\!20$.}
\label{fig:keen-wave-dist}
\end{center}
\end{figure}

\section{Conclusion}
\label{sec:conclusion}
In this work, we proposed a high-order semi-Lagrangian discontinuous Galerkin framework for the multi-species 1D1V Vlasov-Amp\`ere system. The scheme is mass-conserving, positivity-preserving, and unconditionally stable. The method makes use of high-order operator splitting coupled to semi-Lagrangian DG transport solvers for the free-streaming and acceleration substeps. The electric field is computed from exaclty solving a harmonic oscillator equation in the acceleration substep. For improved efficiency, each species is discretized on a different phase-space mesh.

Manufactured solution tests were used to verify the designed order of accuracy in both the
single- and two-species settings. In addition, the method was tested on a suite of standard plasma
benchmarks. For the single-species case, we considered the two-stream instability and both weak and strong Landau damping, while for the two-species case we studied the ion-acoustic wave, ion-acoustic shock, and the KEEN wave. Collectively, these numerical experiments confirmed long-time robustness in regimes with strong ion-electron scale separation and demonstrated that the method preserves species-wise particle number and positivity. While not exactly momentum or total energy conserving, the scheme is shown to preserve these quantities to relatively high accuracy.

Future work will extend the present framework to high-dimensional Vlasov-Amp\`ere systems. We
also plan to adapt the methodology to the nonrelativistic and relativistic Vlasov-Maxwell
equations, while preserving the same accuracy, efficiency, and robustness properties. %

\printcredits

\section*{Statements and Declarations}
\begin{description}
\item[{\bf Funding.}] This research was partially funded by US National Science Foundation Grant DMS--2410538. DG was supported by Simons Foundation International Grant SFI-MPS-SDF-00026661.

\medskip

\item[{\bf Competing interests.}] The authors have no
conflicts of interest to disclose.

\medskip

\item[{\bf Data availability statement.}] Data sharing does not apply to this article as no datasets were generated or analyzed during the current study.
\end{description}

\bibliographystyle{cas-model2-names}

\clearpage

\begin{thebibliography}{66}
\expandafter\ifx\csname natexlab\endcsname\relax\def\natexlab#1{#1}\fi
\providecommand{\url}[1]{\texttt{#1}}
\providecommand{\href}[2]{#2}
\providecommand{\path}[1]{#1}
\providecommand{\DOIprefix}{doi:}
\providecommand{\ArXivprefix}{arXiv:}
\providecommand{\URLprefix}{URL: }
\providecommand{\Pubmedprefix}{pmid:}
\providecommand{\doi}[1]{\href{http://dx.doi.org/#1}{\path{#1}}}
\providecommand{\Pubmed}[1]{\href{pmid:#1}{\path{#1}}}
\providecommand{\bibinfo}[2]{#2}
\ifx\xfnm\relax \def\xfnm[#1]{\unskip,\space#1}\fi
%
\bibitem[{Afeyan et~al.(2014)Afeyan, Casas, Crouseilles, Dodhy, Faou,
  Mehrenberger and Sonnendr{\"u}cker}]{AfeyanEtAl2014}
\bibinfo{author}{Afeyan, B.}, \bibinfo{author}{Casas, F.},
  \bibinfo{author}{Crouseilles, N.}, \bibinfo{author}{Dodhy, A.},
  \bibinfo{author}{Faou, E.}, \bibinfo{author}{Mehrenberger, M.},
  \bibinfo{author}{Sonnendr{\"u}cker, E.}, \bibinfo{year}{2014}.
\newblock \bibinfo{title}{Simulations of kinetic electrostatic electron
  nonlinear ({KEEN}) waves with variable velocity resolution grids and
  high-order time-splitting}.
\newblock \bibinfo{journal}{The European Physical Journal D}
  \bibinfo{volume}{68}, \bibinfo{pages}{295}.
\newblock \DOIprefix\doi{10.1140/epjd/e2014-50212-6}.
%
\bibitem[{Banks and Hittinger(2010)}]{BanksHittinger2010}
\bibinfo{author}{Banks, J.W.}, \bibinfo{author}{Hittinger, J.A.F.},
  \bibinfo{year}{2010}.
\newblock \bibinfo{title}{A new class of nonlinear finite-volume methods for
  {V}lasov simulation}.
\newblock \bibinfo{journal}{IEEE Transactions on Plasma Science}
  \bibinfo{volume}{38}, \bibinfo{pages}{2198--2207}.
\newblock \DOIprefix\doi{10.1109/TPS.2010.2056937}.
%
\bibitem[{Banks et~al.(2019)Banks, Odu, Berger, Chapman, Arrighi and
  Brunner}]{BanksOduBergerChapmanArrighiBrunner2019}
\bibinfo{author}{Banks, J.W.}, \bibinfo{author}{Odu, A.G.},
  \bibinfo{author}{Berger, R.}, \bibinfo{author}{Chapman, T.},
  \bibinfo{author}{Arrighi, W.}, \bibinfo{author}{Brunner, S.},
  \bibinfo{year}{2019}.
\newblock \bibinfo{title}{High-order accurate conservative finite difference
  methods for {V}lasov equations in {2D+2V}}.
\newblock \bibinfo{journal}{SIAM Journal on Scientific Computing}
  \bibinfo{volume}{41}, \bibinfo{pages}{B953--B982}.
\newblock \DOIprefix\doi{doi.org/10.1137/19M1238551}.
%
\bibitem[{Besse and Mehrenberger(2008)}]{BesseMehrenberger2008}
\bibinfo{author}{Besse, N.}, \bibinfo{author}{Mehrenberger, M.},
  \bibinfo{year}{2008}.
\newblock \bibinfo{title}{Convergence of classes of high-order
  semi-{L}agrangian schemes for the {V}lasov--{P}oisson system}.
\newblock \bibinfo{journal}{Mathematics of Computation} \bibinfo{volume}{77},
  \bibinfo{pages}{93--123}.
\newblock \DOIprefix\doi{10.1090/S0025-5718-07-01912-6}.
%
\bibitem[{Birdsall et~al.(1991)Birdsall, Langdon and
  Langdon}]{BirdsallLangdon2004}
\bibinfo{author}{Birdsall, C.K.}, \bibinfo{author}{Langdon, A.B.},
  \bibinfo{author}{Langdon, A.B.}, \bibinfo{year}{1991}.
\newblock \bibinfo{title}{Plasma Physics via Computer Simulation}.
\newblock \bibinfo{edition}{1st} ed., \bibinfo{publisher}{CRC Press}.
\newblock \DOIprefix\doi{doi.org/10.1201/9781315275048}.
%
\bibitem[{Blanes and Moan(2002)}]{BlanesMoan2002}
\bibinfo{author}{Blanes, S.}, \bibinfo{author}{Moan, P.C.},
  \bibinfo{year}{2002}.
\newblock \bibinfo{title}{Practical symplectic partitioned {R}unge--{K}utta and
  {R}unge--{K}utta--{N}ystr{\"o}m methods}.
\newblock \bibinfo{journal}{Journal of Computational and Applied Mathematics}
  \bibinfo{volume}{142}, \bibinfo{pages}{313--330}.
\newblock \DOIprefix\doi{10.1016/S0377-0427(01)00492-7}.
%
\bibitem[{Cai et~al.(2021)Cai, Boscarino and Qiu}]{CaiBoscarinoQiu2021}
\bibinfo{author}{Cai, X.}, \bibinfo{author}{Boscarino, S.},
  \bibinfo{author}{Qiu, J.M.}, \bibinfo{year}{2021}.
\newblock \bibinfo{title}{High order semi-{L}agrangian discontinuous {G}alerkin
  method coupled with {R}unge--{K}utta exponential integrators for nonlinear
  {V}lasov dynamics}.
\newblock \bibinfo{journal}{Journal of Computational Physics}
  \bibinfo{volume}{427}, \bibinfo{pages}{110036}.
\newblock \DOIprefix\doi{10.1016/j.jcp.2020.110036}.
%
\bibitem[{Cai et~al.(2018)Cai, Guo and Qiu}]{CaiGuoQiu2018}
\bibinfo{author}{Cai, X.}, \bibinfo{author}{Guo, W.}, \bibinfo{author}{Qiu,
  J.M.}, \bibinfo{year}{2018}.
\newblock \bibinfo{title}{A high order semi-{L}agrangian discontinuous
  {G}alerkin method for {V}lasov--{P}oisson simulations without operator
  splitting}.
\newblock \bibinfo{journal}{Journal of Computational Physics}
  \bibinfo{volume}{354}, \bibinfo{pages}{529--551}.
\newblock \DOIprefix\doi{doi.org/10.1016/j.jcp.2017.10.048}.
%
\bibitem[{Cai et~al.(2019)Cai, Guo and Qiu}]{CaiGuoQiu2019}
\bibinfo{author}{Cai, X.}, \bibinfo{author}{Guo, W.}, \bibinfo{author}{Qiu,
  J.M.}, \bibinfo{year}{2019}.
\newblock \bibinfo{title}{A high order semi-{L}agrangian discontinuous
  {G}alerkin method for the two-dimensional incompressible {E}uler equations
  and the guiding center {V}lasov model without operator splitting}.
\newblock \bibinfo{journal}{Journal of Scientific Computing}
  \bibinfo{volume}{79}, \bibinfo{pages}{1111--1134}.
\newblock \DOIprefix\doi{10.1007/s10915-018-0889-1}.
%
\bibitem[{Chen et~al.(2011)Chen, Chac{\'o}n and Barnes}]{ChenChaconBarnes2011}
\bibinfo{author}{Chen, G.}, \bibinfo{author}{Chac{\'o}n, L.},
  \bibinfo{author}{Barnes, D.C.}, \bibinfo{year}{2011}.
\newblock \bibinfo{title}{An energy- and charge-conserving, implicit,
  electrostatic particle-in-cell algorithm}.
\newblock \bibinfo{journal}{Journal of Computational Physics}
  \bibinfo{volume}{230}, \bibinfo{pages}{7018--7036}.
\newblock \DOIprefix\doi{10.1016/j.jcp.2011.05.031}.
%
\bibitem[{Cheng and Knorr(1976)}]{ChengKnorr1976}
\bibinfo{author}{Cheng, C.Z.}, \bibinfo{author}{Knorr, G.},
  \bibinfo{year}{1976}.
\newblock \bibinfo{title}{The integration of the {V}lasov equation in
  configuration space}.
\newblock \bibinfo{journal}{Journal of Computational Physics}
  \bibinfo{volume}{22}, \bibinfo{pages}{330--351}.
\newblock \DOIprefix\doi{doi.org/10.1016/0021-9991(76)90053-X}.
%
\bibitem[{Cheng et~al.(2015)Cheng, Christlieb and
  Zhong}]{ChengChristliebZhong2015_TwoSpeciesVA}
\bibinfo{author}{Cheng, Y.}, \bibinfo{author}{Christlieb, A.J.},
  \bibinfo{author}{Zhong, X.}, \bibinfo{year}{2015}.
\newblock \bibinfo{title}{Numerical study of the two-species
  {V}lasov--{A}mp{\`e}re system: Energy-conserving schemes and the
  current-driven ion-acoustic instability}.
\newblock \bibinfo{journal}{Journal of Computational Physics}
  \bibinfo{volume}{288}, \bibinfo{pages}{66--85}.
\newblock \DOIprefix\doi{10.1016/j.jcp.2015.02.020}.
%
\bibitem[{Cockburn and Shu(1998)}]{CockburnShu1998_JCP_V}
\bibinfo{author}{Cockburn, B.}, \bibinfo{author}{Shu, C.W.},
  \bibinfo{year}{1998}.
\newblock \bibinfo{title}{The {R}unge--{K}utta discontinuous {G}alerkin method
  for conservation laws {V}: Multidimensional systems}.
\newblock \bibinfo{journal}{Journal of Computational Physics}
  \bibinfo{volume}{141}, \bibinfo{pages}{199--224}.
\newblock \DOIprefix\doi{10.1006/jcph.1998.5892}.
%
\bibitem[{Courant et~al.(1928)Courant, Friedrichs and Lewy}]{article:CFL1928}
\bibinfo{author}{Courant, R.}, \bibinfo{author}{Friedrichs, K.},
  \bibinfo{author}{Lewy, H.}, \bibinfo{year}{1928}.
\newblock \bibinfo{title}{{\"U}ber die partiellen {D}ifferenzengleichungen der
  mathematischen {P}hysik}.
\newblock \bibinfo{journal}{Mathematische Annalen} \bibinfo{volume}{100},
  \bibinfo{pages}{32--74}.
\newblock \DOIprefix\doi{10.1007/BF01448839}.
%
\bibitem[{Crouseilles et~al.(2011)Crouseilles, Faou and
  Mehrenberger}]{article:Crouseilles2011}
\bibinfo{author}{Crouseilles, N.}, \bibinfo{author}{Faou, E.},
  \bibinfo{author}{Mehrenberger, M.}, \bibinfo{year}{2011}.
\newblock \bibinfo{title}{High order {R}unge-{K}utta-{N}ystrom splitting
  methods for the {V}lasov-{P}oisson equation}.
\newblock \URLprefix
  \url{https://www.i2m.univ-amu.fr/perso/mehrenberg.m/cfm.pdf}.
%
\bibitem[{Crouseilles et~al.(2014)Crouseilles, Glanc, Hirstoaga, Madaule,
  Mehrenberger and
  P{\'e}tri}]{CrouseillesGlancHirstoagaMadauleMehrenbergerPetri2014}
\bibinfo{author}{Crouseilles, N.}, \bibinfo{author}{Glanc, P.},
  \bibinfo{author}{Hirstoaga, S.}, \bibinfo{author}{Madaule, E.},
  \bibinfo{author}{Mehrenberger, M.}, \bibinfo{author}{P{\'e}tri, J.},
  \bibinfo{year}{2014}.
\newblock \bibinfo{title}{A new fully two-dimensional conservative
  semi-{L}agrangian method: {A}pplications on polar grids, from diocotron
  instability to {ITG} turbulence}.
\newblock \bibinfo{journal}{European Physical Journal D} \bibinfo{volume}{68},
  \bibinfo{pages}{1--10}.
\newblock \DOIprefix\doi{10.1140/epjd/e2014-50180-9}.
%
\bibitem[{Crouseilles et~al.(2025)Crouseilles, Liu and
  Yue}]{CrouseillesLiuYue2025}
\bibinfo{author}{Crouseilles, N.}, \bibinfo{author}{Liu, H.},
  \bibinfo{author}{Yue, Y.}, \bibinfo{year}{2025}.
\newblock \bibinfo{title}{Semi-lagrangian {SAV} method for {V}lasov--{M}axwell
  equations}.
\newblock \bibinfo{journal}{Journal of Computational Physics}
  \bibinfo{volume}{549}, \bibinfo{pages}{114606}.
\newblock \DOIprefix\doi{10.1016/j.jcp.2025.114606}.
%
\bibitem[{Crouseilles et~al.(2010)Crouseilles, Mehrenberger and
  Sonnendr{\"u}cker}]{CrouseillesMehrenbergerSonnendrucker2010}
\bibinfo{author}{Crouseilles, N.}, \bibinfo{author}{Mehrenberger, M.},
  \bibinfo{author}{Sonnendr{\"u}cker, E.}, \bibinfo{year}{2010}.
\newblock \bibinfo{title}{Conservative semi-{L}agrangian schemes for {V}lasov
  equations}.
\newblock \bibinfo{journal}{Journal of Computational Physics}
  \bibinfo{volume}{229}, \bibinfo{pages}{1927--1953}.
\newblock \DOIprefix\doi{10.1016/j.jcp.2009.11.007}.
%
\bibitem[{Despr{\'e}s(2008)}]{Despres2008}
\bibinfo{author}{Despr{\'e}s, B.}, \bibinfo{year}{2008}.
\newblock \bibinfo{title}{Stability of high order finite volume schemes for the
  {1D} transport equation}, in: \bibinfo{booktitle}{Finite Volumes for Complex
  Applications V}. \bibinfo{publisher}{ISTE}, \bibinfo{address}{London}, pp.
  \bibinfo{pages}{337--342}.
\newblock \DOIprefix\doi{10.1016/j.matcom.2025.02.017}.
%
\bibitem[{Einkemmer and Joseph(2021)}]{EinkemmerJoseph2021}
\bibinfo{author}{Einkemmer, L.}, \bibinfo{author}{Joseph, I.},
  \bibinfo{year}{2021}.
\newblock \bibinfo{title}{A mass, momentum, and energy conservative dynamical
  low-rank scheme for the {V}lasov equation}.
\newblock \bibinfo{journal}{Journal of Computational Physics}
  \bibinfo{volume}{443}, \bibinfo{pages}{110495}.
\newblock \DOIprefix\doi{10.1016/j.jcp.2021.110495}.
%
\bibitem[{Einkemmer and Lubich(2018)}]{EinkemmerLubich2018}
\bibinfo{author}{Einkemmer, L.}, \bibinfo{author}{Lubich, C.},
  \bibinfo{year}{2018}.
\newblock \bibinfo{title}{A low-rank projector-splitting integrator for the
  {V}lasov--{P}oisson equation}.
\newblock \bibinfo{journal}{SIAM Journal on Scientific Computing}
  \bibinfo{volume}{40}, \bibinfo{pages}{B1330--B1360}.
\newblock \DOIprefix\doi{10.1137/18M116383X}.
%
\bibitem[{Einkemmer and Lubich(2019)}]{EinkemmerLubich2019}
\bibinfo{author}{Einkemmer, L.}, \bibinfo{author}{Lubich, C.},
  \bibinfo{year}{2019}.
\newblock \bibinfo{title}{A quasi-conservative dynamical low-rank algorithm for
  the {V}lasov equation}.
\newblock \bibinfo{journal}{SIAM Journal on Scientific Computing}
  \bibinfo{volume}{41}, \bibinfo{pages}{B1061--B1081}.
\newblock \DOIprefix\doi{10.1137/18M1218686}.
%
\bibitem[{Einkemmer et~al.(2020)Einkemmer, Ostermann and
  Piazzola}]{EinkemmerOstermannPiazzola2020}
\bibinfo{author}{Einkemmer, L.}, \bibinfo{author}{Ostermann, A.},
  \bibinfo{author}{Piazzola, C.}, \bibinfo{year}{2020}.
\newblock \bibinfo{title}{A low-rank projector-splitting integrator for the
  {V}lasov--{M}axwell equations with divergence correction}.
\newblock \bibinfo{journal}{Journal of Computational Physics}
  \bibinfo{volume}{403}, \bibinfo{pages}{109063}.
\newblock \DOIprefix\doi{10.1016/j.jcp.2019.109063}.
%
\bibitem[{Filbet and Sonnendr{\"u}cker(2003)}]{FilbetSonnendrucker2003}
\bibinfo{author}{Filbet, F.}, \bibinfo{author}{Sonnendr{\"u}cker, E.},
  \bibinfo{year}{2003}.
\newblock \bibinfo{title}{Comparison of {E}ulerian {V}lasov solvers}.
\newblock \bibinfo{journal}{Computer Physics Communications}
  \bibinfo{volume}{150}, \bibinfo{pages}{247--266}.
\newblock \DOIprefix\doi{10.1016/S0010-4655(02)00694-X}.
%
\bibitem[{Filbet et~al.(2001)Filbet, Sonnendr{\"u}cker and
  Bertrand}]{FilbetSonnendruckerBertrand2001}
\bibinfo{author}{Filbet, F.}, \bibinfo{author}{Sonnendr{\"u}cker, E.},
  \bibinfo{author}{Bertrand, P.}, \bibinfo{year}{2001}.
\newblock \bibinfo{title}{Conservative numerical schemes for the {V}lasov
  equation}.
\newblock \bibinfo{journal}{Journal of Computational Physics}
  \bibinfo{volume}{172}, \bibinfo{pages}{166--187}.
\newblock \DOIprefix\doi{10.1006/jcph.2001.6818}.
%
\bibitem[{Finn et~al.(2023)Finn, Knepley, Pusztay and Adams}]{article:Finn2023}
\bibinfo{author}{Finn, D.S.}, \bibinfo{author}{Knepley, M.G.},
  \bibinfo{author}{Pusztay, J.V.}, \bibinfo{author}{Adams, M.F.},
  \bibinfo{year}{2023}.
\newblock \bibinfo{title}{A numerical study of {L}andau damping with
  {PETSc-PIC}}.
\newblock \bibinfo{journal}{Communications in Applied Mathematics and
  Computational Science} \bibinfo{volume}{18}, \bibinfo{pages}{135--152}.
\newblock \DOIprefix\doi{10.2140/camcos.2023.18.135}.
%
\bibitem[{Fitzpatrick(2023)}]{link:Fitzpatrick2023}
\bibinfo{author}{Fitzpatrick, R.}, \bibinfo{year}{2023}.
\newblock \bibinfo{title}{Plasma physics}.
\newblock \URLprefix
  \url{https://farside.ph.utexas.edu/teaching/plasma/lectures1/node8.html}.
%
\bibitem[{Forest and Ruth(1990)}]{article:ForestRuth1990}
\bibinfo{author}{Forest, E.}, \bibinfo{author}{Ruth, R.D.},
  \bibinfo{year}{1990}.
\newblock \bibinfo{title}{Fourth-order symplectic integration}.
\newblock \bibinfo{journal}{Physica D: Nonlinear Phenomena}
  \bibinfo{volume}{43}, \bibinfo{pages}{105--117}.
\newblock \DOIprefix\doi{https://doi.org/10.1016/0167-2789(90)90019-L}.
%
\bibitem[{Giraldo(1998)}]{Giraldo1998}
\bibinfo{author}{Giraldo, F.X.}, \bibinfo{year}{1998}.
\newblock \bibinfo{title}{The {L}agrange--{G}alerkin spectral element method on
  unstructured quadrilateral grids}.
\newblock \bibinfo{journal}{Journal of Computational Physics}
  \bibinfo{volume}{147}, \bibinfo{pages}{114--146}.
\newblock \DOIprefix\doi{10.1006/jcph.1998.6078}.
%
\bibitem[{Golub and Welsch(1969)}]{article:GolubWelsch1969}
\bibinfo{author}{Golub, G.H.}, \bibinfo{author}{Welsch, J.H.},
  \bibinfo{year}{1969}.
\newblock \bibinfo{title}{Calculation of {G}auss quadrature rules}.
\newblock \bibinfo{journal}{Mathematics of Computation} \bibinfo{volume}{23},
  \bibinfo{pages}{221--s10}.
\newblock \DOIprefix\doi{10.2307/2004418}.
%
\bibitem[{G{\"u}{\c c}l{\"u} et~al.(2014)G{\"u}{\c c}l{\"u}, Christlieb and
  Hitchon}]{article:GucluChristliebHitchon2014}
\bibinfo{author}{G{\"u}{\c c}l{\"u}, Y.}, \bibinfo{author}{Christlieb, A.J.},
  \bibinfo{author}{Hitchon, W.N.G.}, \bibinfo{year}{2014}.
\newblock \bibinfo{title}{Arbitrarily high order {C}onvected {S}cheme solution
  of the {V}lasov-{P}oisson system}.
\newblock \bibinfo{journal}{Journal of Computational Physics}
  \bibinfo{volume}{270}, \bibinfo{pages}{711--752}.
\newblock \DOIprefix\doi{10.1016/j.jcp.2014.04.003}.
%
\bibitem[{Heath et~al.(2012)Heath, Gamba, Morrison and
  Michler}]{HeathGambaMorrisonMichler2012}
\bibinfo{author}{Heath, R.E.}, \bibinfo{author}{Gamba, I.M.},
  \bibinfo{author}{Morrison, P.J.}, \bibinfo{author}{Michler, C.},
  \bibinfo{year}{2012}.
\newblock \bibinfo{title}{A discontinuous {G}alerkin method for the
  {V}lasov--{P}oisson system}.
\newblock \bibinfo{journal}{Journal of Computational Physics}
  \bibinfo{volume}{231}, \bibinfo{pages}{1140--1174}.
\newblock \DOIprefix\doi{10.1016/j.jcp.2011.09.020}.
%
\bibitem[{Huot et~al.(2003)Huot, Ghizzo, Bertrand, Sonnendr{\"u}cker and
  Coulaud}]{HuotGhizzoBertrandSonnendruckerCoulaud2003}
\bibinfo{author}{Huot, F.}, \bibinfo{author}{Ghizzo, A.},
  \bibinfo{author}{Bertrand, P.}, \bibinfo{author}{Sonnendr{\"u}cker, E.},
  \bibinfo{author}{Coulaud, O.}, \bibinfo{year}{2003}.
\newblock \bibinfo{title}{Instability of the time splitting scheme for the
  one-dimensional and relativistic {V}lasov--{M}axwell system}.
\newblock \bibinfo{journal}{Journal of Computational Physics}
  \bibinfo{volume}{185}, \bibinfo{pages}{512--531}.
\newblock \DOIprefix\doi{10.1016/S0021-9991(02)00079-7}.
%
\bibitem[{Johnson et~al.(2023)Johnson, Rossmanith and
  Vaughan}]{article:JohnsonRossmanithVaughan2023}
\bibinfo{author}{Johnson, E.R.}, \bibinfo{author}{Rossmanith, J.A.},
  \bibinfo{author}{Vaughan, C.}, \bibinfo{year}{2023}.
\newblock \bibinfo{title}{Positivity-preserving {L}ax--{W}endroff discontinuous
  {G}alerkin schemes for quadrature-based moment-closure approximations of
  kinetic models}.
\newblock \bibinfo{journal}{Journal of Scientific Computing}
  \bibinfo{volume}{95}, \bibinfo{pages}{19}.
\newblock \DOIprefix\doi{10.1007/s10915-023-02117-5}.
%
\bibitem[{Landau(1946)}]{article:Landau1946}
\bibinfo{author}{Landau, L.}, \bibinfo{year}{1946}.
\newblock \bibinfo{title}{On the vibrations of the electronic plasma}.
\newblock \bibinfo{journal}{J.Phys.(USSR)} \bibinfo{volume}{10},
  \bibinfo{pages}{25--34}.
\newblock \DOIprefix\doi{10.3367/UFNr.0093.196711m.0527}.
%
\bibitem[{Le~Bourdiec et~al.(2006)Le~Bourdiec, de~Vuyst and
  Jacquet}]{LeBourdiecDeVuystJacquet2006}
\bibinfo{author}{Le~Bourdiec, S.}, \bibinfo{author}{de~Vuyst, F.},
  \bibinfo{author}{Jacquet, L.}, \bibinfo{year}{2006}.
\newblock \bibinfo{title}{Numerical solution of the {V}lasov--{P}oisson system
  using generalized {H}ermite functions}.
\newblock \bibinfo{journal}{Computer Physics Communications}
  \bibinfo{volume}{175}, \bibinfo{pages}{528--544}.
\newblock \DOIprefix\doi{10.1016/j.cpc.2006.07.004}.
%
\bibitem[{Marsden and Weinstein(1982)}]{article:Marsden1982}
\bibinfo{author}{Marsden, J.E.}, \bibinfo{author}{Weinstein, A.},
  \bibinfo{year}{1982}.
\newblock \bibinfo{title}{The {H}amiltonian structure of the {M}axwell-{V}lasov
  equations}.
\newblock \bibinfo{journal}{Physica D: Nonlinear Phenomena}
  \bibinfo{volume}{4}, \bibinfo{pages}{394--406}.
\newblock \DOIprefix\doi{10.1016/0167-2789(82)90043-4}.
%
\bibitem[{Morrison(1980)}]{article:Morrison1980}
\bibinfo{author}{Morrison, P.J.}, \bibinfo{year}{1980}.
\newblock \bibinfo{title}{The {M}axwell-{V}lasov equations as a continuous
  hamiltonian system}.
\newblock \bibinfo{journal}{Physics Letters A} \bibinfo{volume}{80},
  \bibinfo{pages}{383--386}.
\newblock \DOIprefix\doi{10.1016/0375-9601(80)90776-8}.
%
\bibitem[{Mouhot and Villani(2011)}]{article:MouhotVillani2011}
\bibinfo{author}{Mouhot, C.}, \bibinfo{author}{Villani, C.},
  \bibinfo{year}{2011}.
\newblock \bibinfo{title}{On {L}andau damping}.
\newblock \bibinfo{journal}{Acta Mathematica} \bibinfo{volume}{207}.
\newblock \DOIprefix\doi{10.1007/s11511-011-0068-9}.
%
\bibitem[{Palmroth et~al.(2025)Palmroth, Ganse, Pfau-Kempf, Battarbee, Alho,
  N{\"a}ttil{\"a}, Zaitsev, Cozzani, Papadakis, Kotipalo, Zhou, Turc,
  Hoilijoki, Grandin, P{\"a}nk{\"a}l{\"a}inen, Sandroos and {von
  Alfthan}}]{PalmrothGansePfauKempfEtAl2025}
\bibinfo{author}{Palmroth, M.}, \bibinfo{author}{Ganse, U.},
  \bibinfo{author}{Pfau-Kempf, Y.}, \bibinfo{author}{Battarbee, M.},
  \bibinfo{author}{Alho, M.}, \bibinfo{author}{N{\"a}ttil{\"a}, J.},
  \bibinfo{author}{Zaitsev, I.}, \bibinfo{author}{Cozzani, G.},
  \bibinfo{author}{Papadakis, K.}, \bibinfo{author}{Kotipalo, L.},
  \bibinfo{author}{Zhou, H.}, \bibinfo{author}{Turc, L.},
  \bibinfo{author}{Hoilijoki, S.}, \bibinfo{author}{Grandin, M.},
  \bibinfo{author}{P{\"a}nk{\"a}l{\"a}inen, L.}, \bibinfo{author}{Sandroos,
  A.}, \bibinfo{author}{{von Alfthan}, S.}, \bibinfo{year}{2025}.
\newblock \bibinfo{title}{Vlasov methods in space physics and astrophysics}.
\newblock \bibinfo{journal}{Living Reviews in Computational Astrophysics}
  \bibinfo{volume}{11}, \bibinfo{pages}{3}.
\newblock \DOIprefix\doi{10.1007/s41115-025-00024-0}. \bibinfo{note}{article
  number 3; published 14 Nov 2025.}
%
\bibitem[{Parker et~al.(1993)Parker, Friedman, Ray and
  Birdsall}]{article:Parker1993}
\bibinfo{author}{Parker, S.E.}, \bibinfo{author}{Friedman, A.},
  \bibinfo{author}{Ray, S.L.}, \bibinfo{author}{Birdsall, C.K.},
  \bibinfo{year}{1993}.
\newblock \bibinfo{title}{Bounded multi-scale plasma simulation: {A}pplication
  to sheath problems}.
\newblock \bibinfo{journal}{Journal of Computational Physics}
  \bibinfo{volume}{107}, \bibinfo{pages}{388--402}.
\newblock \DOIprefix\doi{10.1006/jcph.1993.1153}.
%
\bibitem[{Qiu and Christlieb(2010)}]{qiu2010}
\bibinfo{author}{Qiu, J.M.}, \bibinfo{author}{Christlieb, A.},
  \bibinfo{year}{2010}.
\newblock \bibinfo{title}{A conservative high order semi-lagrangian {WENO}
  method for the {V}lasov equation}.
\newblock \bibinfo{journal}{Journal of Computational Physics}
  \bibinfo{volume}{229}, \bibinfo{pages}{1130--1149}.
\newblock \DOIprefix\doi{10.1016/j.jcp.2009.10.016}.
%
\bibitem[{Qiu and Shu(2011)}]{QiuShu2011}
\bibinfo{author}{Qiu, J.M.}, \bibinfo{author}{Shu, C.W.}, \bibinfo{year}{2011}.
\newblock \bibinfo{title}{Positivity preserving semi-{L}agrangian discontinuous
  {G}alerkin formulation: {T}heoretical analysis and application to the
  {V}lasov--{P}oisson system}.
\newblock \bibinfo{journal}{Journal of Computational Physics}
  \bibinfo{volume}{230}, \bibinfo{pages}{8386--8409}.
\newblock \DOIprefix\doi{10.1016/j.jcp.2011.07.018}.
%
\bibitem[{Ren and Lapenta(2024)}]{article:RenLapenta2024}
\bibinfo{author}{Ren, J.}, \bibinfo{author}{Lapenta, G.}, \bibinfo{year}{2024}.
\newblock \bibinfo{title}{Recent development of fully kinetic particle-in-cell
  method and its application to fusion plasma instability study}.
\newblock \bibinfo{journal}{Frontiers in Physics} \bibinfo{volume}{12}.
\newblock \DOIprefix\doi{10.3389/fphy.2024.1340736}.
%
\bibitem[{Rossmanith(2026)}]{dogpack}
\bibinfo{author}{Rossmanith, J.A.}, \bibinfo{year}{2026}.
\newblock \bibinfo{title}{{\sc DoGPack}}.
\newblock \bibinfo{howpublished}{\url{http://www.dogpack-code.org/}}.
%
\bibitem[{Rossmanith and Seal(2011)}]{RossmanithSeal2011}
\bibinfo{author}{Rossmanith, J.A.}, \bibinfo{author}{Seal, D.C.},
  \bibinfo{year}{2011}.
\newblock \bibinfo{title}{A positivity-preserving high-order semi-{L}agrangian
  discontinuous {G}alerkin scheme for the {V}lasov--{P}oisson equations}.
\newblock \bibinfo{journal}{Journal of Computational Physics}
  \bibinfo{volume}{230}, \bibinfo{pages}{6203--6232}.
\newblock \DOIprefix\doi{10.1016/j.jcp.2011.04.018}.
%
\bibitem[{Rossmanith and Vaughan(2026)}]{article:RossmanithVaughan2026}
\bibinfo{author}{Rossmanith, J.A.}, \bibinfo{author}{Vaughan, C.},
  \bibinfo{year}{2026}.
\newblock \bibinfo{title}{A novel splitting method for {V}lasov-{A}mp{\`e}re},
  in: \bibinfo{editor}{Grabe, M.}, \bibinfo{editor}{Oblapenko, G.},
  \bibinfo{editor}{Torrilhon, M.} (Eds.), \bibinfo{booktitle}{Rarefied Gas
  Dynamics}, \bibinfo{publisher}{Springer Nature Switzerland}. pp.
  \bibinfo{pages}{451--459}.
\newblock \DOIprefix\doi{10.1007/978-3-032-00094-1_43}.
%
\bibitem[{Sagdeev et~al.(1988)Sagdeev, Usikov and
  Zaslavsky}]{SagdeevUsikovZaslavsky1988}
\bibinfo{author}{Sagdeev, R.Z.}, \bibinfo{author}{Usikov, D.A.},
  \bibinfo{author}{Zaslavsky, G.M.}, \bibinfo{year}{1988}.
\newblock \bibinfo{title}{Nonlinear Physics: From the Pendulum to Turbulence
  and Chaos}.
\newblock \bibinfo{publisher}{Harwood Academic Publishers},
  \bibinfo{address}{New York}.
%
\bibitem[{Seal(2012)}]{thesis:Seal2012}
\bibinfo{author}{Seal, D.}, \bibinfo{year}{2012}.
\newblock \bibinfo{title}{Discontinous {G}alerkin methods for {V}lasov models
  of plasma}.
\newblock Ph.D. thesis. University of Wisconsin.
\newblock
  \bibinfo{note}{\url{https://www.proquest.com/pqdtglobal/docview/1022644798}}.
%
\bibitem[{Shay et~al.(2007)Shay, Drake and Dorland}]{article:Shay2007}
\bibinfo{author}{Shay, M.A.}, \bibinfo{author}{Drake, J.F.},
  \bibinfo{author}{Dorland, B.}, \bibinfo{year}{2007}.
\newblock \bibinfo{title}{Equation {F}ree {P}rojective {I}ntegration: {A}
  multiscale method applied to a plasma ion acoustic wave}.
\newblock \bibinfo{journal}{Journal of Computational Physics}
  \bibinfo{volume}{226}, \bibinfo{pages}{571--585}.
\newblock \DOIprefix\doi{10.1016/j.jcp.2007.04.016}.
%
\bibitem[{Stix(1992)}]{Stix1992}
\bibinfo{author}{Stix, T.H.}, \bibinfo{year}{1992}.
\newblock \bibinfo{title}{Waves in Plasmas}.
\newblock \bibinfo{publisher}{American Institute of Physics},
  \bibinfo{address}{New York}.
%
\bibitem[{Strang(1968)}]{Strang1968}
\bibinfo{author}{Strang, G.}, \bibinfo{year}{1968}.
\newblock \bibinfo{title}{On the construction and comparison of difference
  schemes}.
\newblock \bibinfo{journal}{SIAM Journal on Numerical Analysis}
  \bibinfo{volume}{5}, \bibinfo{pages}{506--517}.
\newblock \DOIprefix\doi{10.1137/0705041}.
%
\bibitem[{Taitano et~al.(2026)Taitano, Burby and
  Alekseenko}]{TaitanoBurbyAlekseenko2026}
\bibinfo{author}{Taitano, W.}, \bibinfo{author}{Burby, J.},
  \bibinfo{author}{Alekseenko, A.}, \bibinfo{year}{2026}.
\newblock \bibinfo{title}{A novel conditional formulation of the
  {Vlasov--Amp{\`e}re} equations: a conservative, positivity, asymptotic and
  {Gauss} law preserving scheme}.
\newblock \bibinfo{journal}{Journal of Plasma Physics} \bibinfo{volume}{92},
  \bibinfo{pages}{E45}.
\newblock \DOIprefix\doi{10.1017/S0022377826101408}.
%
\bibitem[{Vlasov(1968)}]{article:Vlasov1968}
\bibinfo{author}{Vlasov, A.}, \bibinfo{year}{1968}.
\newblock \bibinfo{title}{The vibrational properties of an electron gas}.
\newblock \bibinfo{journal}{Soviet Physics Uspekhi} \bibinfo{volume}{10},
  \bibinfo{pages}{721--733}.
\newblock \DOIprefix\doi{10.1070/PU1968v010n06ABEH003709}.
%
\bibitem[{Whealton et~al.(1986)Whealton, McGaffey and
  Meszaros}]{WhealtonMcGaffeyMeszaros1986}
\bibinfo{author}{Whealton, J.H.}, \bibinfo{author}{McGaffey, R.W.},
  \bibinfo{author}{Meszaros, P.S.}, \bibinfo{year}{1986}.
\newblock \bibinfo{title}{A finite difference {3-D} {P}oisson--{V}lasov
  algorithm for ions extracted from a plasma}.
\newblock \bibinfo{journal}{Journal of Computational Physics}
  \bibinfo{volume}{63}, \bibinfo{pages}{20--32}.
\newblock \DOIprefix\doi{10.1016/0021-9991(86)90082-3}.
%
\bibitem[{Wilhelm et~al.(2025a)Wilhelm, Bacchini, Sch{\"o}ps, Torrilhon, Merkel
  and Kirchhart}]{NufiEtAlMultispeciesVM}
\bibinfo{author}{Wilhelm, R.P.}, \bibinfo{author}{Bacchini, F.},
  \bibinfo{author}{Sch{\"o}ps, S.}, \bibinfo{author}{Torrilhon, M.},
  \bibinfo{author}{Merkel, M.}, \bibinfo{author}{Kirchhart, M.},
  \bibinfo{year}{2025}a.
\newblock \bibinfo{title}{Extending the {N}umerical {F}low {I}teration to the
  multi-species {V}lasov--{M}axwell system through {H}amiltonian splitting}.
\newblock \bibinfo{journal}{arXiv} \DOIprefix\doi{10.48550/arXiv.2511.11322}.
%
\bibitem[{Wilhelm et~al.(2026)Wilhelm, Krah, Schneider, Bacchini and
  Grandgirard}]{WilhelmEtAl2026}
\bibinfo{author}{Wilhelm, R.P.}, \bibinfo{author}{Krah, P.},
  \bibinfo{author}{Schneider, K.}, \bibinfo{author}{Bacchini, F.},
  \bibinfo{author}{Grandgirard, V.}, \bibinfo{year}{2026}.
\newblock \bibinfo{title}{Revisiting kinetic electrostatic electron non-linear
  ({KEEN}) waves in the presence of dynamical ions}.
\newblock \DOIprefix\doi{10.48550/arXiv.2609.04941},
  \href{http://arxiv.org/abs/2609.04941}{\tt arXiv:2609.04941}.
  \bibinfo{note}{arXiv preprint arXiv:2609.04941}.
%
\bibitem[{Wilhelm and
  Torrilhon(2025)}]{WilhelmTorrilhon2025NFIInstabilitiesRGD}
\bibinfo{author}{Wilhelm, R.P.}, \bibinfo{author}{Torrilhon, M.},
  \bibinfo{year}{2025}.
\newblock \bibinfo{title}{Simulation of multi-species kinetic instabilities
  with the {N}umerical {F}low {I}teration}.
\newblock \bibinfo{journal}{arXiv} \DOIprefix\doi{10.48550/arXiv.2511.02405}.
%
\bibitem[{Wilhelm
  et~al.(2025b)}]{WilhelmEtAl2025HighFidelityMultispeciesVlasov}
\bibinfo{author}{Wilhelm, R.P.}, et~al., \bibinfo{year}{2025}b.
\newblock \bibinfo{title}{High fidelity simulations of the multi-species
  {V}lasov equation in the electro-static, collisional-less limit}.
\newblock \bibinfo{journal}{Plasma Physics and Controlled Fusion}
  \bibinfo{volume}{67}, \bibinfo{pages}{025011}.
\newblock \DOIprefix\doi{10.1088/1361-6587/ad9fdb}.
%
\bibitem[{Xiong et~al.(2008)Xiong, Cohen, Rognlien and
  Xu}]{XiongCohenRognlienXu2008}
\bibinfo{author}{Xiong, Z.}, \bibinfo{author}{Cohen, R.H.},
  \bibinfo{author}{Rognlien, T.D.}, \bibinfo{author}{Xu, X.Q.},
  \bibinfo{year}{2008}.
\newblock \bibinfo{title}{A high-order finite-volume algorithm for
  {F}okker--{P}lanck collisions in magnetized plasmas}.
\newblock \bibinfo{journal}{Journal of Computational Physics}
  \bibinfo{volume}{227}, \bibinfo{pages}{7192--7205}.
\newblock \DOIprefix\doi{10.1016/j.jcp.2008.04.004}.
%
\bibitem[{Yoshida(1990)}]{article:Yoshida1990}
\bibinfo{author}{Yoshida, H.}, \bibinfo{year}{1990}.
\newblock \bibinfo{title}{Construction of higher order symplectic integrators}.
\newblock \bibinfo{journal}{Physics Letters A} \bibinfo{volume}{150},
  \bibinfo{pages}{262--268}.
\newblock \DOIprefix\doi{10.1016/0375-9601(90)90092-3}.
%
\bibitem[{Yoshida(1993)}]{article:Yoshida1993}
\bibinfo{author}{Yoshida, H.}, \bibinfo{year}{1993}.
\newblock \bibinfo{title}{Recent progress in the theory and application of
  symplectic integrators}.
\newblock \bibinfo{journal}{Celestial Mechanics and Dynamical Astronomy}
  \bibinfo{volume}{56}, \bibinfo{pages}{27--43}.
\newblock \DOIprefix\doi{10.1007/BF00699717}.
%
\bibitem[{Zhang and Shu(2010)}]{ZhangShu2010}
\bibinfo{author}{Zhang, X.}, \bibinfo{author}{Shu, C.W.}, \bibinfo{year}{2010}.
\newblock \bibinfo{title}{On maximum-principle-satisfying high order schemes
  for scalar conservation laws}.
\newblock \bibinfo{journal}{Journal of Computational Physics}
  \bibinfo{volume}{229}, \bibinfo{pages}{3091--3120}.
\newblock \DOIprefix\doi{10.1016/j.jcp.2009.12.030}.
%
\bibitem[{Zheng et~al.(2022)Zheng, Cai, Qiu and Qiu}]{ZhengCaiQiuQiu2022}
\bibinfo{author}{Zheng, N.}, \bibinfo{author}{Cai, X.}, \bibinfo{author}{Qiu,
  J.M.}, \bibinfo{author}{Qiu, J.}, \bibinfo{year}{2022}.
\newblock \bibinfo{title}{A fourth-order conservative semi-{L}agrangian finite
  volume {WENO} scheme without operator splitting for kinetic and fluid
  simulations}.
\newblock \bibinfo{journal}{Computer Methods in Applied Mechanics and
  Engineering} \bibinfo{volume}{395}, \bibinfo{pages}{114973}.
\newblock \DOIprefix\doi{10.1016/j.cma.2022.114973}.
%
\bibitem[{Zheng et~al.(2025)Zheng, Hayes, Christlieb and
  Qiu}]{ZhengHayesChristliebQiu2025}
\bibinfo{author}{Zheng, N.}, \bibinfo{author}{Hayes, D.},
  \bibinfo{author}{Christlieb, A.}, \bibinfo{author}{Qiu, J.M.},
  \bibinfo{year}{2025}.
\newblock \bibinfo{title}{A semi-{L}agrangian adaptive-rank ({SLAR}) method for
  linear advection and nonlinear {V}lasov--{P}oisson system}.
\newblock \bibinfo{journal}{Journal of Computational Physics}
  \bibinfo{volume}{532}, \bibinfo{pages}{113970}.
\newblock \DOIprefix\doi{10.1016/j.jcp.2025.113970}.
%
\bibitem[{Zhou et~al.(2001)Zhou, Guo and Shu}]{article:ZhouGuoShu2001}
\bibinfo{author}{Zhou, T.}, \bibinfo{author}{Guo, Y.}, \bibinfo{author}{Shu,
  C.W.}, \bibinfo{year}{2001}.
\newblock \bibinfo{title}{Numerical study on {L}andau damping}.
\newblock \bibinfo{journal}{Physica D: Nonlinear Phenomena}
  \bibinfo{volume}{157}, \bibinfo{pages}{322--333}.
\newblock \DOIprefix\doi{10.1016/S0167-2789(01)00289-5}.

\end{thebibliography}

\FloatBarrier

\end{document}